\documentclass[12pt,reqno]{amsart}

\usepackage{hyperref}
\usepackage[usenames]{color}
\usepackage[latin1]{inputenc} 
\usepackage{float}
\usepackage{bbm}

\usepackage[displaymath, mathlines]{lineno}

\usepackage[english]{babel} 

\usepackage{amssymb, amsmath}
\usepackage{geometry,graphicx}
\usepackage{cite}
\usepackage{enumitem}

\usepackage{xurl}

\allowdisplaybreaks

\theoremstyle{plain}
\newtheorem{Theorem}{Theorem}[section] 

\newtheorem{Lemma}[Theorem]{Lemma} 

\newtheorem{Proposition}[Theorem]{Proposition}
\newtheorem{Corollary}[Theorem]{Corollary}

\newtheorem{Remark}[Theorem]{Remark}

\newtheorem{Example}[Theorem]{Example}

\numberwithin{equation}{section} 
\begin{document}

\begin{center}
 \textbf{Theory on linear L-fractional differential equations and a new Mittag--Leffler-type function}
\end{center}

\begin{center}
 Marc Jornet
\end{center}

\begin{center}
Departament de Matem\`atiques, Universitat de Val\`encia, 46100 Burjassot, Spain. \\
email: marc.jornet@uv.es \\
ORCID: 0000-0003-0748-3730
\end{center}

\noindent
\textit{Reference:} Fractal Fract. 2024, 8(7), 411; https://doi.org/10.3390/fractalfract8070411

\ \\
\textbf{Abstract.} The L-fractional derivative is defined as a certain normalization of the well-known Caputo derivative, so  alternative properties hold: smoothness and finite slope at the origin for the solution, velocity units for the vector field, and a differential form associated to the system. We develop a theory of this fractional derivative as follows. We prove a fundamental theorem of calculus. We deal with linear systems of autonomous homogeneous parts, which correspond to Caputo linear equations of non-autonomous homogeneous parts. The associated L-fractional integral operator, which is closely related to the beta function and the beta probability distribution, and the estimates for its norm in the Banach space of continuous functions play a key role in the development. The explicit solution is built by means of Picard's iterations from a Mittag--Leffler-type function that mimics the standard exponential function. In the second part of the paper, we address autonomous linear equations of sequential type. We start with sequential order two and then move to arbitrary order by dealing with a power series. The classical theory of linear ordinary differential equations with constant coefficients is generalized, and we establish an analog of the method of undetermined coefficients. The last part of the paper is concerned with sequential linear equations of analytic coefficients and order two. \\
\\
\textit{Keywords: non-integer-order differential equation; Leibniz and Caputo fractional operators; linear and sequential linear equations; Mittag--Leffler function; power series; Picard's iterations} \\
\\
\textit{AMS Classification 2020: 34A08; 34K06; 33E12; 34A25; 60E05}

\section{Introduction} \label{sec_intro}

\subsection{Literature Review}

Fractional calculus is concerned with non-integer differentiation, where the new derivative operator is often presented as an integral expression with respect to a kernel function. The operator depends on the fractional order or index, which may be a real number in $(0,1)$, a real number with no bounds, or even a complex value, and the ordinary derivative is retrieved for order $1$. Good expositions of the topic are given in the monographs~\cite{podl_llibre,kilbas,diethelm_llibre,abbas,yong,ascione}. There are many notions of fractional derivatives, and different approaches and rules have been followed to study these operators and associated differential equations~\cite{tenreiro,ortigueira,teodoro,weber_pap,dieth,dieth2,vanm}. Among all of the definitions, in this paper, we will consider the important Caputo fractional operator, with the consequent Caputo fractional differential equations. This operator was proposed nearly sixty years ago in~\cite{caputo} in the context of viscoelasticity theory. However, it is still of use in current mathematical and applied research; see for example the recent publications~\cite{nou0,nou1,nou2,nou3,nou4,nou5,nou6}. The operator is defined as a convolution with respect to a singular kernel so that a continuous delay is incorporated into the differential equation. This definition brings about a new kind of functional differential equations, which exhibit memory and hereditary effects that may capture different dynamics more flexibly. Compared to the Riemann--Liouville formulation, the ordinary derivative is placed within the integral so that initial conditions are posed as in the classical sense. Due to the applicability of fractional calculus and the similarities with ordinary calculus, some definitions and computations in the literature lack sufficient rigor, as pointed out in~\cite{weber_pap}; thus, we aim at giving precise results, in line with~\cite{diethelm_llibre,weber_pap}, for instance.

Throughout this article, we will be interested in explicit and closed-form solutions to fractional differential equations. In fact, we will build a theory on a new class of fractional differential equations and their corresponding solutions, but details will be given later. {By explicit solution, we mean a state or response function that can be solved and isolated, whereas a closed-form solution refers to a more detailed final expression in terms of the input data.} For Caputo fractional differential equations, there are many works that construct explicit solutions, often in the realm of applicable models. In the homogeneous and autonomous linear case, the solution depends on the most important function in fractional calculus, the (one-parameter) Mittag--Leffler function, which is defined as a power-series expansion that extends the Taylor series of the exponential function~\cite{mitag1,mitag2,mitag3,mitagVan,mitag4}. The theory of fractional Taylor series was first introduced in~\cite{taylor_generi}, where some examples of homogeneous linear equations were shown. For non-homogeneous linear models, the two-parameter Mittag--Leffler function appears in the solution's expression too, within a convolution; this result can be deduced by means of Picard's iterations~\cite{mitag110}. When moving to nonlinear equations, fractional power series may be employed as well, albeit the recursive relation for the expansion's coefficients is not solvable in closed form. Some examples, which were published quite recently, are the logistic equation~\cite{logistic_nieto}, the Bernoulli equation~\cite{ovidio2}, SIS equations~\cite{ovidio}, and general compartmental models with polynomial nonlinearity~\cite{compart_meu}. In fact, the Cauchy--Kovalevskaya theorem has just been proved for systems of Caputo fractional differential equations with analytic inputs~\cite{ck_meu}, hence giving a theoretical justification of the method in general. Since fractional calculus differs from standard calculus (product rule, chain rule, etc.~\cite{arxiv_cr}), the contribution~\cite{ck_meu} circumvents the problems and employs the method of majorants and the implicit-function theorem to achieve a proof of the Cauchy--Kovalevskaya theorem. For variations and generalizations of the Caputo operator, which expand the possible kernel functions, power series also play an important role as well, for instance, for Prabhakar fractional logistic equations~\cite{prab_area} and Caputo generalized proportional fractional logistic equations~\cite{nieto_gener}. Further explicit expressions compared to power-series expansions are not usually available; see the discussions in~\cite{west1,west2,west3}. There are alternative analytical techniques, such as the Laplace-transform method~\cite{nou3,kexue,lapl2,lapl3} (which is often applied with formal calculations), which has even been used in the stochastic sense together with other probabilistic tools~\cite{meu_batim}. As most Caputo models do not possess explicit solutions, numerical schemes have been implemented to compute approximations on mesh discretizations~\cite{garrap1,garrap2,garrap3}. Building numerical solvers for fractional models is much more difficult than in the standard integer-order case due to persistent memory terms. Here, we will not use Laplace transforms or numerical resolutions; we will focus on power-series-related methods instead, with rigorous proofs of convergence.

{Motivated by issues with the Caputo fractional derivative, in this paper, we investigate a variant that has been applied in mechanics, already called the L-fractional derivative by other authors, with associated L-fractional differential equations~\cite{lazo_linear,lazo_linear2}. It has also been introduced in the logistic equation for growth processes~\cite{nito_aml}. The definition is based on normalizing the Caputo operator, so that the fractional derivative of the identity function is $1$. With such an approach, as will be seen, the class of fractional differentiable functions is enlarged from absolute continuity to classical analyticity so that the calculus is less restrictive. It is true that the normalization of fractional derivatives has a straightforward definition, but it gives rise to distinct and interesting geometrical, physical, and qualitative features. Thus, it should be further investigated in theory and in modeling. See~\cite{nito_aml} and the recent arXiv preprint~\cite{rxi_jo}, for example. In contrast to the Caputo derivative, the ordinary derivative of an L-fractional solution is always finite at the initial instant, which likely makes more sense when modeling real dynamics. The L-fractional derivative can be interpreted in terms of differentials~\cite{lazo1,lazo2,cotrill,ada,taras_dif}, with usual units of time$^{-1}$ in the vector field in the model. Thus, the disadvantages of the Caputo derivative are overcome. Although the normalization is directly related to the original Caputo fractional derivative and numerical solvers available for Caputo fractional differential equations are readily extended to the L-fractional situation, the new L-fractional differential equations exhibit many properties, and the search for solutions thus deserves specific attention. We develop a complete theory on linear L-fractional differential equations, with ideas that might be adapted to other fractional operators. Interestingly, the theory provides a new insight into the classical exposition of linear ordinary differential equations, and it gives rise to the definition of a new Mittag--Leffler-type function with a certain power series. As in other treatments for the Caputo derivative~\cite{vatsala,ugurlu}, we deal with sequential-type models by composing the L-fractional derivative.}

A related fractional derivative that could be investigated in the future is the $\Lambda$-fractional derivative, which normalizes the Riemann--Lioville operator instead~\cite{lambda1,lambda2}. 

{In the article, we fix the fractional order $\alpha\in (0,1)$. The case $\alpha=1$ is possible as well, and it corresponds to the classical integer-order setting.}

\subsection{Previous Context} \label{subs_conte}

We base this on the references previously cited. In this paper, all integrals will be understood in the sense of Lebesgue, which may be interpreted as improper Riemann integrals or Riemann integrals under appropriate conditions, for example, the continuity of the integrand. Let $\mathrm{L}^1[0,T]$ be the Lebesgue space of integrable functions on the interval $[0,T]$, $T>0$. If the function $x:[0,T]\rightarrow\mathbb{C}^d$ belongs to $\mathrm{L}^1[0,T]$, then its Riemann--Liouville fractional integral is defined as~\cite{diethelm_llibre,weber_pap}
\begin{equation}
 {}^{RL}\! J^\alpha x(t)=\frac{1}{\Gamma(\alpha)}\int_0^t (t-\tau)^{\alpha-1}x(\tau)\mathrm{d}\tau=\frac{1}{\Gamma(\alpha)}(t^{\alpha-1}\ast x)(t),
 \label{rl_intg}
\end{equation}
where $\alpha\in (0,1)$ $\ast$ is the convolution and
\[ \Gamma(z)=\int_0^\infty \tau^{z-1}\mathrm{e}^{-\tau}\mathrm{d}\tau \]
is the gamma function. The gamma function generalizes the factorial: $\Gamma(n+1)=n!$, for integers $n\geq0$. As $x\in \mathrm{L}^1[0,T]$ and $t^{\alpha-1}\in \mathrm{L}^1[0,T]$, a standard result tells us that the convolution in~\eqref{rl_intg} is defined as an $\mathrm{L}^1[0,T]$ function; in particular, it is pointwise defined almost everywhere on $[0,T]$ (i.e., everywhere except a set of Lebesgue measure zero). Of course, there are functions for which the Riemann--Liouville integral exists for every $t\in [0,T]$. Some texts define~\eqref{rl_intg} whenever the integral exists, but that is certainly imprecise.

We say that $x:[0,T]\rightarrow\mathbb{C}^d$ is absolutely continuous if its derivative $x'$ exists almost everywhere, $x'\in\mathrm{L}^1[0,T]$, and 
\begin{equation} x(t)=x(0)+\int_0^t x'(s)\mathrm{d}s \label{cantorrr} \end{equation}
for all $t\in [0,T]$, i.e., Barrow's rule holds in the Lebesgue sense. These conditions are weaker than the continuous differentiability demanded by Riemann integration. Essentially, we are saying that $x$ belongs to the Sobolev space $W^{1,1}[0,T]$ with values on $\mathbb{C}^d$. The identity~\eqref{cantorrr} is necessary, as the Cantor's function shows. According to \cite{weber_pap} (Proposition~3.2), the operator ${}^{RL}\! J^\alpha$ from~\eqref{rl_intg} maps $W^{1,1}[0,T]$ into $W^{1,1}[0,T]$ (it does not map infinitely differentiable functions $\mathcal{C}^\infty[0,T]$ into continuously differentiable functions $\mathcal{C}^1[0,T]$ in general). For absolutely continuous functions on $[0,T]$, the Riemann--Liouville fractional derivative is defined as~\cite{diethelm_llibre,weber_pap}
\begin{equation}
 {}^{RL}\! D^\alpha x(t)=\frac{\mathrm{d}}{\mathrm{d}t} \,{}^{RL}\! J^{1-\alpha} x(t)=\frac{1}{\Gamma(1-\alpha)}\frac{\mathrm{d}}{\mathrm{d}t}\int_0^t \frac{x(\tau)}{(t-\tau)^\alpha}\mathrm{d}\tau,
 \label{rl_deirv}
\end{equation}
where $\alpha\in (0,1)$ is the fractional order of differentiation. Note that ${}^{RL}\! J^{1-\alpha} x$ is absolutely continuous on $[0,T]$; therefore, it makes sense to differentiate ${}^{RL}\! J^{1-\alpha} x$ almost everywhere on $[0,T]$.

The Caputo fractional derivative is defined as~\cite{diethelm_llibre,weber_pap}
\begin{equation} {}^C\! D^\alpha x(t)={}^{RL}\! J^{1-\alpha} x'(t)=\frac{1}{\Gamma(1-\alpha)}\int_0^t \frac{x'(\tau)}{(t-\tau)^\alpha}\mathrm{d}\tau,
\label{derC} \end{equation}
where $\alpha\in (0,1)$ is the fractional order of differentiation and $t\in [0,T]$. Compared to~\eqref{rl_deirv}, the ordinary derivative is placed within the integral. The operator~\eqref{derC} is a convolution with continuous delay with respect to a singular kernel 
\[ \mathcal{K}(t-\tau)=(t-\tau)^{-\alpha}. \]

Since $x'\in\mathrm{L}^1[0,T]$, the Caputo derivative ${}^{RL}\! J^{1-\alpha} x'(t)$ exists almost everywhere on $[0,T]$, and it belongs to $\mathrm{L}^1[0,T]$. The boundary values of the operator are 
\begin{equation} {}^C\! D^{0^+} x(t)=x(t)-x(0), \label{bailando_AA} \end{equation}
for every $t\in [0,T]$, and if $f$ is continuously differentiable on $[0,T]$ \cite{diethelm_llibre} (page~37),
\begin{equation} {}^C\! D^{1^-} x(t)=x'(t), \label{bailando_BB} \end{equation}
for all $t\in [0,T]$. Then, it interpolates between the discrete difference $x(t)-x(0)=\int_0^t x'(\tau)\mathrm{d}\tau$, which is related to the mean value, and the ordinary derivative $x'(t)$. 

Useful examples of computation for~\eqref{derC} are
\begin{equation} {}^C\! D^\alpha t^\beta=\frac{\Gamma(\beta+1)}{\Gamma(\beta-\alpha+1)}t^{\beta-\alpha} \label{cond1}, \end{equation}
for powers $\beta>0$~\cite{samko}. In particular,
\begin{equation} {}^C\! D^\alpha t=\frac{1}{\Gamma(2-\alpha)}t^{1-\alpha} \label{particularrr} \end{equation}
and
\begin{equation} {}^C\! D^\alpha 1=0. \label{der_ct} \end{equation}

Therefore, while ${}^C\! D^\alpha c=0$ holds for constants $c\in\mathbb{C}$, it is not true that ${}^C\! D^\alpha t=1$.

Motivated by the definition of ordinary differential equations, a Caputo fractional differential equation is an equation of the form
\begin{equation}
 {}^C\! D^\alpha x(t)=f(t,x(t)),
 \label{ode_cap}
\end{equation}
with an initial condition or state $x(0)=x_0$, where $f:[0,T]\times \Omega\subseteq [0,T]\times\mathbb{R}^d\rightarrow\mathbb{R}^d$, or $f:[0,T]\times \Omega\subseteq [0,T]\times\mathbb{C}^d\rightarrow\mathbb{C}^d$, is a continuous function such that $x_0\in\Omega$. Problem~\eqref{ode_cap} can be interpreted in an almost-everywhere sense, considering that ${}^C\! D^\alpha x\in\mathrm{L}^1[0,T]$. As usual, the equation is said to be autonomous if $f(t,x)$ does not depend on $t$ explicitly, i.e., $f(t,x)\equiv f(x)$, so that the involved input parameters are constant. Equation~\eqref{ode_cap} exhibits non-local behavior due to the delay involved in ${}^C\! D^\alpha x(t)$. The units in~\eqref{ode_cap} are time$^{-\alpha}$.

In general, the solution of~\eqref{ode_cap} cannot be twice continuously differentiable on $[0,T]$. Indeed, if it were, then we could apply integration by parts on~\eqref{derC} so that the kernel would become non-singular:
\vspace{6pt}
\begin{equation}
\begin{split}
 {}^C\! D^\alpha x(t)= {} & \frac{1}{\Gamma(1-\alpha)}\left( \frac{t^{1-\alpha}}{1-\alpha}x'(0)+\frac{1}{1-\alpha}\int_0^t (t-\tau)^{1-\alpha}x''(\tau)\mathrm{d}\tau\right) \\
= {} & \frac{1}{\Gamma(2-\alpha)}\left(t^{1-\alpha}x'(0)+\int_0^t (t-\tau)^{1-\alpha}x''(\tau)\mathrm{d}\tau\right).
\end{split}
\label{caputo_parttts}
\end{equation}

Then, at $t=0$, $f(0,x_0)={}^C\! D^\alpha x(0)=0$, which is not often the case. In practice, one has $|x'(0)|=\infty$. This comment highlights the need to consider absolutely continuous functions in the setting of Caputo fractional differential equations.

As  occurs with classical differential equation problems, Caputo Equation~\eqref{ode_cap} does not usually have explicit solutions, and numerical methods must be used. When possible, analytical or semi-analytical techniques that have been employed to derive solutions are Laplace transform and power series. For example, the simplest linear model
\begin{equation}
 {}^C\! D^\alpha x(t)=\lambda x(t),
 \label{ode_cap_lineal}
\end{equation}
where $\lambda\in\mathbb{C}$ and $x(0)=x_0$, can be solved with those techniques. 

The fractional power-series solution (i.e., a power series evaluated at $t^\alpha$)
\begin{equation} x(t)=\sum_{n=0}^\infty x_n (t^{\alpha})^n=\sum_{n=0}^\infty x_n t^{\alpha n}, \label{pw_capi} \end{equation}
where $x_n\in\mathbb{C}$ and $t\geq0$, formally satisfies
\begin{equation} \lambda \sum_{n=0}^\infty x_n t^{\alpha n} = \sum_{n=0}^\infty x_n\cdot {}^C\! D^\alpha t^{\alpha n}= \sum_{n=0}^\infty x_{n+1} \frac{\Gamma((n+1)\alpha+1)}{\Gamma(n\alpha+1)} t^{\alpha n} \label{ja_no_es_faormal} \end{equation}
in~\eqref{ode_cap_lineal}, by~\eqref{cond1}. (We use a centered dot for the notation of the product when there may be confusion with superscripts.) After matching terms,
\begin{equation} x_{n+1}=\frac{\Gamma(n\alpha+1)}{\Gamma((n+1)\alpha+1)}\lambda x_n \label{recur_capi} \end{equation}
is the first-order difference equation for the coefficients. Notice that the property~\eqref{der_ct} is key in the development. The closed-form solution to~\eqref{recur_capi} is
\[ x_n=\frac{\lambda^n}{\Gamma(n\alpha+1)}x_0. \]

The solution~\eqref{pw_capi} is then expressed as
\begin{equation} x(t)=E_{\alpha}(\lambda t^\alpha)x_0, \label{forma_pw} \end{equation}
where
\begin{equation} E_{\alpha}(s)=\sum_{n=0}^\infty \frac{s^n}{\Gamma(n\alpha+1)} \label{mlf} \end{equation}
is the well-known Mittag--Leffler function~\cite{mitag1,mitag2,mitag3,mitag4}. It is an entire function on the complex plane $\mathbb{C}$ and extends the exponential function through its Taylor series.

The Laplace-transform technique can also be used to derive~\eqref{forma_pw} and~\eqref{mlf}. The Laplace transform is defined as
\[ \mathcal{L}[x](s)=\int_0^\infty x(t)\mathrm{e}^{-st}\mathrm{d}t. \]

The most important property of $\mathcal{L}$ is
\begin{equation}
 \mathcal{L}[{}^C\! \mathcal{D}^\alpha x](s)=s^\alpha \mathcal{L}[x](s)-s^{\alpha-1}x(0),
 \label{laplace}
\end{equation}
see~\cite{podl_llibre} (page~81). By applying~\eqref{laplace} into~\eqref{ode_cap_lineal},
\vspace{6pt}
\[ s^\alpha \tilde{x}(s)-s^{\alpha-1}x(0)=\lambda \tilde{x}(s), \]
where $\tilde{x}=\mathcal{L}x$ for simplicity. That is,
\[ \tilde{x}(s)=\frac{s^{\alpha-1}}{s^\alpha-\lambda}x_0. \]

It is known~\cite{podl_llibre} (chapter~4) that
\[ \mathcal{L}[E_{\alpha}(\lambda t^\alpha)](s)=\frac{s^{\alpha-1}}{s^\alpha+\lambda}. \]

Hence,~\eqref{forma_pw} is obtained again.

Problem~\eqref{ode_cap_lineal} and the solution~\eqref{forma_pw} can be extended to the matrix case $\lambda=A\in\mathbb{C}^{d\times d}$. The Mittag--Leffler function~\eqref{mlf} is defined for matrix arguments $s=A\in\mathbb{C}^{d\times d}$, with the same series.

A general result is the following: if
\begin{equation}
 {}^C\! D^\alpha x(t)=A x(t) + b(t),
 \label{ode_cap_lineal_compl}
\end{equation}
where $A\in\mathbb{C}^{d\times d}$ is a matrix and $b:[0,T]\rightarrow\mathbb{C}^d$ is a continuous vector function, then
\begin{equation} x(t)=E_{\alpha}(A t^\alpha)x_0+\int_0^t \tau^{\alpha-1}E_{\alpha,\alpha}(A \tau^\alpha)b(t-\tau)\mathrm{d}\tau=E_{\alpha}(A t^\alpha)x_0+\left(t^{\alpha-1}E_{\alpha,\alpha}(A t^\alpha)\right)\ast b(t), \label{forma_pw_compl} \end{equation}
where 
\begin{equation} E_{\alpha,\beta}(s)=\sum_{n=0}^\infty \frac{s^n}{\Gamma(n\alpha+\beta)} \label{mlf_beta} \end{equation}
is the two-parameter Mittag--Leffler function. The procedure to derive~\eqref{forma_pw_compl} relies on solving Picard's iterative scheme~\cite{mitag110}, via the associated Volterra integral operator
\begin{equation}
 {}^C\! J^\alpha x(t)= \frac{1}{\Gamma(\alpha)}\int_0^t (t-s)^{\alpha-1}x(s)\mathrm{d}s=\frac{1}{\Gamma(\alpha)}t^{\alpha-1}\ast x(t)={}^{RL}\! J^\alpha x(t),
 \label{volt_capi}
\end{equation}
which is defined for integrable or continuous functions on $[0,T]$. Although the properties
\begin{equation}
 {}^C\! J^\alpha \circ {}^C\! D^\alpha x(t)=x(t)-x(0)
 \label{probl1_c}
\end{equation}
and
\begin{equation}
 {}^C\! D^\alpha \circ {}^C\! J^\alpha x(t)=x(t),
 \label{probl2_c}
\end{equation}
where $\circ$ denotes the composition of operators, are often used in the literature without detailed explanations, they deserve an in-depth discussion~\cite{diethelm_llibre,weber_pap} (all this will be carried out in Lemma~\ref{lema_rigor_FC_Cap}, Remarks~\ref{cor_D_s_remarK} and~\ref{rmk_villi_rel}). They are analogous to the relationship between the Lebesgue integral and the standard derivative (Barrow's rule and the fundamental theorem of calculus, respectively). Only in that case,~\eqref{ode_cap} would be equivalent to the fixed-point problem
\begin{equation} x(t)=x_0+{}^C\! J^\alpha f(t,x(t)), \label{fixed_pin} \end{equation}
the details of which can be found in \cite{weber_pap} (Remark~5.2 and Addendum). For the complete linear Equation~\eqref{ode_cap_lineal_compl}, the authors of~\cite{mitag110} define the Picard's iterative scheme from~\eqref{fixed_pin} and then obtain~\eqref{forma_pw_compl} with~\eqref{mlf_beta}.

\subsection{Objectives}

A great deal of research in applied mathematics is concerned with obtaining analytical or semi-analytical solutions of models. The present contribution continues this purpose, with the use of power series for fractional models.

The homogeneous part of Equation~\eqref{ode_cap_lineal_compl}, ${}^C\! D^\alpha x(t)=A x(t)$, is autonomous, meaning that $A$ does not depend on $t$. An aim of our paper is to address a situation of the time dependency of $A$, specifically,
\begin{equation} {}^C\! D^\alpha x(t)=t^{1-\alpha} Ax(t)+b(t), \label{nostre_c} \end{equation}
where $A\in\mathbb{C}^{d\times d}$ is a matrix and $b:[0,T]\rightarrow\mathbb{C}^d$ is a continuous vector function. To the best of our knowledge, this type of model has not previously been solved in the literature in closed form. We also deal with the case in which $b(t)$ is given by certain fractional-power functions, for which specific closed forms of the solution appear. 

The key fact is that~\eqref{nostre_c} can be transformed into a complete linear equation with an autonomous homogeneous part, but with respect to the other fractional derivative, ${}^L\! D^\alpha$. The L-fractional derivative, as will be seen, has many properties that may be advantageous compared with the conventional Caputo derivative. With~\eqref{nostre_c} and this alternative derivative, a new Mittag--Leffler-type function $\mathcal{E}_\alpha$ emerges, with a similar structure to~\eqref{mlf}. This fact opens up a wide range of research possibilities. 

To deal with linear L-fractional differential equations and build their solution with Picard's iterations, the associated L-fractional integral operator, the fundamental theorem of L-fractional calculus, and the estimates for its norm have a relevant role in the development. Due to the form of the kernel function, many of the computations are related to the beta function and the beta probability distribution. Considering this fact, the form of the solution and the proposed Mittag--Leffler-type function are analyzed probabilistically.

In the second part of the paper, we address autonomous linear equations of sequential type to extend scalar homogeneous first-order linear models. We  base it entirely on power-series expressions. The classical theory of linear ordinary differential equations is fully generalized, where the alternative Mittag--Leffler function substitutes the exponential function of the algebraic basis in a suitable way. In the non-homogeneous case, some forcing terms with a special form (polynomials and ordinary derivatives of the Mittag--Leffler-type function) are allowed  to extend the well-known method of undetermined coefficients to the fractional context.

Finally, a class of sequential non-autonomous linear equations is studied of order two and analytic coefficients. The solutions are expressed by means of power series, where the coefficients satisfy recursive relations but are not given in closed form in general. Two important models are illustrated  in the fractional sense: Airy's and Hermite's equation.

The techniques used in the article are essentially based on power series, integral equations and operators, norm estimates, Picard's iterations, probability distributions, and the algebra of vector spaces and operators, in the setting of fractional calculus.

Some  equations related to~\eqref{nostre_c} have been investigated in the literature. For example, papers~\cite{arran_nonh,arran_nonh2} study linear fractional differential equations with variable coefficients, of the Riemann--Liouville and Caputo type. The solutions are given by a convergent infinite series involving compositions of fractional integrals. Our methodology and results are distinct and more specific to L-fractional differential equations. In~\cite{vatsala}, the authors examine the problem ${}^C\! D^\alpha x(t)=\lambda t^\alpha x(t)$, where $\lambda\in\mathbb{C}$, and formally build the fractional power-series solution. In~\cite{vatsala2}, the authors solve the complete non-autonomous linear problem in symbolic form, with a distinct expression for the solution.

\subsection{Organization}

Concisely, the plan of the paper is the following. In Section~\ref{sec_L}, we introduce and work with the alternative L-fractional derivative and pose the linear-equation problem~\eqref{nostre_c} in the setting of L-fractional calculus. In Section~\ref{sec_ml}, we address L-fractional autonomous homogeneous linear equations with power series and define a new Mittag--Leffler-type function. In Section~\ref{sec_integ}, we study the associated integral operator of the L-fractional derivative, with the fundamental theorem of calculus, explicit computations, and norm estimates. This is necessary to solve, in Section~\ref{sec_comple}, the complete linear equation in the L-fractional sense with Picard's iterations, which corresponds to~\eqref{nostre_c}. The form of the solution and the proposed Mittag--Leffler-type function are analyzed with probabilistic arguments. The concrete case of the fractional-power source term is addressed. The uniqueness of the L-fractional solutions is justified and discussed. In Section~\ref{sec_sequ}, we investigate linear L-fractional differential equations of sequential type, with constant coefficients. We start with sequential order two and then turn to any order. By using power series, the main result is the derivation of the algebraic basis of solutions for an arbitrary order, in terms of the alternative Mittag--Leffler function. This is a nice extension of the classical theory. Some non-homogeneous equations are solved, with a generalized method of undetermined coefficients. In Section~\ref{sec_sequ_AA}, the investigation is concerned with linear L-fractional differential equations of the sequential type, with analytic coefficients and order two. Power series are employed again, where the coefficients of the solution satisfy recurrence relations. Lastly, Section~\ref{sec_concl} is devoted to future research lines.

\section{The L-Fractional Derivative and Formulation of the Complete Linear Equation} \label{sec_L}

The (Leibniz) L-fractional derivative of an absolutely continuous function $x:[0,T]\rightarrow\mathbb{C}^d$ is~\cite{lazo_linear,lazo_linear2}
\begin{equation} {}^L\! D^\alpha x(t)=\frac{{}^C\! D^\alpha x(t)}{{}^C\! D^\alpha t}, \label{derL} \end{equation}
where $\alpha\in (0,1)$ is the fractional order of differentiation, $t\in (0,T]$, and ${}^C\! D^\alpha$ is the Caputo fractional derivative~\eqref{derC}. We know that ${}^L\! D^\alpha x$ is defined almost everywhere on $[0,T]$, at least, by the properties of the Riemann--Liouville and Caputo operators. This fractional derivative~\eqref{derL} was envisioned to deal with fractional differentials in geometry~\cite{lazo1,lazo2}, 
\[
 \mathrm{d}^\alpha x(t)={}^L\! D^\alpha x(t)\,\mathrm{d}^\alpha t, 
 \]
and it has recently been utilized in~\cite{nito_aml} for logistic growth.

By~\eqref{particularrr},
\begin{equation} {}^L\! D^\alpha x(t)=\frac{\Gamma(2-\alpha)}{t^{1-\alpha}}\,{}^C\! D^\alpha x(t). \label{conddd2} \end{equation}

Two important properties of the L-fractional derivative are
\begin{equation} {}^L\! D^\alpha 1=0, \label{dLis0} \end{equation}
by~\eqref{der_ct}, and, in contrast to the Caputo derivative,
\[ {}^L\! D^\alpha t=1. \]

Property~\eqref{dLis0} will be very important when dealing with initial states in fractional differential equations and with power series, to derive difference equations for the expansion's coefficients. For the Riemann--Liouville or the $\Lambda$-derivative, the corresponding result~\eqref{dLis0} does not hold.

If 
\[
\Delta_s x(t)=\frac{x(t)-x(s)}{t-s}
\]
is the derivative discretization (mean past velocity over $[s,t]$), then the fractional derivative~\eqref{conddd2} interpolates between 
\[
 \Delta_0 x(t)=\frac{x(t)-x(0)}{t}=\underbrace{\frac{1}{t}\int_0^t x'(\tau)\mathrm{d}\tau}_{\text{mean value of $x'$}}, \text{ when }\alpha\rightarrow0^{+}, 
 \]
and, if $x$ is continuously differentiable on $[0,T]$,
\[
 {\lim_{s\rightarrow t}\Delta_s x(t)=x'(t),} \text{ when }\alpha\rightarrow 1^{-};
 \]
see~\eqref{bailando_AA} and~\eqref{bailando_BB}. We notice that, for the Caputo derivative, the value at $\alpha=0^+$ is $x(t)-x(0)$ instead of $(x(t)-x(0))/t$, which is not the mean value on $[0,t]$ exactly. 

Analogously to~\eqref{ode_cap}, an L-fractional differential equation is
\begin{equation} {}^L\! D^\alpha x(t)=f(t,x(t)), \label{LEDO} \end{equation}
for $t\in (0,T]$, with an initial condition or state $x(0)=x_0$, where $f:[0,T]\times \Omega\subseteq [0,T]\times\mathbb{R}^d\rightarrow\mathbb{R}^d$, or $f:[0,T]\times \Omega\subseteq [0,T]\times\mathbb{C}^d\rightarrow\mathbb{C}^d$, is a continuous function such that $x_0\in\Omega$. We remove $t=0$ from~\eqref{LEDO} by the division of $t^{1-\alpha}$ in~\eqref{conddd2}. In fact, we can interpret~\eqref{LEDO} in the almost-everywhere sense. Due to the close relation between ${}^L\! D^\alpha$ and ${}^C\! D^\alpha$, 
\[ {}^C\! D^\alpha x(t)=\frac{t^{1-\alpha}}{\Gamma(2-\alpha)}f(t,x(t)), \]
there are of course methods and results for Caputo fractional differential equations that readily apply to the L-fractional counterpart. For example, the finite difference scheme from~\cite{garrap2} is suitable, taking $t^{1-\alpha}$ into account. The proof of the Cauchy--Kovalevskaya theorem from~\cite{ck_meu} works as well, just by modifying the gamma-function factor in \cite{ck_meu} (expression~(2.19)), which is bounded too. Despite these matching properties, other topics on L-fractional differential equations deserve further attention; for example, the analysis of associated geometrical/physical features, the attainment of explicit and closed-form solutions, or the applicability in modeling. This paper is devoted to   obtaining solutions, which offers some insight into their behavior and the derivative.

By considering~\eqref{conddd2}, the target Equation~\eqref{nostre_c} can be transformed into
\begin{equation}
 {}^L\! D^\alpha x(t)=\mathcal{A}x(t)+\vartheta(t), \label{model2}
\end{equation}
where $\mathcal{A}\in\mathbb{C}^{d\times d}$ is a matrix and $\vartheta:[0,T]\rightarrow\mathbb{C}^d$ is a continuous function. The relations
\begin{equation} \mathcal{A}=\Gamma(2-\alpha)A,\quad \vartheta(t)=\frac{\Gamma(2-\alpha)}{t^{1-\alpha}}b(t) \label{relation_CL} \end{equation}
hold. The new system~\eqref{model2} has an autonomous homogeneous part, which is a key reduction to solve~\eqref{conddd2}. Due to the equivalence between~\eqref{conddd2} and~\eqref{model2} through~\eqref{relation_CL}, we will work with~\eqref{model2}.

As will be seen, solutions of~\eqref{model2} are $\mathcal{C}^\infty$ and analytic, with power-series expansions expressed in terms of $t^n$, not $t^{\alpha n}$. For L-fractional differential equations, the units of the vector field $f$ are time$^{-1}$, instead of time$^{-\alpha}$.

As the solution $x$ is smooth and not only absolutely continuous, we can conduct integration by parts on~\eqref{derC} and~\eqref{conddd2}, so the equalities~\eqref{caputo_parttts} and
\begin{equation} {}^L\! D^\alpha x(t)=x'(0)+\frac{1}{t^{1-\alpha}}\int_0^t (t-\tau)^{1-\alpha}x''(\tau)\mathrm{d}\tau \label{ja_ve_tr} \end{equation}
hold, pointwise, on $(0,T]$. Thus, these fractional derivatives contain a non-singular kernel function that is continuous on $[0,T]$,
\[ \tilde{\mathcal{K}}(t-\tau)=(t-\tau)^{1-\alpha}, \]
with the second-order derivative of $x$. Nevertheless, the L-fractional derivative has the denominator $t^{1-\alpha}$ that controls ${}^L\! D^\alpha x(t)$ when $t\rightarrow 0^+$, so  
\[ {}^C\! D^\alpha x(0)=0\neq {}^L\! D^\alpha x(0) \]
in general, and no controversies arise at the initial instant. This is a relevant property, considering the documented deficiencies of certain fractional operators with non-singular kernels~\cite{dieth}. As an illustration, the Caputo--Fabrizio derivative~\cite{cf} 
\vspace{6pt}
\[ {}^{CF}\! D^\alpha x(t)=\frac{1}{1-\alpha}\int_{0}^t \mathrm{e}^{-\frac{\alpha}{1-\alpha}(t-s)}x'(s)\mathrm{d}s \]
is always subject to the restriction 
\[ {}^{CF}\! D^\alpha x(0)=0, \]
so for applications on fractional differential equations, one is forced to work with the Losada--Nieto integral problem, which is equivalent to a certain ordinary differential  equation~\cite{losada,losada2,nieto_p2}. For the L-fractional derivative, the factor $1/t^{1-\alpha}$ avoids issues associated with bounded kernels and makes dimensionality consistent  so that the vector field $f$ is a true velocity from a physical viewpoint. In fact, for smooth functions $x$ on $[0,T]$, we have ${}^L\! D^\alpha x(0)=x'(0)$ by~\eqref{ja_ve_tr}, and we can consider $t=0$ in Equation~\eqref{LEDO} as well. Indeed, by translation in the integral (commutativity of the convolution) and L'H\^opital's rule,
\begin{equation}
 \begin{split}
\lim_{t\rightarrow 0^+} {}^L\! D^\alpha x(t)-x'(0)={} & \lim_{t\rightarrow 0^+}\frac{1}{t^{1-\alpha}}\int_0^t (t-\tau)^{1-\alpha}x''(\tau)\mathrm{d}\tau \; \text{ (by~\eqref{ja_ve_tr})} \\
= {} & \lim_{t\rightarrow 0^+}\frac{1}{t^{1-\alpha}}\int_0^t \tau^{1-\alpha}x''(t-\tau)\mathrm{d}\tau \; \text{ (convolution)} \\
= {} & \lim_{t\rightarrow 0^+}\frac{t^\alpha}{1-\alpha}\left( t^{1-\alpha}x''(0) + \int_0^t \tau^{1-\alpha}x'''(t-\tau)\mathrm{d}\tau\right) \; \text{ (L'H\^opital)} \\
= {} & 0.
\end{split}
\label{tre_acabanti}
\end{equation}
In the third equality above, we differentiated the denominator, which gives $(1-\alpha)t^{-\alpha}$, and the numerator, which is a parametric integral. The function ${}^L\! D^\alpha x$ in~\eqref{ja_ve_tr} is then continuous on $[0,T]$.

Table~\ref{tab} reports a schematized comparison between the L- and the Caputo fractional derivatives. It highlights the changes when normalizing the standard operator.

\begin{table}[hbtp!]
\begin{center}
\begin{tabular}{c|cc} 
 & Caputo derivative & L derivative \\ \hline
$\alpha=1$ & $x'(t)$ & $x'(t)$ \\ \hline
$\alpha=0$ & $x(t)-x(0)$ & $(x(t)-x(0))/t$ \\ \hline
$\alpha=0$, $t=0$ & $0$ & $x'(0)$ \\ \hline
derivative of constants & $0$ & $0$ \\ \hline
initial condition & $x(0)=x_0$ & $x(0)=x_0$ \\ \hline
derivative of $t$ & $\neq 1$ & $1$ \\ \hline
power series & fractional ($t^{\alpha n}$) & classical ($t^n$) \\ \hline
regularity of solution & absolutely continuous & smooth \\ \hline
$\alpha\in (0,1)$, $x'(0)$ & $\pm\infty$ & it is ${}^L\! D^\alpha x(0)\in (-\infty,\infty)$ \\ \hline
kernel & singular & singular and non-singular \\ \hline
issues at $t=0$ & no & no \\ \hline
units & time$^{-\alpha}$ & time$^{-1}$ \\ \hline
differential form & $\mathrm{d}^\alpha x(t)/(\mathrm{d} t)^{\alpha}$ & $\mathrm{d}^\alpha x(t)/\mathrm{d}^\alpha t$ \\ \hline
velocity & no & yes \\ \hline
fluxes & no & yes \\ \hline
memory & yes & yes \\ \hline
``exponential'' function & yes (Mittag-Leffler) & yes (another Mittag-Leffler) \\ \hline
\end{tabular}
\caption{Comparison between the Caputo and the L-fractional derivatives and their applications in differential equations.}
\label{tab}
\end{center}
\end{table}

In the notation, we will follow the convention that $\sum_{j=1}^0 =0$ and $\prod_{j=1}^0 =1$; that is, an empty sum is zero and an empty product is one. In the power series, $s^0=1$ for every $s\in\mathbb{C}$, even for $s=0$.

\section{Homogeneous Linear Equation: A New Mittag--Leffler-Type Function} \label{sec_ml}

Let us consider the simplest problem of L-fractional differential equations: 
\begin{equation} {}^L\! D^\alpha x=\lambda x, \label{simplest_L} \end{equation}
where $\lambda\in\mathbb{C}$, $t\geq0$, and the dimension $d$ is $1$. Analogously to Section~\ref{subs_conte}, which was focused on the Caputo setting, we consider a Taylor-series solution, but now in terms of $t^n$ instead of $t^{\alpha n}$. The motivation for this thought is the dimensionality time$^{-1}$ of the problem, instead of time$^{-\alpha}$; see the previous section.

The candidate power-series solution 
\begin{equation} x(t)=\sum_{n=0}^\infty x_n t^{n} \label{pw_Ln} \end{equation} 
satisfies, in a formal sense,
\begin{equation} \lambda \sum_{n=0}^\infty x_n t^{n}=\sum_{n=0}^\infty x_n\cdot {}^L\! D^\alpha(t^{n})=\sum_{n=0}^\infty x_{n+1} \frac{\Gamma(n+2)\Gamma(2-\alpha)}{\Gamma(n+2-\alpha)}t^n, \label{formal_stl}
\end{equation}
as per~\eqref{cond1}. After the terms are equated, the recursive equation for the coefficients is given by
\begin{equation} x_{n+1}=\frac{\Gamma(n+2-\alpha)}{\Gamma(n+2)\Gamma(2-\alpha)}\lambda x_n. \label{xn1_LL} \end{equation}

As it occurs with Caputo fractional equations, the fact that the L-fractional derivative of a constant is zero---see~\eqref{dLis0}---is key to deriving a first-order difference equation. The relation~\eqref{xn1_LL} can be solved:
\[ x_n=\frac{\lambda^n}{\Gamma(2-\alpha)^n \prod_{j=1}^n \frac{\Gamma(j+1)}{\Gamma(j+1-\alpha)}}x_0=\frac{\lambda^n}{\Gamma(2-\alpha)^n\Gamma(1+\alpha)^n \prod_{j=1}^n \binom{j}{j-\alpha}}x_0, \]
where $x_0=x(0)\in\mathbb{C}$ is the initial value. The solution of~\eqref{simplest_L} is thus expressed as
\begin{equation} x(t)=\mathcal{E}_{\alpha}(\lambda t)x_0, \label{xeat} \end{equation}
where
\begin{equation} \mathcal{E}_{\alpha}(s)=\sum_{n=0}^\infty \frac{s^n}{\Gamma(2-\alpha)^n \prod_{j=1}^n \frac{\Gamma(j+1)}{\Gamma(j+1-\alpha)}}=\sum_{n=0}^\infty \frac{s^n}{\Gamma(2-\alpha)^n\Gamma(1+\alpha)^n \prod_{j=1}^n \binom{j}{j-\alpha}}, \label{mlf2} \end{equation}
for $s\in\mathbb{C}$. This is a new extension of the exponential function, an alternative to the Mittag--Leffler formulation~\eqref{mlf}. It is related to the family of functions studied in~\cite{kiria}, with a distinct motivation.

For $\alpha\in (0,1]$, convergence of the new function~\eqref{mlf2} holds on $\mathbb{C}$ by the ratio test:
\begin{equation}
\begin{split}
\lim_{n\rightarrow\infty} \frac{\Gamma(2-\alpha)^n \prod_{j=1}^n \frac{\Gamma(j+1)}{\Gamma(j+1-\alpha)}}{\Gamma(2-\alpha)^{n+1} \prod_{j=1}^{n+1} \frac{\Gamma(j+1)}{\Gamma(j+1-\alpha)}}= {} & \frac{1}{\Gamma(2-\alpha)}\lim_{n\rightarrow\infty} \frac{\Gamma(n+2-\alpha)}{\Gamma(n+2)} \\
= {} & \frac{1}{\Gamma(2-\alpha)}\lim_{n\rightarrow\infty} \frac{1}{(n+2-\alpha)^\alpha} \\
= {} & 0.
\end{split}
 \label{turmell}
\end{equation}

The asymptotic relation 
\begin{equation}
 \Gamma(y+\alpha)\sim \Gamma(y)y^\alpha, 
 \label{saps_quina_e}
\end{equation}
when $y\rightarrow\infty$, which is a consequence of Stirling's formula, has been used. For the standard Mittag--Leffler function~\eqref{mlf}, the corresponding quotient~\eqref{turmell} behaves asymptotically as 
\[ \frac{1}{((n+1)\alpha+1-\alpha)^\alpha}, \]
which is lower by the factor $\Gamma(2-\alpha)\in (0,1)$ compared to our $\mathcal{E}_\alpha$. The fastest rate of convergence occurs for the classical exponential function, when $\alpha=1$, as the corresponding quotient~\eqref{turmell} is $1/(n+1)$ asymptotically.

The boundary values of $\mathcal{E}_\alpha$ are
\[ \mathcal{E}_0(s)=\frac{1}{1-s},\;|s|<1, \]
and
\[ \mathcal{E}_1(s)=\mathrm{e}^s,\; s\in\mathbb{C}. \]

Actually, although the solution~\eqref{xeat} converges by~\eqref{turmell}, it is still formal; see~\eqref{formal_stl}. Later, through the integral operator associated with the L-fractional derivative, we will prove that~\eqref{xeat} is indeed the solution for~\eqref{simplest_L} (Theorem~\ref{th_sol_comp}). For now, in this section, we are only interested in how the new Mittag--Leffler-type function~\eqref{mlf2} is built.

From~\eqref{mlf2}, a nice identity is
\vspace{-6pt}
\[ \mathcal{E}_{1/2}(s)=\sum_{n=0}^\infty \frac{s^n}{2^{n^2}}\prod_{j=1}^n \binom{2j}{j}. \]

This gives a new interpretation of the product of central binomial coefficients, 
\[ \prod_{j=1}^n \binom{2j}{j}, \]
in terms of the power-series solution to the fractional problem
\[ {}^L\! D^{1/2}x=x,\quad x(0)=1. \]

The development of this section can be readily adapted to matrix arguments. Let
\begin{equation} {}^L\! D^\alpha x=\mathcal{A} x, \label{adapted_se} \end{equation}
where $\mathcal{A}\in\mathbb{C}^{d\times d}$ is a matrix and $x$ takes vector values in $\mathbb{C}^d$. Then, the power-series method can be employed, which yields
\begin{equation} x(t)=\mathcal{E}_{\alpha}(\mathcal{A} t)x_0, \label{is_connn} \end{equation}
where $x_0\in\mathbb{C}^d$.

\section{On the Associated Integral Operator} \label{sec_integ}

In this section, we study the integral operator associated with the L-fractional derivative.

\subsection{Introduction}

By~\eqref{volt_capi} and~\eqref{conddd2}, the integral operator associated with ${}^L\! D^\alpha$ is
\vspace{6pt}
\begin{equation}
\begin{split}
 {}^L\! J^\alpha x(t)= {} & \frac{1}{\Gamma(\alpha)\Gamma(2-\alpha)}\int_0^t (t-s)^{\alpha-1}s^{1-\alpha}x(s)\mathrm{d}s \\
 = {} & \frac{t^{\alpha-1}}{\Gamma(\alpha)} \ast \left(\frac{t^{1-\alpha}}{\Gamma(2-\alpha)} x(t)\right) \\
 = {} & {}^C\! J^\alpha \left[ \frac{t^{1-\alpha}}{\Gamma(2-\alpha)} x\right](t).
\end{split}
 \label{convolL}
\end{equation}

If $x\in\mathrm{L}^1[0,T]$, then ${}^L\! J^\alpha x\in\mathrm{L}^1[0,T]$, by standard properties of the convolution. Note that, if $x$ is continuous on $[0,T]$, then ${}^L\! J^\alpha x$ is well defined everywhere on $[0,T]$ and poses no problem at $t=0$. Indeed,
\begin{align*}
 |{}^L\! J^\alpha x(t)|\leq {} & \left( \max_{[0,T]} |x|\right) \frac{T^{1-\alpha}}{\Gamma(\alpha)\Gamma(2-\alpha)}\int_0^t (t-s)^{\alpha-1}\mathrm{d}s \\
= {} & \left( \max_{[0,T]} |x|\right) \frac{T^{1-\alpha}t^\alpha}{\Gamma(\alpha)\Gamma(2-\alpha)\alpha}\stackrel{t\rightarrow0^+}{\longrightarrow}0.
\end{align*}

The same occurs for ${}^C\! J^\alpha x$.

We rigorously prove the L-fractional fundamental theorem of calculus in the following proposition. We first need a lemma on~\eqref{probl1_c} and~\eqref{probl2_c} concerning the Caputo fractional calculus. We  emphasize here the important remarks of~\cite{weber_pap} about the conditions and assumptions in fractional computations, as well as the rigorous results in~\cite{diethelm_llibre}.

\begin{Lemma} \label{lema_rigor_FC_Cap}
If $x:[0,T]\rightarrow\mathbb{C}$ is absolutely continuous, then~\eqref{probl1_c} holds for all $t\in [0,T]$ and~\eqref{probl2_c} holds for almost every $t\in [0,T]$. If $x$ is given by a fractional power series on $[0,T]$ (i.e., a power series evaluated at $t^\alpha$), then~\eqref{probl2_c} is verified at every $t\in [0,T]$.
\end{Lemma}
\begin{proof}
When $x$ is absolutely continuous, we know that 
\[ y={}^C\! D^\alpha x={}^{RL}\! J^{1-\alpha} x'\in\mathrm{L}^1[0,T] \]
exists almost everywhere. Indeed, recall that ${}^{RL}\! J^{1-\alpha}$ maps $\mathrm{L}^1[0,T]$ into $\mathrm{L}^1[0,T]$. Then,
\[ {}^C\! J^\alpha y={}^{RL}\! J^\alpha y={}^{RL}\! J^\alpha \circ {}^{RL}\! J^{1-\alpha} x'={}^{RL}\! J^1 x'=x-x_0, \]
for all $t\in [0,T]$. We used the integral operators~\eqref{rl_deirv} and~\eqref{volt_capi}, as well as \cite{weber_pap} (Lemma~3.4) for the composition ${}^{RL}\! J^\alpha \circ {}^{RL}\! J^{1-\alpha}$. The idea of this part has been taken from the first paragraph of the proof of \cite{weber_pap} (Theorem 5.1).

On the other hand, we know that ${}^{RL}\! J^{\alpha}x$ is absolutely continuous on $[0,T]$; see \cite{weber_pap} (Proposition~3.2) (it states that ${}^{RL}\! J^{\alpha}$ maps absolutely continuous functions onto absolutely continuous functions, among other results). We also know that
\[ {}^C\! D_{\ast}^\alpha\circ {}^C\! J^\alpha x(t)=x(t) \]
for all $t$ in $[0,T]$, where 
\begin{equation} {}^C\! D_{\ast}^\alpha={}^{RL}\! D^\alpha[x-x_0] \label{modified_cpitt} \end{equation}
is a modified Caputo operator~\cite{diethelm_llibre,weber_pap}. Since ${}^C\! J^\alpha x={}^{RL}\! J^{\alpha}x$ is absolutely continuous, Ref. \cite{diethelm_llibre} (Theorem~3.1) ensures that
\[ {}^C\! D^\alpha\circ {}^C\! J^\alpha x(t)={}^C\! D_{\ast}^\alpha\circ {}^C\! J^\alpha x(t)=x(t) \]
holds almost everywhere.

We remark that, in the literature, one usually finds applications of~\eqref{probl2_c} for every $t$ and when $x$ is merely continuous. This result is not true, because ${}^{RL}\! J^{\alpha}$ does not necessarily map continuous functions into absolutely continuous functions (see \cite{weber_pap} (Addendum~(3)) on a paper by Hardy and Littlewood), and ${}^C\! D^\alpha$ is not identically equal to ${}^C\! D_{\ast}^\alpha$.

The case of $x$ being given by a fractional power series on $[0,T]$ is postponed to  Remark~\ref{cor_D_s_remarK} after Corollary~\ref{cor_D_s}.
\end{proof}

\begin{Proposition} \label{propiidsf}
If $x:[0,T]\rightarrow\mathbb{C}$ is absolutely continuous, then
\begin{equation} {}^L\! J^\alpha \circ {}^L\! D^\alpha x(t)=x(t)-x(0) 
\label{probl1}
\end{equation}
for all $t\in [0,T]$, and 
\begin{equation} {}^L\! D^\alpha \circ {}^L\! J^\alpha x(t)=x(t) 
\label{probl2}
\end{equation}
for almost every $t\in (0,T]$. If $x$ is real analytic at $t=0$ with a radius of convergence $\geq T$, then~\eqref{probl2} is verified at every $t\in [0,T]$.
\end{Proposition}
\begin{proof}
On the one hand, by~\eqref{probl1_c} (see Lemma~\ref{lema_rigor_FC_Cap}),
\begin{equation*}
 \begin{split}
 {}^L\! J^\alpha \circ {}^L\! D^\alpha x(t)= {} & \frac{1}{\Gamma(\alpha)\Gamma(2-\alpha)}\int_0^t (t-s)^{\alpha-1}s^{1-\alpha}\cdot {}^L\! D^\alpha x(s)\mathrm{d}s \\
= {} & \frac{1}{\Gamma(\alpha)\Gamma(2-\alpha)}\int_0^t (t-s)^{\alpha-1}s^{1-\alpha}\cdot \frac{\Gamma(2-\alpha)}{s^{1-\alpha}}\cdot {}^C\! D^\alpha x(s)\mathrm{d}s \\
= {} & \frac{1}{\Gamma(\alpha)}\int_0^t (t-s)^{\alpha-1}\cdot {}^C\! D^\alpha x(s)\mathrm{d}s \\
= {} & {}^C\! J^\alpha\circ {}^C\! D^\alpha x(t) \\
= {} & x(t)-x(0).
\end{split}
\end{equation*}

On the other hand, clearly,
\[ \frac{t^{1-\alpha}}{\Gamma(2-\alpha)} x \]
is absolutely continuous on $[0,T]$. Then, for almost every $t$,
\begin{equation*}
\begin{split}
 {}^L\! D^\alpha \circ {}^L\! J^\alpha x(t)= {} & {}^L\! D^\alpha \circ {}^C\! J^\alpha \left[ \frac{t^{1-\alpha}}{\Gamma(2-\alpha)} x\right](t) \\
= {} & \frac{\Gamma(2-\alpha)}{t^{1-\alpha}}{}^C\! D^\alpha \circ {}^C\! J^\alpha \left[ \frac{t^{1-\alpha}}{\Gamma(2-\alpha)} x\right](t) \\
= {} & \frac{\Gamma(2-\alpha)}{t^{1-\alpha}}\frac{t^{1-\alpha}}{\Gamma(2-\alpha)} x(t) \\
= {} & x(t).
\end{split}
\end{equation*}

We used~\eqref{probl2_c} (see Lemma~\ref{lema_rigor_FC_Cap}). 

The part on $x$ being real analytic will be justified in Corollary~\ref{cor_D_s}.
\end{proof}

We denote
\begin{equation} \underbrace{ {}^L\! J^\alpha \circ \cdots \circ {}^L\! J^\alpha }_{m\text{ times}} ={}^L\! J^{m\circ \alpha},\;\;\; \underbrace{ {}^L\! D^\alpha \circ \cdots \circ {}^L\! D^\alpha }_{m\text{ times}} ={}^L\! D^{m\circ \alpha}, \label{notation_circ} \end{equation}
for $m\geq1$, whenever the compositions make sense. We use this notation to distinguish ${}^L\! J^{m\circ \alpha}$ from ${}^L\! J^{m\alpha}$ and ${}^L\! D^{m\circ \alpha}$ from ${}^L\! D^{m\alpha}$, where $m\alpha$ is another fractional index; see the following proposition. The main ideas are taken from the interesting note~\cite{cao_labora}.
\begin{Proposition}
${}^L\! J^{2\circ \alpha}\neq {}^L\! J^{2\alpha}$ and ${}^L\! D^{2\circ \alpha}\neq {}^L\! D^{2\alpha}$, if $0<\alpha\leq 1/2$.
\end{Proposition}
\begin{proof}
Suppose that 
\[ {}^L\! J^{2\circ \alpha}= {}^L\! J^{2\alpha}. \]

By definition~\eqref{notation_circ}, this means that 
\[ {}^L\! J^\alpha \circ {}^L\! J^\alpha= {}^L\! J^{2\alpha}. \]

Then, by~\eqref{probl1}, 
\[ {}^L\! J^\alpha \circ {}^L\! J^\alpha \circ {}^L\! D^{2\alpha} = {}^L\! J^{2\alpha} \circ {}^L\! D^{2\alpha}=\mathrm{Id}-x_0. \]

By~\eqref{probl2}, 
\begin{equation} {}^L\! D^\alpha \circ {}^L\! D^\alpha= {}^L\! D^{2\alpha}. \label{mis_jih} \end{equation}

This is the negation of the second condition in the proposition. Let us see that we arrive at a contradiction, with an adequate set of functions. We consider the operators from $\mathcal{C}^{\omega}$ to $\mathcal{C}^{\omega}$, where $\mathcal{C}^{\omega}$ is the vector space of real analytic functions at $t=0$ with values in $\mathbb{C}$. According to Proposition~\ref{propiidsf}, the fundamental theorem of calculus holds for \textit{every} 
 point $t$ with functions of $\mathcal{C}^{\omega}$, so the above compositions are justified. Since
\[ {}^L\! D^\beta\left(\sum_{n=0}^\infty y_n t^n\right)=\sum_{n=0}^\infty y_n\cdot {}^L\! D^\beta(t^{n})=\sum_{n=0}^\infty y_{n+1} \frac{\Gamma(n+2)\Gamma(2-\beta)}{\Gamma(n+2-\beta)}t^n, \]
for $0<\beta\leq1$ and $\sum_{n=0}^\infty |y_n| t^n<\infty$ (see the forthcoming Corollary~\ref{cor_D_s} for rigorous details), the operator 
\[ {}^L\! D^\beta:\mathcal{C}^{\omega}\rightarrow\mathcal{C}^{\omega} \]
is surjective and 
\begin{equation} \mathrm{Ker}({}^L\! D^\beta)=\{k:\,k\in\mathbb{C}\}. \label{boni2} \end{equation}

Consequently, given $1\in\mathcal{C}^{\omega}$, there exists $y\in\mathcal{C}^{\omega}$ such that
\begin{equation} 1={}^L\! D^\alpha y. \label{boni1} \end{equation}

Therefore, by~\eqref{boni1} and~\eqref{mis_jih},
\[ 0={}^L\! D^\alpha 1={}^L\! D^\alpha\circ {}^L\! D^\alpha y={}^L\! D^{2\alpha} y, \]
which implies that $y\in \mathrm{Ker}({}^L\! D^{2\alpha})$. Then, $y$ is constant by~\eqref{boni2} and 
\[ {}^L\! D^\alpha y=0, \]
contradicting~\eqref{boni1} and completing the proof.
\end{proof}

Recall~\cite{brezis} that a linear map $\Lambda$ between normed spaces $X$ and $Y$, expressed by $\Lambda:X\rightarrow Y$, is continuous if and only if there exists a constant $K>0$ such that 
\begin{equation} \| \Lambda x\|\leq K\|x\| \label{fita_cont} \end{equation}
for all $x\in X$. In such a case, the induced norm for $\Lambda$ is 
\begin{equation} \|\Lambda\|=\sup_{\|x\|\leq 1} \|\Lambda x\|=\sup_{\|x\|=1} \|\Lambda x\|=\min\{K>0:\,\|\Lambda x\|\leq K\|x\|,\;\forall\,x\in X\}. \label{norma_cont} \end{equation}

We denote by $\mathcal{L}(X,Y)$ the normed space of linear continuous maps from $X$ and $Y$, so that $\Lambda\in\mathcal{L}(X,Y)$. If $Y$ is a Banach space, then $\mathcal{L}(X,Y)$ is Banach too.

Let $|\cdot|$ be the usual Euclidean norm for vectors, which becomes the absolute value for real scalars and the modulus for complex scalars. The induced norm for matrices $A\in\mathbb{C}^{d\times d}$ is also denoted by $|\cdot|$: 
\[ |A|=\sup_{v\in\mathbb{C}^d,\,|v|\leq 1} |Av|=\sup_{v\in\mathbb{C}^d,\,|v|= 1} |Av|. \]

It satisfies the submultiplicative property, namely $|AB|\leq |A||B|$, for all matrices $A$ and $B$. We work with the specific case of $X=Y=\mathcal{C}[0,T]$, which is the Banach space of continuous functions $y=(y_1,\ldots,y_d):[0,T]\rightarrow\mathbb{C}^d$ with the supremum norm
\[ \|y\|_\infty = \max_{t\in [0,T]} |y(t)|. \]

The Banach space $\mathcal{L}(X,Y)$ is then denoted by $\mathcal{L}(\mathcal{C}[0,T])$, with the induced operator's norm $\|\cdot\|_\infty$ defined by~\eqref{norma_cont}.

The set $\mathcal{C}^p[0,T]$, for integers $p\geq1$ or $p=\infty$, is given by the functions that have partial derivatives up to order $p$ and are continuous on $[0,T]$.

\subsection{List of Results}

We state and prove the results that are needed to solve~\eqref{model2}. These are concerned with explicit computations, especially regarding the so-called beta function and norm estimates in $\mathcal{L}(\mathcal{C}[0,T])$.

\begin{Lemma} \label{lema_beta}
If $\delta>\alpha-2$ and $t>0$, then
\[ \int_0^t (t-s)^{\alpha-1} s^{1-\alpha+\delta}\mathrm{d}s=t^{1+\delta}\frac{\Gamma(2-\alpha+\delta)\Gamma(\alpha)}{\Gamma(2+\delta)}. \]
\end{Lemma}
\begin{proof}
We make the change of variable $s=tu$, $\mathrm{d}s=t\mathrm{d}u$, and the resulting integral is related to the beta function
\begin{equation} B(z_1,z_2)=\int_0^1 u^{z_1-1}(1-u)^{z_2-1}\mathrm{d}u, \label{beta1} \end{equation}
defined for complex numbers $z_1$ and $z_2$ such that $\mathrm{Re}(z_1)>0$ and $\mathrm{Re}(z_2)>0$. A key property~\cite{artin} of the beta function is its connection with the gamma function:
\begin{equation} B(z_1,z_2)=\frac{\Gamma(z_1)\Gamma(z_2)}{\Gamma(z_1+z_2)}. \label{beta2} \end{equation}

In our case, we have
\begin{equation*}
\begin{split}
\int_0^t (t-s)^{\alpha-1} s^{1-\alpha+\delta}\mathrm{d}s = {} & t\int_0^1 (t-tu)^{\alpha-1} (tu)^{1-\alpha+\delta}\mathrm{d}u \\
= {} & t^{1+\delta} \int_0^1 (1-u)^{\alpha-1}u^{1-\alpha+\delta}\mathrm{d}u \\
= {} & t^{1+\delta}\frac{\Gamma(2-\alpha+\delta)\Gamma(\alpha)}{\Gamma(2+\delta)},
\end{split}
\end{equation*}
where $z_1=2-\alpha+\delta>0$ and $z_2=\alpha>0$ in the notation of~\eqref{beta1} and~\eqref{beta2}.
\end{proof}

\begin{Proposition} \label{lema_coni}
If $\mathcal{C}[0,T]$ is endowed with the supremum norm $\|\cdot\|_\infty$, then
${}^L\! J^\alpha:\mathcal{C}[0,T]\rightarrow \mathcal{C}[0,T]$ is a continuous operator. Thus, if $\mathcal{L}(\mathcal{C}[0,T])$ denotes the Banach space of linear continuous maps from $\mathcal{C}[0,T]$ to $\mathcal{C}[0,T]$ with the induced norm $\|\cdot\|_\infty$, then ${}^L\! J^\alpha\in \mathcal{L}(\mathcal{C}[0,T])$ and $\|{}^L\! J^\alpha\|_\infty \leq T$.
\end{Proposition}
\begin{proof}
We first check that ${}^L\! J^\alpha$ is well defined from $\mathcal{C}[0,T]$ into $\mathcal{C}[0,T]$. Let $y\in \mathcal{C}[0,T]$. We rewrite the convolution~\eqref{convolL} as
\[ {}^L\! J^\alpha y(t)=\frac{1}{\Gamma(\alpha)\Gamma(2-\alpha)}\int_0^t s^{\alpha-1}(t-s)^{1-\alpha}y(t-s)\mathrm{d}s. \]

If $0<h<1$, then
\begin{align}
 {} & | {}^L\! J^\alpha y(t+h) - {}^L\! J^\alpha y(t)| \label{h0} \\
= {} & \frac{1}{\Gamma(\alpha)\Gamma(2-\alpha)}\left| \int_0^{t+h} s^{\alpha-1}(t+h-s)^{1-\alpha}y(t+h-s)\mathrm{d}s - \int_0^t s^{\alpha-1}(t-s)^{1-\alpha}y(t-s)\mathrm{d}s \right| \nonumber \\
\leq {} & \frac{1}{\Gamma(\alpha)\Gamma(2-\alpha)}\int_t^{t+h} s^{\alpha-1}(t+h-s)^{1-\alpha}|y(t+h-s)|\mathrm{d}s \nonumber \\
{} & + \frac{1}{\Gamma(\alpha)\Gamma(2-\alpha)}\int_0^t s^{\alpha-1} |(t+h-s)^{1-\alpha}y(t+h-s) - (t-s)^{1-\alpha}y(t-s)|\mathrm{d}s \nonumber \\
\leq {} & \frac{1}{\Gamma(\alpha)\Gamma(2-\alpha)} (T+1)\|y\|_\infty \int_t^{t+h} s^{\alpha-1}\mathrm{d}s \label{h1} \\
{} & + \frac{1}{\Gamma(\alpha)\Gamma(2-\alpha)}\int_0^t s^{\alpha-1} |(t+h-s)^{1-\alpha}y(t+h-s) - (t-s)^{1-\alpha}y(t-s)|\mathrm{d}s. \label{h2}
\end{align}
The bound $(t+h-s)^{1-\alpha}\leq (T+1)^{1-\alpha}\leq T+1$ has been used. We analyze the limit of both~\eqref{h1} and~\eqref{h2} when $h\rightarrow0$. On the one hand, for~\eqref{h1},
\[ \int_t^{t+h} s^{\alpha-1}\mathrm{d}s=\frac{(t+h)^\alpha-t^\alpha}{\alpha} \stackrel{h\rightarrow0}{\longrightarrow}0. \]

On the other hand, for~\eqref{h2},
\[ |(t+h-s)^{1-\alpha}y(t+h-s) - (t-s)^{1-\alpha}y(t-s)|\stackrel{h\rightarrow0}{\longrightarrow}0 \]
and 
\begin{align*}
{} & s^{\alpha-1} |(t+h-s)^{1-\alpha}y(t+h-s) - (t-s)^{1-\alpha}y(t-s)| \\
\leq {} & s^{\alpha-1} \left( (t+h-s)^{1-\alpha}|y(t+h-s)| + (t-s)^{1-\alpha} |y(t-s)| \right) \\
\leq {} & 2(T+1)\|y\|_\infty s^{\alpha-1}\in \mathrm{L}^1([0,T],\mathrm{d}s), 
\end{align*}
so the dominated convergence theorem ensures that
\[ \int_0^t s^{\alpha-1} |(t+h-s)^{1-\alpha}y(t+h-s) - (t-s)^{1-\alpha}y(t-s)|\mathrm{d}s \stackrel{h\rightarrow0}{\longrightarrow}0. \]

Thus, from~\eqref{h0},
\[ | {}^L\! J^\alpha y(t+h) - {}^L\! J^\alpha y(t)|\stackrel{h\rightarrow0}{\longrightarrow}0. \]

For $h<0$, one proceeds analogously, and then ${}^L\! J^\alpha\in \mathcal{C}[0,T]$, as wanted.

The linearity of ${}^L\! J^\alpha$ is clear, based on the properties of the integral. Now we prove continuity of ${}^L\! J^\alpha$ by using~\eqref{fita_cont}. If $y\in \mathcal{C}[0,T]$, then
\begin{align}
 |{}^L\! J^\alpha y(t)|\leq {} & \frac{1}{\Gamma(\alpha)\Gamma(2-\alpha)}\int_0^t (t-s)^{\alpha-1}s^{1-\alpha}|y(s)|\mathrm{d}s \nonumber \\
\leq {} & \|y\|_\infty \frac{1}{\Gamma(\alpha)\Gamma(2-\alpha)}\int_0^t (t-s)^{\alpha-1}s^{1-\alpha} \mathrm{d}s \nonumber \\
= {} & \|y\|_\infty \frac{1}{\Gamma(\alpha)\Gamma(2-\alpha)}t\frac{\Gamma(2-\alpha)\Gamma(\alpha)}{\Gamma(2)} \label{first_equ} \\
= {} & t\|y\|_\infty \label{prev_est} \\
\leq {} & T\|y\|_\infty. \nonumber
\end{align}

In the first equality~\eqref{first_equ}, Lemma~\ref{lema_beta} has been employed with $\delta=0>\alpha-2$. Hence,
\[ \|{}^L\! J^\alpha y\|_\infty \leq T\|y\|_\infty, \]
${}^L\! J^\alpha \in \mathcal{L}(\mathcal{C}[0,T])$,
and
\[ \|{}^L\! J^\alpha\|_\infty \leq T, \]
by~\eqref{norma_cont}.
\end{proof}

\begin{Corollary} \label{cor_D_s}
If
\[ \sum_{n=0}^\infty |x_n|t^n<\infty \]
for all $t\in [0,\epsilon]$, where $\epsilon>0$ and $x_n\in\mathbb{C}$, then
\[ {}^L\! D^\alpha \left( \sum_{n=0}^\infty x_n t^n \right)=\sum_{n=0}^\infty x_n \cdot {}^L\! D^\alpha (t^n) \]
on $[0,\epsilon]$. Furthermore,~\eqref{probl2} holds for all $t\in [0,\epsilon]$ for $\sum_{n=0}^\infty x_n t^n$, not just almost everywhere, hence completing the statement of the fundamental theorem of L-fractional calculus; see Proposition~\ref{propiidsf}.
\end{Corollary}
\begin{proof}
Let $x:[0,\epsilon]\rightarrow\mathbb{C}$ be defined by the power series,
\[ x(t)=\sum_{n=0}^\infty x_n t^n. \]

Consider new coefficients
\[ \tilde{x}_n=\frac{\Gamma(2-\alpha)\Gamma(n+2)}{\Gamma(2-\alpha+n)}x_{n+1}, \]
for $n\geq0$. Notice that
\[ \sum_{n=0}^\infty |\tilde{x}_n| t^n<\infty \]
on $[0,\epsilon]$, because
\[ \lim_{n\rightarrow\infty} \frac{\frac{\Gamma(2-\alpha)\Gamma(n+2)}{\Gamma(2-\alpha+n)}}{\frac{\Gamma(2-\alpha)\Gamma(n+1)}{\Gamma(1-\alpha+n)}}=1. \]

Then,
\begin{align}
{}^L\! J^\alpha \left( \sum_{n=0}^\infty \tilde{x}_n t^n\right)= {} & \sum_{n=0}^\infty \tilde{x}_n \cdot {}^L\! J^\alpha(t^n) \label{vest1} \\
= {} & \sum_{n=0}^\infty \tilde{x}_n \frac{\Gamma(2-\alpha+n)}{\Gamma(2-\alpha)\Gamma(n+2)}t^{n+1} \label{vest2} \\
= {} & \sum_{n=0}^\infty x_{n+1}t^{n+1}=x(t)-x_0. \label{vest3}
\end{align}

Equality~\eqref{vest1} holds by Proposition~\ref{lema_coni} (the convergence of $\sum_{n=0}^\infty \tilde{x}_n t^n$ is uniform on $[0,\epsilon]$, i.e., in the space $\mathcal{C}[0,\epsilon]$). In~\eqref{vest2}, the computation in Lemma~\ref{lema_beta} is used. Consequently,
\begin{align}
{}^L\! D^\alpha x(t)= {} & {}^L\! D^\alpha (x-x_0)(t) \nonumber \\
= {} & \sum_{n=0}^\infty \tilde{x}_n t^n \label{estaa1} \\
= {} & \sum_{n=0}^\infty \frac{\Gamma(2-\alpha)\Gamma(n+2)}{\Gamma(2-\alpha+n)}x_{n+1} t^n \nonumber \\
= {} & \sum_{n=0}^\infty x_n\cdot {}^L\! D^\alpha(t^n), \label{estaa3}
\end{align}
almost everywhere. For~\eqref{estaa1}, we use~\eqref{vest3} and~\eqref{probl2}. Now, as $x$ and $t^n$ are smooth, both ${}^L\! D^\alpha x$ and~\eqref{estaa3} are continuous on $[0,\epsilon]$; see~\eqref{ja_ve_tr} and~\eqref{tre_acabanti}. Hence, the previous equality almost everywhere becomes a pointwise equality for every $t\in [0,\epsilon]$. The point $t=0$ does not pose any problem, because 
\[ {}^L\! D^\alpha x(0)=x'(0)=x_1=\sum_{n=0}^\infty x_n\cdot {}^L\! D^\alpha(t^n)|_{t=0}, \]
by~\eqref{ja_ve_tr} and~\eqref{tre_acabanti}.

Finally, we need to check that~\eqref{probl2} holds for all $t\in [0,\epsilon]$, from the obtained results:
\begin{align*}
{}^L\! D^\alpha \circ {}^L\! J^\alpha x(t)= {} & {}^L\! D^\alpha\left[ \sum_{n=0}^\infty x_n \frac{\Gamma(2-\alpha+n)}{\Gamma(2-\alpha)\Gamma(n+2)}t^{n+1}\right] \\
= {} & \sum_{n=0}^\infty x_n \frac{\Gamma(2-\alpha+n)}{\Gamma(2-\alpha)\Gamma(n+2)}\cdot {}^L\! D^\alpha t^{n+1} \\
= {} & \sum_{n=0}^\infty x_n \frac{\Gamma(2-\alpha+n)}{\Gamma(2-\alpha)\Gamma(n+2)}\cdot \frac{\Gamma(2-\alpha)\Gamma(n+2)}{\Gamma(2-\alpha+n)}t^n \\
= {} & \sum_{n=0}^\infty x_n t^n \\
= {} & x(t).
\end{align*}
\end{proof}

\begin{Remark} \label{cor_D_s_remarK}
In 
 the Caputo fractional calculus, the previous Corollary~\ref{cor_D_s} reads as follows:

``If
\[ \sum_{n=0}^\infty |x_n|t^{\alpha n}<\infty \]
for all $t\in [0,\epsilon]$, where $\epsilon>0$ and $x_n\in\mathbb{C}$, then
\begin{equation} {}^C\! D^\alpha \left( \sum_{n=0}^\infty x_n t^{\alpha n} \right)=\sum_{n=0}^\infty x_n \cdot {}^C\! D^\alpha (t^{\alpha n}) \label{ksmdwjws} \end{equation}
on $[0,\epsilon]$. Furthermore,~\eqref{probl2_c} holds for all $t\in [0,\epsilon]$ for $\sum_{n=0}^\infty x_n t^{\alpha n}$, not just almost everywhere, hence completing the statement of the fundamental theorem of Caputo fractional calculus, see Lemma~\ref{lema_rigor_FC_Cap}''.

This property is often used in the literature when solving linear and nonlinear fractional models in the Caputo sense; see, for example,~\eqref{ja_no_es_faormal}. Here, we validate it rigorously based on the operator's theory.

We note that 
\[ x(t)=\sum_{n=0}^\infty x_n t^{\alpha n} \]
is absolutely continuous on $[0,\epsilon]$. Indeed, we decompose $x$ as
\begin{equation} x(t)=\sum_{n=0}^{N_{\alpha}-1} x_n t^{\alpha n}+\sum_{n=N_{\alpha}}^\infty x_n t^{\alpha n}, \label{chalorisdf} \end{equation}
where $N_{\alpha}\geq1$ is an integer satisfying $\alpha\cdot N_{\alpha}\geq1$. The first sum in~\eqref{chalorisdf} is a finite combination of absolutely continuous functions and hence absolutely continuous. The second sum in~\eqref{chalorisdf} is $\mathcal{C}^1[0,\epsilon]$, because the series of ordinary derivatives converges uniformly.

The proof of the corresponding formula
\vspace{6pt}
\[ {}^C\! J^\alpha \left( \sum_{n=0}^\infty \tilde{x}_n t^{\alpha n}\right)=x(t)-x_0, \]
where
\[ \tilde{x}_n=\frac{\Gamma((n+1)\alpha+1)}{\Gamma(n\alpha+1)}x_{n+1}, \]
is analogous to Corollary~\ref{cor_D_s} until~\eqref{vest3}. This is due to the fact that ${}^C\! J^\alpha$ is also an element of $\mathcal{L}(\mathcal{C}[0,\epsilon])$.

Now, the part of the proof until~\eqref{estaa3} in the Caputo setting, which justifies the equality~\eqref{ksmdwjws} almost everywhere on $[0,\epsilon]$, is analogous too. One needs to use~\eqref{probl2_c} from Lemma~\ref{lema_rigor_FC_Cap}. To finally prove that the equality almost everywhere becomes a pointwise equality for every $t\in [0,\epsilon]$, we notice that the right-hand side of~\eqref{ksmdwjws} is clearly continuous (by uniform convergence), and the left-hand side of~\eqref{ksmdwjws} satisfies, at every $t$,
\begin{equation} {}^C\! D^\alpha x(t)=\sum_{n=0}^{N_{\alpha}} x_n\cdot {}^C\! D^\alpha t^{\alpha n}+{}^C\! D^\alpha\left(\sum_{n=N_{\alpha}+1}^\infty x_n t^{\alpha n}\right), \label{uhsdkjk} \end{equation}
as per~\eqref{chalorisdf}. The first sum in~\eqref{uhsdkjk} is finite and continuous. The second part of~\eqref{uhsdkjk} is continuous too, because
\[ \sum_{n=N_{\alpha}+1}^\infty x_n t^{\alpha n}\in\mathcal{C}^2[0,\epsilon] \]
and~\eqref{caputo_parttts} holds. Therefore, both sides of~\eqref{ksmdwjws} are continuous on $[0,\epsilon]$, so we have the equality~\eqref{ksmdwjws} at every $t\in [0,\epsilon]$, as wanted. The remark is concluded.
\end{Remark}

\begin{Proposition} \label{lemma_L_pow}
If $\delta>\alpha-2$, $m\geq1$ and $t>0$, then
\begin{equation}
 {}^L\! J^{m\circ \alpha} t^\delta=\frac{\prod_{i=2}^{m+1} \Gamma(i-\alpha+\delta)}{\Gamma(2-\alpha)^m \prod_{i=2}^{m+1} \Gamma(i+\delta)}t^{m+\delta}. \label{Jtd}
\end{equation}
\end{Proposition}
\begin{proof}
By induction on $m$, for $m=1$, we have
\begin{align*}
 {}^L\! J^{\alpha} t^\delta= {} & \frac{1}{\Gamma(\alpha)\Gamma(2-\alpha)}\int_0^t (t-s)^{\alpha-1}s^{1-\alpha}s^\delta\mathrm{d}s \\
= {} & \frac{1}{\Gamma(\alpha)\Gamma(2-\alpha)}\times t^{1+\delta}\frac{\Gamma(2-\alpha+\delta)\Gamma(\alpha)}{\Gamma(2+\delta)} \\
= {} & \frac{\Gamma(2-\alpha+\delta)}{\Gamma(2-\alpha)\Gamma(2+\delta)}t^{1+\delta},
\end{align*}
after applying Lemma~\ref{lema_beta}. Now suppose the result is true for $m-1$ (induction hypothesis). Then,
\begin{align}
 {}^L\! J^{m\circ \alpha} t^\delta= {} & {}^L\! J^\alpha\circ {}^L\! J^{(m-1)\circ \alpha} t^\delta \label{firr_e} \\
= {} & {}^L\! J^\alpha\left( \frac{\prod_{i=2}^{m} \Gamma(i-\alpha+\delta)}{\Gamma(2-\alpha)^{m-1} \prod_{i=2}^{m} \Gamma(i+\delta)}t^{m-1+\delta}\right) \label{seco_eq} \\
= {} & \frac{\prod_{i=2}^{m} \Gamma(i-\alpha+\delta)}{\Gamma(2-\alpha)^{m-1} \prod_{i=2}^{m} \Gamma(i+\delta)} {}^L\! J^\alpha t^{m-1+\delta} \nonumber \\
= {} & \frac{\prod_{i=2}^{m} \Gamma(i-\alpha+\delta)}{\Gamma(2-\alpha)^{m-1} \prod_{i=2}^{m} \Gamma(i+\delta)}\times \frac{1}{\Gamma(\alpha)\Gamma(2-\alpha)}\int_0^t (t-s)^{\alpha-1}s^{1-\alpha}s^{m-1+\delta}\mathrm{d}s \nonumber
 \end{align}
 \begin{align}
= {} & \frac{\prod_{i=2}^{m} \Gamma(i-\alpha+\delta)}{\Gamma(2-\alpha)^{m-1} \prod_{i=2}^{m} \Gamma(i+\delta)} \times \frac{1}{\Gamma(\alpha)\Gamma(2-\alpha)} t^{1+(m-1+\delta)}\frac{\Gamma(2-\alpha+(m-1+\delta))\Gamma(\alpha)}{\Gamma(2+(m-1+\delta))} \label{fif_eq} \\
= {} & \frac{\prod_{i=2}^{m+1} \Gamma(i-\alpha+\delta)}{\Gamma(2-\alpha)^m \prod_{i=2}^{m+1} \Gamma(i+\delta)}t^{m+\delta}. \nonumber
\end{align}
The first equality~\eqref{firr_e} is the definition~\eqref{notation_circ}. In the second equality~\eqref{seco_eq}, the induction hypothesis is employed. In the fifth equality~\eqref{fif_eq}, Lemma~\ref{lema_beta} is used with $m-1+\delta$ instead of $\delta$.
\end{proof}

\begin{Proposition} \label{em_fit_inf}
If $m\geq1$, $t\in [0,T]$ and $y\in\mathcal{C}[0,T]$, then
\begin{equation}
 |{}^L\! J^{m\circ \alpha}y(t)| \leq \|y\|_{\infty}\frac{\prod_{i=2}^{m+1} \Gamma(i-\alpha)}{\Gamma(2-\alpha)^{m} \prod_{i=2}^{m+1} \Gamma(i)}t^m \label{Jtd22ante}
\end{equation}
and
\begin{equation}
 \|{}^L\! J^{m\circ \alpha}\|_{\infty}\leq \frac{\prod_{i=2}^{m+1} \Gamma(i-\alpha)}{\Gamma(2-\alpha)^{m} \prod_{i=2}^{m+1} \Gamma(i)}T^m. \label{Jtd22}
\end{equation}
\end{Proposition}
\begin{proof}
We first notice that ${}^L\! J^{m\circ \alpha}\in\mathcal{L}(\mathcal{C}[0,T])$, by Proposition~\ref{lema_coni} and definition~\eqref{notation_circ}. Second,~\eqref{Jtd22} is a consequence of~\eqref{Jtd22ante}. For~\eqref{Jtd22ante}, we proceed by induction on $m\geq1$. For $m=1$, the result is known by our previous estimate~\eqref{prev_est}. Suppose the inequality for $m-1$ (induction hypothesis), and let us prove it for $m$. We have
\begin{align*}
{}^L\! J^{m\circ \alpha}y(t)= {} & {}^L\! J^\alpha\circ {}^L\! J^{(m-1)\circ \alpha} y(t) \\
= {} & \frac{1}{\Gamma(\alpha)\Gamma(2-\alpha)}\int_0^t (t-s)^{\alpha-1}s^{1-\alpha}\cdot {}^L\! J^{(m-1)\circ \alpha} y(s)\mathrm{d}s.
\end{align*}

By applying $|\cdot|$, we have
\begin{align}
|{}^L\! J^{m\circ \alpha}y(t)|\leq {} & \frac{1}{\Gamma(\alpha)\Gamma(2-\alpha)}\int_0^t (t-s)^{\alpha-1}s^{1-\alpha}|{}^L\! J^{(m-1)\circ \alpha} y(s)|\mathrm{d}s \nonumber \\
\leq {} & \|y\|_\infty \frac{\prod_{i=2}^{m} \Gamma(i-\alpha)}{\Gamma(2-\alpha)^{m-1} \prod_{i=2}^{m} \Gamma(i)} \frac{1}{\Gamma(\alpha)\Gamma(2-\alpha)}\int_0^t (t-s)^{\alpha-1}s^{1-\alpha+(m-1)}\mathrm{d}s \label{hert_dos} \\
= {} & \|y\|_\infty \frac{\prod_{i=2}^{m} \Gamma(i-\alpha)}{\Gamma(2-\alpha)^{m-1} \prod_{i=2}^{m} \Gamma(i)}\frac{1}{\Gamma(\alpha)\Gamma(2-\alpha)}t^{m}\frac{\Gamma(2-\alpha+(m-1))\Gamma(\alpha)}{\Gamma(2+(m-1))} \label{hert_tre} \\
= {} & \|y\|_{\infty}\frac{\prod_{i=2}^{m+1} \Gamma(i-\alpha)}{\Gamma(2-\alpha)^{m} \prod_{i=2}^{m+1} \Gamma(i)}t^m. \nonumber
\end{align}

In the second inequality~\eqref{hert_dos}, the induction hypothesis is used. In the first equality~\eqref{hert_tre}, Lemma~\ref{lema_beta} is employed with $\delta=m$.
\end{proof}

\begin{Proposition} \label{ja_esta}
The series of operators
\vspace{6pt}
\begin{equation} \sum_{j=0}^\infty \mathcal{A}^j \cdot {}^L\! J^{(j+1)\circ \alpha} \label{meitat} \end{equation}
is convergent in $\mathcal{L}(\mathcal{C}[0,T])$.
\end{Proposition}
\begin{proof}
We first notice that $\mathcal{A}^j \cdot {}^L\! J^{(j+1)\circ \alpha}\in \mathcal{L}(\mathcal{C}[0,T])$, with
\[ \|\mathcal{A}^j \cdot {}^L\! J^{(j+1)\circ \alpha}\|_\infty=|\mathcal{A}^j|\|{}^L\! J^{(j+1)\circ \alpha}\|_\infty\leq |\mathcal{A}|^j \|{}^L\! J^{(j+1)\circ \alpha}\|_\infty. \]

The submultiplicative property of the matrix norm has been used. Since $\mathcal{L}(\mathcal{C}[0,T])$ is a Banach space, for~\eqref{meitat}, it suffices to prove that
\begin{equation} \sum_{j=0}^\infty |\mathcal{A}|^j \|{}^L\! J^{(j+1)\circ \alpha}\|_\infty<\infty. \label{series1} \end{equation}

By Proposition~\ref{em_fit_inf}, specifically inequality~\eqref{Jtd22}, we bound the series in~\eqref{series1} as
\begin{equation} \sum_{j=0}^\infty |\mathcal{A}|^j \|{}^L\! J^{(j+1)\circ \alpha}\|_\infty \leq \sum_{j=0}^\infty |\mathcal{A}|^j \frac{\prod_{i=2}^{j+2} \Gamma(i-\alpha)}{\Gamma(2-\alpha)^{j+1} \prod_{i=2}^{j+2} \Gamma(i)}T^{j+1}. \label{ratio}
\end{equation}

To justify the convergence of the right-hand series in~\eqref{ratio}, we employ the ratio test:
\begin{align*}
 \lim_{j\rightarrow\infty} \frac{|\mathcal{A}|^{j+1} \frac{\prod_{i=2}^{j+3} \Gamma(i-\alpha)}{\Gamma(2-\alpha)^{j+2} \prod_{i=2}^{j+3} \Gamma(i)}T^{j+2}}{|\mathcal{A}|^j \frac{\prod_{i=2}^{j+2} \Gamma(i-\alpha)}{\Gamma(2-\alpha)^{j+1} \prod_{i=2}^{j+2} \Gamma(i)}T^{j+1}} {} & =\frac{|\mathcal{A}| T}{\Gamma(2-\alpha)}\lim_{j\rightarrow\infty} \frac{\Gamma(j+3-\alpha)}{\Gamma(j+3)} \\
= {} & \frac{|\mathcal{A}| T}{\Gamma(2-\alpha)}\lim_{j\rightarrow\infty} \frac{1}{\Gamma(j+3-\alpha)^\alpha}=0. 
\end{align*}
\end{proof}

\section{Solution of the Complete Linear Equation} \label{sec_comple}

In this section, we solve the complete linear equation in the L-fractional sense with Picard's iterations. Later, we give a probabilistic form to this solution by using the beta-distributed delay of the L-fractional operators. The new Mittag--Leffler-type function is connected with basic probability theory as well, via generalized moment-generating functions. The concrete case of fractional-power source term is addressed. Finally, the uniqueness of L-fractional solutions is justified and discussed.

\subsection{General Equation and Explicit Solution}

We give the explicit solution to~\eqref{model2}. We remark on the difference between~\eqref{model2} and the integral problem~\eqref{deixar}, considering the absolutely continuous functions (Proposition~\ref{propiidsf}); this is not usually carried out in the literature, which states an equivalence vaguely.

\begin{Proposition} \label{ard}
The new Mittag--Leffler-type function $\mathcal{E}_{\alpha}$---see~\eqref{mlf2}---converges for matrix arguments $s=\mathcal{A}\in\mathbb{C}^{d\times d}$. The convergence for $\mathcal{E}_{\alpha}(\mathcal{A}t)$ is uniform on $[0,T]$, and hence~\eqref{is_connn} belongs to $\mathcal{C}[0,T]$.
\end{Proposition}
\begin{proof}
For $t\in [0,T]$, we have
\begin{align*}
 \sum_{n=0}^\infty \frac{|A^n|t^n}{\Gamma(2-\alpha)^n \prod_{j=1}^n \frac{\Gamma(j+1)}{\Gamma(j+1-\alpha)}} {} & \leq \sum_{n=0}^\infty \frac{|A|^nt^n}{\Gamma(2-\alpha)^n \prod_{j=1}^n \frac{\Gamma(j+1)}{\Gamma(j+1-\alpha)}} \\
\leq {} & \sum_{n=0}^\infty \frac{|A|^n T^n}{\Gamma(2-\alpha)^n \prod_{j=1}^n \frac{\Gamma(j+1)}{\Gamma(j+1-\alpha)}},
\end{align*}
by the submultiplicative property of the matrix norm. The convergence of the last series, which is independent of $t$, is checked with the ratio test; see~\eqref{turmell}. Thus, the series of $\mathcal{E}_{\alpha}(\mathcal{A}t)$ exhibits uniform convergence on $[0,T]$. In particular, for $t=1$, the function $\mathcal{E}_{\alpha}(\mathcal{A})$ is well defined. Finally, the continuity of $t\mapsto \mathcal{E}_{\alpha}(\mathcal{A}t)$ is clear, because it is the uniform limit of polynomials, which are continuous.
\end{proof}

\begin{Theorem} \label{th_sol_comp}
The solution of
\begin{equation}
 x(t)=x_0+{}^L\! J^\alpha (\mathcal{A}x(t)+\vartheta(t))=x_0+\frac{1}{\Gamma(\alpha)\Gamma(2-\alpha)}\int_0^t (t-s)^{\alpha-1}s^{1-\alpha}(\mathcal{A}x(s)+\vartheta(s))\mathrm{d}s
 \label{deixar}
\end{equation}
on $[0,T]$, with initial condition $x(0)=x_0$, is
\begin{equation} x(t)=\mathcal{E}_{\alpha}(\mathcal{A}t)x_0+\sum_{j=0}^\infty \mathcal{A}^j \cdot {}^L\! J^{(j+1)\circ \alpha}\vartheta(t). \label{imposs} \end{equation}

If $x$ and $\vartheta$ are absolutely continuous on $[0,T]$, then~\eqref{imposs} solves~\eqref{model2} almost everywhere on $[0,T]$. If $x$ and $\vartheta$ are given by power series on $[0,T]$, then~\eqref{imposs} solves~\eqref{model2} for every $t\in [0,T]$.
\end{Theorem}
\begin{proof}
Since $\vartheta\in\mathcal{C}[0,T]$, the function $x$ in~\eqref{imposs} belongs to $\mathcal{C}[0,T]$, by Propositions~\ref{ja_esta} and~\ref{ard}. We need to check that $x$ in~\eqref{imposs} is a fixed point of the associated Volterra integral operator~\eqref{deixar}. We build the solution to~\eqref{deixar} with Picard's iteration method:
\begin{equation}
 x_k(t)=x_0+{}^L\! J^\alpha (\mathcal{A}x_{k-1}(t)+\vartheta(t))=x_0+\frac{1}{\Gamma(\alpha)\Gamma(2-\alpha)}\int_0^t (t-s)^{\alpha-1}s^{1-\alpha}(\mathcal{A}x_{k-1}(s)+\vartheta(s))\mathrm{d}s,
 \label{picard}
\end{equation}
for $k\geq1$.

Let us see by induction on $k$ that
\begin{equation}
 x_k(t)=\left(\sum_{j=0}^k t^j \mathcal{A}^j \prod_{i=2}^{j+1} \frac{\Gamma(i-\alpha)}{\Gamma(2-\alpha)\Gamma(i)}\right)x_0+\sum_{j=0}^{k-1} \mathcal{A}^j \cdot {}^L\! J^{(j+1)\circ \alpha}\vartheta(t). \label{forma_a_prov}
\end{equation}

For $k=0$, it is clear because the identity $x_0=x_0$ is obtained. Suppose that the expression is true for $k-1$. We have
\begin{align}
x_k(t)= {} & x_0+{}^L\! J^\alpha (\mathcal{A}x_{k-1}(t)+\vartheta(t)) \label{maris1} \\
= {} & x_0 + \left(\sum_{j=0}^{k-1} \left( {}^L\! J^\alpha t^j\right) \mathcal{A}^{j+1} \prod_{i=2}^{j+1} \frac{\Gamma(i-\alpha)}{\Gamma(2-\alpha)\Gamma(i)}\right)x_0 + \sum_{j=0}^{k-2} \mathcal{A}^{j+1} \cdot {}^L\! J^{(j+2)\circ \alpha}\vartheta(t) + {}^L\! J^\alpha\vartheta(t) \label{maris2} \\
= {} & x_0 + \left[\sum_{j=0}^{k-1} \left( t^{1+j}\frac{\Gamma(2-\alpha+j)}{\Gamma(2-\alpha)\Gamma(2+j)}\right) \mathcal{A}^{j+1} \prod_{i=2}^{j+1} \frac{\Gamma(i-\alpha)}{\Gamma(2-\alpha)\Gamma(i)}\right]x_0 \nonumber \\
{} & + \sum_{j=0}^{k-2} \mathcal{A}^{j+1} \cdot {}^L\! J^{(j+2)\circ \alpha}\vartheta(t) + {}^L\! J^\alpha\vartheta(t) \label{maris3} \\
= {} & x_0 + \left(\sum_{j=0}^{k-1} t^{1+j} \mathcal{A}^{j+1} \prod_{i=2}^{j+2} \frac{\Gamma(i-\alpha)}{\Gamma(2-\alpha)\Gamma(i)}\right)x_0 + \sum_{j=0}^{k-2} \mathcal{A}^{j+1} \cdot {}^L\! J^{(j+2)\circ \alpha}\vartheta(t) + {}^L\! J^\alpha\vartheta(t) \nonumber \\
= {} & \left(\sum_{j=0}^k t^j \mathcal{A}^j \prod_{i=2}^{j+1} \frac{\Gamma(i-\alpha)}{\Gamma(2-\alpha)\Gamma(i)}\right)x_0+\sum_{j=0}^{k-1} \mathcal{A}^j \cdot {}^L\! J^{(j+1)\circ \alpha}\vartheta(t), \nonumber
\end{align}
which is exactly~\eqref{forma_a_prov}. In the first equality~\eqref{maris1}, we use~\eqref{picard}. The second equality~\eqref{maris2} is the induction hypothesis. The third equality~\eqref{maris3} is obtained from Lemma~\ref{lema_beta}.

From the form of $x_k$ in~\eqref{forma_a_prov} and Propositions~\ref{ja_esta} and~\ref{ard}, the convergence of $x_k$ toward  $x$ in~\eqref{imposs} is guaranteed in $\mathcal{C}[0,T]$. We need to check that this $x$ indeed solves~\eqref{deixar}.

Since $x_k\rightarrow x$ in the sense of $\mathcal{C}[0,T]$ as $k\rightarrow\infty$, we obtain that $\mathcal{A}x_{k-1}+\vartheta \rightarrow \mathcal{A}x+\vartheta$ in $\mathcal{C}[0,T]$. By Proposition~\ref{lema_coni}, we know that ${}^L\! J^\alpha\in\mathcal{L}(\mathcal{C}[0,T])$; therefore,
\[ \lim_{k\rightarrow\infty} {}^L\! J^\alpha (\mathcal{A}x_{k-1}+\vartheta)={}^L\! J^\alpha (\mathcal{A}x+\vartheta) \]
in $\mathcal{C}[0,T]$. Thus, taking limits as $k\rightarrow\infty$ in the recurrence's definition~\eqref{picard}, the fixed-point identity~\eqref{deixar} is established, as wanted.

By~\eqref{probl2} in Proposition~\ref{propiidsf}, if $x$ and $\vartheta$ are absolutely continuous on $[0,T]$, then~\eqref{imposs} solves~\eqref{model2} almost everywhere on $[0,T]$. If $x$ and $\vartheta$ are given by power series on $[0,T]$, then~\eqref{imposs} solves~\eqref{model2} for every $t\in [0,T]$.
\end{proof}

\subsection{A Link with Probability Theory} \label{link_prrrp}

For computations and proofs concerning ${}^L\! J^\alpha$, the incorporation of $t^{1-\alpha}$ in the  convolution of~\eqref{convolL} is obviously a handicap. For the Caputo fractional derivative, the fact that ${}^C\! J^\alpha y(t)=\frac{1}{\Gamma(\alpha)}t^{\alpha-1}\ast y(t)$ (see~\eqref{volt_capi}) and the associative property of the convolution permit having the iterations of ${}^C\! J^\alpha$:
\begin{equation}
\begin{split}
 \underbrace{{}^C\! J^\alpha \circ \cdots \circ {}^C\! J^\alpha}_{m\text{ times}} y(t)= {} & \frac{1}{\Gamma(\alpha)^m}(\underbrace{t^{\alpha-1}\ast \cdots \ast t^{\alpha-1}}_{m\text{ times}})\ast y(t) \\
= {} & \frac{1}{\Gamma(\alpha)^m}t^{m\alpha-1}\ast y(t).
\end{split}
 \label{molta_civada}
\end{equation}

Unfortunately, this is not the case for the L-fractional derivative and its iterated integral operator ${}^L\! J^{m\circ \alpha}$, which has an effect on the computation of the solution~\eqref{imposs}.

A probabilistic interpretation~\cite{barros} may help us understand the structure of ${}^L\! J^{m\circ \alpha}$ more. From the definition~\eqref{convolL}, we notice that
\begin{equation} {}^L\! J^\alpha y(t)=t \mathbb{E}[y(t\mathfrak{U})], \label{proba_L} \end{equation}
where $\mathfrak{U}$ is a random variable with distribution $\mathrm{Beta}(2-\alpha,\alpha)$ and $\mathbb{E}$ is the expectation operator. The L-fractional derivative~\eqref{conddd2} is
\begin{equation} {}^L\! D^\alpha y(t)=\mathbb{E}[y'(t\mathfrak{W})], \label{memo_L} \end{equation}
where $\mathfrak{W}$ is a random variable with distribution $\mathrm{Beta}(1,1-\alpha)$. Expression~\eqref{memo_L} emphasizes the memory property and the non-local behavior associated with the fractional derivative. Lemma~\ref{lema_beta} is, in fact, a result of statistical moments of the beta distribution. When $\alpha=1$, we obtain the ordinary operators that depend on $\mathrm{Uniform}(0,1)$ distributions. The iterations of~\eqref{proba_L} are the following: 
\begin{align*}
 {}^L\! J^{2\circ \alpha} y(t)= {} & t \mathbb{E}_{\mathfrak{U}_2}[{}^L\! J^\alpha(t\mathfrak{U}_2)] \\
= {} & t \mathbb{E}_{\mathfrak{U}_2}[ t\mathfrak{U}_2 \mathbb{E}_{\mathfrak{U}_1}[y(t\mathfrak{U}_1\mathfrak{U}_2)]] \\
= {} & t^2 \mathbb{E}_{\mathfrak{U}_2}[ \mathfrak{U}_2 \mathbb{E}_{\mathfrak{U}_1}[y(t\mathfrak{U}_1\mathfrak{U}_2)]],
\end{align*}
\begin{align*}
 {}^L\! J^{3\circ \alpha} y(t)= {} & t\mathbb{E}_{\mathfrak{U}_3}[{}^L\! J^{2\circ \alpha} y(t\mathfrak{U}_3)] \\
= {} & t\mathbb{E}_{\mathfrak{U}_3}[(t\mathfrak{U}_3)^2\mathbb{E}_{\mathfrak{U}_2}[\mathfrak{U}_2 \mathbb{E}_{\mathfrak{U}_1}[y(t\mathfrak{U}_1\mathfrak{U}_2\mathfrak{U}_3]]] \\
= {} & t^3\mathbb{E}_{\mathfrak{U}_3}[\mathfrak{U}_3^2\mathbb{E}_{\mathfrak{U}_2}[\mathfrak{U}_2 \mathbb{E}_{\mathfrak{U}_1}[y(t\mathfrak{U}_1\mathfrak{U}_2\mathfrak{U}_3]]],\;\ldots,
\end{align*}
\begin{equation} {}^L\! J^{m\circ \alpha} y(t)=t^m \mathbb{E}_{\mathfrak{U}_m}[\mathfrak{U}_m^{m-1} \mathbb{E}_{\mathfrak{U}_{m-1}}[\mathfrak{U}_{m-1}^{m-2}\cdots \mathbb{E}_{\mathfrak{U}_2}[\mathfrak{U}_2 \mathbb{E}_{\mathfrak{U}_1}[y(t\mathfrak{U}_1\cdots \mathfrak{U}_m)]]\cdots ]], \label{super_famm}
\end{equation}
where $\mathfrak{U}_1,\mathfrak{U}_2,\ldots$ are $\mathrm{Beta}(2-\alpha,\alpha)$-distributed and independent. Here, $\mathbb{E}_{\mathfrak{U}}[g(\mathfrak{U},\mathfrak{V})]=\mathbb{E}[g(\mathfrak{U},\mathfrak{V})|\mathfrak{V}]$ denotes an expectation of $g(\mathfrak{U},\mathfrak{V})$ with respect to $\mathfrak{U}$, as if we were conditioning on the other random quantity $\mathfrak{V}$. We arrive at the following theorem, which highlights the difficulty when dealing with ${}^L\! J^{m\circ \alpha}$.

\begin{Theorem} \label{th_sol_comp_proba}
The solution of~\eqref{deixar} on $[0,T]$, with initial condition $x(0)=x_0$, is
\[ x(t)=\mathcal{E}_{\alpha}(\mathcal{A}t)x_0+\sum_{j=0}^\infty \mathcal{A}^j t^{j+1} \mathbb{E}_{\mathfrak{U}_{j+1}}[\mathfrak{U}_{j+1}^{j} \mathbb{E}_{\mathfrak{U}_{j}}[\mathfrak{U}_{j}^{j-1}\cdots \mathbb{E}_{\mathfrak{U}_2}[\mathfrak{U}_2 \mathbb{E}_{\mathfrak{U}_1}[\vartheta(t\mathfrak{U}_1\cdots \mathfrak{U}_{j+1})]]\cdots ]], \]
where $\mathfrak{U}_1,\mathfrak{U}_2,\ldots$ are $\mathrm{Beta}(2-\alpha,\alpha)$-distributed and independent. If $x$ and $\vartheta$ are absolutely continuous on $[0,T]$, then $x$ solves~\eqref{model2} almost everywhere on $[0,T]$. If $x$ and $\vartheta$ are given by power series on $[0,T]$, then~\eqref{imposs} solves~\eqref{model2} for every $t\in [0,T]$.
\end{Theorem}
\begin{proof}
See~\eqref{super_famm} and the previous development.
\end{proof}

The appearance of $\vartheta(t\mathfrak{U}_1\cdots \mathfrak{U}_{j+1})$ in Theorem~\ref{th_sol_comp_proba} makes us investigate what happens when $\vartheta$ is given by a power function. Indeed, in that case, the various expectations can  be separated.

\begin{Remark} 
The difference between the explicit form of~\eqref{molta_civada} and~\eqref{super_famm} has an effect on the theory of Taylor series and their residuals as well. In the Caputo case, the mean-value theorem is
\[ x(t)-x(0)=\frac{1}{\Gamma(\alpha)} {}^C\! D^\alpha x(\xi)\cdot t^\alpha, \]
where $\xi\in (0,t)$, $t>0$. For the L-fractional derivative,
\begin{align}
x(t)-x(0)= {} & {}^L\! J^\alpha \circ {}^L\! D^\alpha x(t) \label{probl1_USI} \\
= {} & \frac{1}{\Gamma(\alpha)\Gamma(2-\alpha)}\int_0^t (t-s)^{\alpha-1}s^{1-\alpha}\cdot {}^L\! D^\alpha x(s)\mathrm{d}s \label{probl2_USI} \\
= {} & \frac{{}^L\! D^\alpha x(\xi)}{\Gamma(\alpha)\Gamma(2-\alpha)}\int_0^t (t-s)^{\alpha-1}s^{1-\alpha}\mathrm{d}s \label{probl3_USI} \\
= {} & {}^L\! D^\alpha x(\xi)\cdot t. \label{probl4_USI}
\end{align}

In~\eqref{probl1_USI}, the analog of Barrow's rule~\eqref{probl1} is used. In~\eqref{probl2_USI}, definition~\eqref{convolL} is applied. The mean-value theorem gives~\eqref{probl3_USI}, by the continuity of ${}^L\! D^\alpha x$ when $x$ is smooth. Finally, Lemma~\ref{lema_beta} is utilized in the last equality~\eqref{probl4_USI}. Observe, as a  consequence, that
\[ {}^L\! D^\alpha x(0)=\frac{\mathrm{d}x}{\mathrm{d}t}(0)=x'(0)\in (-\infty,\infty), \]
in contrast to the Caputo derivative (see also the justification~\eqref{tre_acabanti}). Hence, locally, at the beginning of the dynamics around $t\approx0$, the system~\eqref{LEDO} is very similar to the ordinary differential equation analog, and the change with $\alpha$ is smoother than in the Caputo case.

The mean-value theorem may be seen as the residual of the zeroth-order Taylor series. When the order of the Taylor series is increased, the Caputo derivative has the residual
\[ x(t)=\sum_{i=0}^n x_i t^{\alpha i} + \frac{{}^C\! D^{(n+1)\circ \alpha} x(\xi)}{\Gamma((n+1)\alpha+1)}t^{(n+1)\alpha}, \]
where $t>0$ and ${}^C\! D^{(n+1)\circ \alpha}={}^C\! D^\alpha\circ \cdots\circ {}^C\! D^\alpha$ is the iterated derivative. This formula mimics the expression for the ordinary derivative ($\alpha=1$), and it is a consequence of~\eqref{molta_civada}. Unfortunately, for the L-fractional derivative, one cannot expect a similar expression for 
\[ x(t)-\sum_{i=0}^n x_i t^i={}^L\! J^{(n+1)\circ \alpha}\circ {}^L\! D^{(n+1)\circ \alpha} x(t), \]
because ${}^L\! J^{(n+1)\circ \alpha}$ is not given in closed form, as a convolution. See \cite{taylor_generi} (expression~(3.11)) for  details in the context of the Caputo fractional calculus. These observations conclude the remark.
\end{Remark}

We noticed that both operators ${}^L\! J^\alpha$ and ${}^L\! D^\alpha$ have a probabilistic interpretation in terms of the beta distribution. Does the new Mittag--Leffler-type function $\mathcal{E}_\alpha$ enjoy a connection with probability theory? If $a$ is a random variable and 
\vspace{6pt}
\[ \varphi_a(s)=\mathbb{E}[\mathrm{e}^{as}] \]
is its moment-generating function, it is known that~\cite{lege_meuu,lege_meuu2}:
\begin{equation} 
\begin{split}
\exists\,C,n_0>0,\,0\leq p<1: {} & \; \frac{\mathbb{E}[|a|^{n+1}]}{\mathbb{E}[|a|^n]}\leq Cn^p,\;\forall\,n\geq n_0\; \\
{} & \Rightarrow\;
\varphi_a(s)<\infty,\;\forall\,s\in\mathbb{R}\;\Leftrightarrow\; \lim_{n\rightarrow\infty} \frac{\|a\|_n}{n}=0; 
\end{split}
\label{partaill1} 
\end{equation}
\begin{equation} 
\begin{split}
\exists\,C,n_0>0,\,0\leq p\leq 1: {} & \; \frac{\mathbb{E}[|a|^{n+1}]}{\mathbb{E}[|a|^n]}\leq Cn^p,\;\forall\,n\geq n_0\; \\
{} & \Leftrightarrow\;\exists\,\delta>0:\; \varphi_a(s)<\infty,\;\forall\,s\in (-\delta,\delta)\;\Leftrightarrow\; \limsup_{n\rightarrow\infty} \frac{\|a\|_n}{n}<\infty, 
\end{split}
\label{partaill2} 
\end{equation}
where $\|a\|_n=\mathbb{E}[|a|^n]^{1/n}$ is the $n$-th norm. In~\eqref{partaill1}, the converse of the first implication is not true, as the Poisson distribution shows with its moments given by the Bell numbers. Let 
\[ \varphi_a^\alpha(s)=\mathbb{E}[\mathcal{E}_\alpha(as)], \]
$s\in\mathbb{R}$, be the L-fractional moment-generating function of $a$, of order $\alpha\in (0,1]$. This is an extension of the usual moment-generating function, which is retrieved for $\alpha=1$. We obtain a partial version of~\eqref{partaill1} and~\eqref{partaill2} for $\varphi_a^\alpha$, because we need to employ the ratio test of convergence instead of the Cauchy--Hadamard theorem.

\begin{Theorem}
The following implications hold:
\[ \exists\,C,n_0>0,\,0\leq p<1:\; \frac{\mathbb{E}[|a|^{n+1}]}{\mathbb{E}[|a|^n]}\leq Cn^{\alpha p},\;\forall\,n\geq n_0\;\Rightarrow\;
\varphi_a^\alpha(s)<\infty,\;\forall\,s\in\mathbb{R}; \]
\[ \exists\,C,n_0>0,\,0\leq p\leq 1:\; \frac{\mathbb{E}[|a|^{n+1}]}{\mathbb{E}[|a|^n]}\leq Cn^{\alpha p},\;\forall\,n\geq n_0\;\Rightarrow\;\exists\,\delta>0:\; \varphi_a^\alpha(s)<\infty,\;\forall\,s\in (-\delta,\delta). \]
\end{Theorem}
\begin{proof}
Considering our definition~\eqref{mlf2}, we aim to prove that
\begin{equation} \sum_{n=0}^\infty \frac{\mathbb{E}[|a|^n]|s|^n}{\Gamma(2-\alpha)^n \prod_{j=1}^n \frac{\Gamma(j+1)}{\Gamma(j+1-\alpha)}}<\infty. \label{should_baix1} \end{equation}

According to~\eqref{turmell}, the ratio test gives 
\begin{align}
\frac{\mathbb{E}[|a|^{n+1}]\Gamma(2-\alpha)^n \prod_{j=1}^n \frac{\Gamma(j+1)}{\Gamma(j+1-\alpha)}}{\mathbb{E}[|a|^n]\Gamma(2-\alpha)^{n+1} \prod_{j=1}^{n+1} \frac{\Gamma(j+1)}{\Gamma(j+1-\alpha)}}\leq {} & Cn^{\alpha p}\frac{\Gamma(2-\alpha)^n \prod_{j=1}^n \frac{\Gamma(j+1)}{\Gamma(j+1-\alpha)}}{ \Gamma(2-\alpha)^{n+1} \prod_{j=1}^{n+1} \frac{\Gamma(j+1)}{\Gamma(j+1-\alpha)}} \nonumber \\
= {} & Cn^{\alpha p} \frac{1}{\Gamma(2-\alpha)} \frac{\Gamma(n+2-\alpha)}{\Gamma(n+2)} \nonumber \\
\sim {} & Cn^{\alpha p}\frac{1}{\Gamma(2-\alpha)} \frac{1}{(n+2-\alpha)^\alpha}. \label{should_baix22}
\end{align}

If $p<1$, then~\eqref{should_baix22} tends to $0$ and~\eqref{should_baix1} holds for $s\in\mathbb{R}$. If $p=1$, then~\eqref{should_baix22} converges to 
\[ \frac{C}{\Gamma(2-\alpha)}>0, \]
\vspace{-6pt}
so~\eqref{should_baix1} is satisfied for $s\in (-\delta,\delta)$, where
\[ \delta=\frac{\Gamma(2-\alpha)}{C}. \]
\end{proof}

If $a$ is a bounded random variable, then 
\[ \frac{\mathbb{E}[|a|^{n+1}]}{\mathbb{E}[|a|^n]}\leq C \]
and $\varphi_a^\alpha$ is finite on $\mathbb{R}$. If $a$ is a Gaussian random variable, then
\[ \frac{\mathbb{E}[|a|^{n+1}]}{\mathbb{E}[|a|^n]}\leq Cn^{1/2}, \]
so $\varphi_a^\alpha$ is finite on the real line for $\alpha>1/2$, and it is finite on a neighborhood of zero for $\alpha=1/2$. Since the gamma distribution satisfies 
\[ \frac{\mathbb{E}[|a|^{n+1}]}{\mathbb{E}[|a|^n]}\leq Cn, \]
one cannot work with $\varphi_a^\alpha$ for $\alpha<1$. Finally, the Weibull distribution with shape parameter $\beta$ has the ratio
\[ \frac{\mathbb{E}[|a|^{n+1}]}{\mathbb{E}[|a|^n]}\leq Cn^{1/\beta}, \]
therefore $\varphi_a^\alpha$ is finite on $\mathbb{R}$ for $\alpha>1/\beta$, and it is finite around zero for $\alpha=1/\beta$. For  information on these distributions, see~\cite{lege_meuu}.

It would be of certain relevance to investigate whether we can expect a better characterization for the finiteness of the fractional moment-generating function of random variables. One would likely need to use the Cauchy--Hadamard theorem, rather than the ratio test. Since the new Mittag--Leffler-type function is defined with products of gamma functions, the ratio test is the most straightforward tool to check the convergence of the series.

\subsection{Fractional Powers and Closed-Form Solutions} \label{subs_jardinir}

For an example of closed form of~\eqref{imposs}, let us consider
\begin{equation} \vartheta(t)=(\ell_1 t^{\delta_1},\ldots,\ell_d t^{\delta_d})^\top, \label{vart_spec} \end{equation}
where $\ell_1,\ldots,\ell_d\in\mathbb{C}$ and $\delta_1,\ldots,\delta_d\in (0,\infty)$. Equivalently,
\[ b(t)=(\kappa_1 t^{\mu_1},\ldots,\kappa_d t^{\mu_d})^\top, \]
where $\kappa_1,\ldots,\kappa_d\in\mathbb{C}$ and $\mu_1,\ldots,\mu_d>1-\alpha$ satisfy
\[ \ell_j=\Gamma(2-\alpha)\kappa_j,\quad \delta_j=\mu_j-1+\alpha; \]
see Section~\ref{sec_L} and, specifically, the link conditions~\eqref{relation_CL}. Here, $\top$ denotes the transpose of the vectors, for column form. The powers $\delta_j$ or $\mu_j$ are not necessarily integers; therefore, they are called fractional.

\begin{Lemma} \label{lema_agafa_clau} (Analogous to Corollary~\ref{cor_D_s})
If
\[ \sum_{n=0}^\infty |x_n|t^{n+1+\delta}<\infty \]
\vspace{-6pt}
for all $t\in [0,\epsilon]$, where $\epsilon>0$, $\delta>0$, and $x_n\in\mathbb{C}$, then
\[ {}^L\! D^\alpha \left( \sum_{n=0}^\infty x_n t^{n+1+\delta} \right)=\sum_{n=0}^\infty x_n \cdot {}^L\! D^\alpha (t^{n+1+\delta}) \]
on $[0,\epsilon]$. Furthermore,~\eqref{probl2} holds for all $t\in [0,\epsilon]$ for $\sum_{n=0}^\infty x_n t^{n+1+\delta}$.
\end{Lemma}
\begin{proof}
The proof is analogous to Corollary~\ref{cor_D_s} and its subsequent Remark~\ref{cor_D_s_remarK}. Conduct the steps in the proof of Corollary~\ref{cor_D_s}, adapted to this case, until~\eqref{estaa3}, which holds almost everywhere. To justify equality everywhere based on continuity at both sides, proceed as in Remark~\ref{cor_D_s_remarK}. Notice that
\[ \sum_{n=0}^\infty x_n t^{n+1+\delta}=x_0 t^{1+\delta} + x_1 t^{2+\delta} + \sum_{n=2}^\infty x_n t^{n+1+\delta}, \]
where
\[ \sum_{n=2}^\infty x_n t^{n+1+\delta}\in\mathcal{C}^3[0,T], \]
so the left-hand side of the corresponding Equation~\eqref{estaa3} is
\[ {}^L\! D^\alpha\left(\sum_{n=0}^\infty x_n t^{n+1+\delta}\right)\in\mathcal{C}[0,T];\]
see~\eqref{ja_ve_tr} and~\eqref{tre_acabanti}.
\end{proof}

\begin{Theorem} \label{socarr}
The solution of~\eqref{model2}, with source term~\eqref{vart_spec} and initial condition $x(0)=x_0$, is
\[ x(t)=\mathcal{E}_{\alpha}(\mathcal{A}t)x_0+\sum_{j=0}^\infty \mathcal{A}^j \nu_j(t), \]
where
\[ \nu_j(t)=\left( \ell_1 \frac{\prod_{i=2}^{j+2} \Gamma(i-\alpha+\delta_1)}{\Gamma(2-\alpha)^{j+1} \prod_{i=2}^{j+2} \Gamma(i+\delta_1)}t^{j+1+\delta_1},\ldots, \ell_d \frac{\prod_{i=2}^{j+2} \Gamma(i-\alpha+\delta_d)}{\Gamma(2-\alpha)^{j+1} \prod_{i=2}^{j+2} \Gamma(i+\delta_d)}t^{j+1+\delta_d}\right)^\top, \]
for every $t$ in $[0,T]$.
\end{Theorem}
\begin{proof}
In the general form~\eqref{imposs} from Theorem~\ref{th_sol_comp}, use Proposition~\ref{lemma_L_pow}. By Lemma~\ref{lema_agafa_clau}, we have a solution for all $t$ in $[0,T]$. (Without Lemma~\ref{lema_agafa_clau}, the conclusion would have been almost everywhere.)
\end{proof}

\begin{Theorem} \label{qpwomchfkdls}
The solution of~\eqref{model2}, with source term~\eqref{vart_spec} and initial condition $x(0)=x_0$,  dimension $d=1$, $\mathcal{A}=a\in\mathbb{C}$, $\ell_1=\ell$ and $\delta_1=\delta$, is
\[ x(t)=\mathcal{E}_{\alpha}(at)x_0+\ell\sum_{j=0}^\infty a^j \frac{\prod_{i=2}^{j+2} \Gamma(i-\alpha+\delta)}{\Gamma(2-\alpha)^{j+1} \prod_{i=2}^{j+2} \Gamma(i+\delta)}t^{j+1+\delta}, \]
for every $t$ in $[0,T]$.
\end{Theorem}
\begin{proof}
Apply Theorem~\ref{socarr} in the scalar case.
\end{proof}

For another example of a closed form of~\eqref{imposs}, now consider
\begin{equation} \vartheta(t)=\sum_{n=0}^\infty \vartheta_n t^n \label{vart_especc2} \end{equation}
on $[0,T]$, where $\vartheta_n\in\mathbb{C}^d$. That is, $\vartheta$ is real analytic at $t=0$ with values in $\mathbb{C}$. Equivalently,
\vspace{6pt}
\[ b(t)=\frac{t^{1-\alpha}}{\Gamma(2-\alpha)}\sum_{n=0}^\infty \vartheta_n t^{n}, \]
by Section~\ref{sec_L} and~\eqref{relation_CL}. In contrast to the previous case, the powers of $\vartheta$ are integer numbers. For $b$, the powers are fractional.

\begin{Theorem} \label{soccarr2}
The solution of~\eqref{model2}, with source term~\eqref{vart_especc2} and initial condition $x(0)=x_0$, is
\begin{equation} x(t)=\mathcal{E}_{\alpha}(\mathcal{A}t)x_0+\sum_{n=0}^\infty \sum_{k=0}^{n-1} \mathcal{A}^k \frac{\prod_{j=n-k}^n \Gamma(j-\alpha+1)}{\Gamma(2-\alpha)^{k+1}\prod_{j=n-k}^n \Gamma(j+1)}\vartheta_{n-k-1} t^n, \label{rapidet} \end{equation}
for every $t$ in $[0,T]$.
\end{Theorem}
\begin{proof}
For $j\geq0$, we perform the following computations:
\begin{align}
{}^L\! J^{(j+1)\circ\alpha} \vartheta(t)= {} & {}^L\! J^{(j+1)\circ\alpha}\left( \sum_{n=0}^\infty \vartheta_n t^n \right) \nonumber \\
= {} & \sum_{n=0}^\infty \vartheta_n \cdot {}^L\! J^{(j+1)\circ\alpha} t^n \label{mal_1} \\
= {} & \sum_{n=0}^\infty \vartheta_n \frac{\prod_{i=2}^{j+2} \Gamma(i-\alpha+n)}{\Gamma(2-\alpha)^{j+1}\prod_{i=2}^{j+2}\Gamma(i+n)} t^{j+1+n} \label{mal_2} \\
= {} & \sum_{l=j+1}^\infty \vartheta_{l-j-1}\frac{\prod_{i=2}^{j+2} \Gamma(i-\alpha+l-j-1)}{\Gamma(2-\alpha)^{j+1}\prod_{i=2}^{j+2}\Gamma(i+l-j-1)} t^{l} \label{mal_3} \\
= {} & \sum_{l=j+1}^\infty \vartheta_{l-j-1}\frac{\prod_{i=l-j}^{l} \Gamma(i-\alpha+1)}{\Gamma(2-\alpha)^{j+1}\prod_{i=l-j}^{l}\Gamma(i+1)} t^{l}. \label{mal_4}
\end{align}
In the equality from~\eqref{mal_1}, the continuity of ${}^L\! J^{(j+1)\circ\alpha}$ is used; see Proposition~\ref{lema_coni}. In the equality from~\eqref{mal_2}, the formula~\eqref{Jtd} in Proposition~\ref{lemma_L_pow} is employed with $m=j+1$ and $\delta=n>\alpha-2$. In the equality from~\eqref{mal_3}, the change in variable $n+j+1=l$ is applied. The equality from~\eqref{mal_4} follows by expanding the product.

From Theorem~\ref{th_sol_comp},~\eqref{imposs} and~\eqref{mal_4},
\begin{align*}
 x(t)= {} & \mathcal{E}_{\alpha}(\mathcal{A}t)x_0+\sum_{j=0}^\infty \mathcal{A}^j \cdot {}^L\! J^{(j+1)\circ \alpha}\vartheta(t) \\
= {} & \mathcal{E}_{\alpha}(\mathcal{A}t)x_0+\sum_{j=0}^\infty \mathcal{A}^j \sum_{l=j+1}^\infty \vartheta_{l-j-1}\frac{\prod_{i=l-j}^{l} \Gamma(i-\alpha+1)}{\Gamma(2-\alpha)^{j+1}\prod_{i=l-j}^{l}\Gamma(i+1)} t^{l} \\
= {} & \mathcal{E}_{\alpha}(\mathcal{A}t)x_0+ \sum_{l=0}^\infty \sum_{j=0}^{l-1} \mathcal{A}^j\vartheta_{l-j-1}\frac{\prod_{i=l-j}^{l} \Gamma(i-\alpha+1)}{\Gamma(2-\alpha)^{j+1}\prod_{i=l-j}^{l}\Gamma(i+1)} t^{l},
\end{align*}
which corresponds to~\eqref{rapidet}. We finally note that $x$ and the corresponding $\vartheta$ are analytic; hence, we have a solution for every $t\in [0,T]$ and not just almost everywhere;  see Theorem~\ref{th_sol_comp}.
\end{proof}

\begin{Remark} \label{rmk_villi_rel} 
In the Caputo case, we have the fact that~\eqref{forma_pw_compl} solves~\eqref{ode_cap_lineal_compl} almost everywhere on $[0,T]$ if $x$ and $b$ are absolutely continuous on $[0,T]$. If $b$ is given by a fractional power series on $[0,T]$ (in terms of $t^{\alpha n}$), then~\eqref{forma_pw_compl} is the solution of~\eqref{ode_cap_lineal_compl} everywhere on $[0,T]$. Otherwise, we only know that~\eqref{forma_pw_compl} solves the corresponding Volterra integral problem associated with~\eqref{ode_cap_lineal_compl}, $x(t)=x_0+{}^C\! J^\alpha (Ax+b)(t)$, for all $t\in [0,T]$. All these assertions are a consequence of Lemma~\ref{lema_rigor_FC_Cap}. Thus, one should be careful when proposing solutions to fractional differential equations; imprecise statements may give rise to solutions of the integral problem or almost-everywhere solutions. If $b$ only belongs to $\mathcal{C}[0,T]$, then~\eqref{forma_pw_compl} solves the modified Caputo equation ${}^C\! D_{\ast}^\alpha x(t)=Ax(t)+b(t)$ for every $t\in [0,T]$, where ${}^C\! D_{\ast}^\alpha$ is defined by~\eqref{modified_cpitt}; see \cite{weber_pap} (Lemma~4.5). If $x$ in~\eqref{forma_pw_compl} is absolutely continuous on $[0,T]$, then ${}^C\! D_{\ast}^\alpha x(t)={}^C\! D^\alpha x(t)$ almost everywhere \cite{weber_pap} (Lemma~4.12), and~\eqref{ode_cap_lineal_compl} holds almost everywhere on $[0,T]$.
\end{Remark}

\subsection{On Uniqueness}

We notice that all of the obtained solutions are unique. For a general L-fractional equation~\eqref{LEDO}, where the input function $f$ can be nonlinear, uniqueness holds if $f$ is Lipschitz-continuous with respect to the second component on every compact set, independently of the size of the Lipschitz constant. Indeed, if there are two solutions ${}_1 x(t)$ and ${}_2 x(t)$ of~\eqref{LEDO} with ${}_1 x(0)={}_2 x(0)=x_0$, then
\begin{equation}
\begin{split}
 |{}_1 x(t) - {}_2 x(t)|= {} & |{}^L\! J^\alpha f(t,{}_1 x(t)) - {}^L\! J^\alpha f(t,{}_2 x(t))| \\
\leq {} & \frac{1}{\Gamma(\alpha)\Gamma(2-\alpha)}\int_0^t (t-s)^{\alpha-1}s^{1-\alpha}|f(s,{}_1 x(s)) - f(s,{}_2 x(s))|\mathrm{d}s, 
\end{split}
\label{dion} 
\end{equation}
by~\eqref{probl1}. Since ${}_1x([0,T])$ and ${}_2x([0,T])$ are bounded, there exists a constant $M>0$ such that
\[ |f(s,{}_1 x(s))-f(s,{}_2 x(s))|\leq M|{}_1 x(s) - {}_2 x(s)|, \]
for all $s\in [0,T]$. As  a  consequence, from~\eqref{dion},
\[ |{}_1 x(t) - {}_2 x(t)|\leq \frac{T^{1-\alpha} M}{\Gamma(\alpha)\Gamma(2-\alpha)}\int_0^t (t-s)^{\alpha-1}|{}_1 x(s) - {}_2 x(s)|\mathrm{d}s. \]

By Gronwall's inequality with singularity~\cite{gronw2}, one concludes that $|{}_1x(t)-{}_2x(t)|=0$ and ${}_1x={}_2x$ on $[0,T]$, as wanted.

The precise statement that has been proved is the following:
\begin{Proposition} \label{prpo_unic}
Given an L-fractional differential equation~\eqref{LEDO}, if $f$ is Lipschitz-continuous with respect to the second component on every compact set (i.e., for every $R>0$, there exists $M>0$ such that $|f(t,{}_1 x)-f(t,{}_2 x)| \leq M|{}_1 x- {}_2 x|$ for all $|t|\leq R$, $|{}_1 x|\leq R$ and $|{}_2 x|\leq R$), then~\eqref{LEDO} has a unique solution for any initial condition $(0,x_0)$ (in the set of absolutely continuous functions).
\end{Proposition}

We observe that Proposition~\ref{prpo_unic} may be proved without relying on Gronwall's inequality with singularity. This is a nice fact because proofs of uniqueness often use Gronwall's lemmas. If $z={}_1 x- {}_2 x$ on $[0,T]$, then
\begin{equation} |z(t)|\leq M\cdot {}^L\! J^\alpha|z(t)| \label{macspr} \end{equation}
for every $t\in [0,T]$, by~\eqref{dion}. If we iterate~\eqref{macspr} $m$ times,
\begin{equation} |z(t)|\leq M^m \cdot {}^L\! J^{m\circ \alpha}|z(t)|. \label{macspr22} \end{equation}

By Proposition~\ref{em_fit_inf},~\eqref{macspr22} continues with
\begin{equation} |z(t)|\leq M^m \cdot {}^L\! J^{m\circ \alpha}|z(t)|\leq M^m \|z\|_\infty \frac{\prod_{i=2}^{m+1} \Gamma(i-\alpha)}{\Gamma(2-\alpha)^m \prod_{i=2}^{m+1}\Gamma(i)}T^m. \label{macspr33} \end{equation}

As $m\rightarrow\infty$, the right-hand side of the inequality~\eqref{macspr33} tends to $0$, because
\[ \sum_{m=1}^\infty M^m \|z\|_\infty \frac{\prod_{i=2}^{m+1} \Gamma(i-\alpha)}{\Gamma(2-\alpha)^m \prod_{i=2}^{m+1}\Gamma(i)}T^m<\infty \]
by the ratio test (see~\eqref{turmell}, for instance). Hence $z(t)=0$, and we are finished.

In spite of this, I am not aware of a proof that does not draw  on the integral operator ${}^L\! J^\alpha$ (or ${}^C\! J^\alpha$). Let us consider the case of dimension $d=1$. If $z={}_1 x- {}_2 x$ were non-zero at some point on $(0,T]$, then we could define 
\[ t^\ast=\max\{t\in [0,T]:\,z([0,t])=\{0\}\}<T. \]

For some $\delta>0$ such that $z(t)\neq0$ on $(t^\ast,t^\ast+\delta)$, we would have
\[ {}^L\! D^\alpha z(t)=\frac{f(s,{}_1 x(s))-f(s,{}_2 x(s))}{z(t)}z(t)=a(t)z(t) \]
on $(t^\ast,t^\ast+\delta)$. By extending $a$ to $[0,t^\ast]$ with the zero value, the equation 
\[ {}^L\! D^\alpha z(t)=a(t)z(t) \]
would hold on $[0,t^\ast+\delta)$. That is, the initial problem is converted into a linear equation. The function $a$ is bounded by $M$, by the Lipschitz condition of $f$; therefore, it is Lebesgue-integrable. In the case $\alpha=1$, one defines 
\[ \tilde{a}(t)=\int_0^t a(\tau)\mathrm{d}\tau. \]

By using the integrating-factor method, 
\[ \mathrm{e}^{-\tilde{a}(t)}z'(t)=\mathrm{e}^{-\tilde{a}(t)}a(t)z(t), \]
i.e., 
\[ (\mathrm{e}^{-\tilde{a}}z)'(t)=0 \]
almost everywhere. Hence $\mathrm{e}^{-\tilde{a}(t)}z(t)=z(0)=0$ and $z(t)=0$ on $[0,t^\ast+\delta)$. For $\alpha<1$, one cannot use the same reasoning, because the product rule is not satisfied.

\section{Sequential Linear Equations with Constant Coefficients: Context and Solution} \label{sec_sequ}

The aim of this section is to investigate autonomous linear L-fractional differential equations of the sequential type. The term ``sequential'' comes from the fact that higher-order derivatives are defined by composition, in a sequential manner, while maintaining the original index $\alpha$ in $(0,1)$. We define these equations and highlight some of the issues and problems that arise. We then proceed to find solutions, by exploiting the vector-space structure. We first elaborate on the case of sequential order two, which gives the necessary intuition to tackle the general case. The novel Mittag--Leffler-type function plays an essential role and gives rise to a new view of how the exponential function works in the setting of linear ordinary differential equations. Most of the development is concerned with the homogeneous model. Eventually, some forcing terms are possible, by extending the method of undetermined coefficients. Several examples illustrate the theory.

\subsection{Definitions and Problems} \label{subsec_defi_pro}

Considering the composition of operators~\eqref{notation_circ}, and in analogy to ordinary differential equations, a sequential L-fractional differential equation of order $m\geq1$ and dimension $d\geq1$ is
\begin{equation} {}^L\! D^{m\circ \alpha} x(t)=f(t,x(t),{}^L\! D^{\alpha}x(t), {}^L\! D^{2\circ \alpha} x(t),\ldots,{}^L\! D^{(m-1)\circ \alpha}x(t)), \label{seq_probl_1} \end{equation}
where $f:[0,T]\times \Omega\subseteq [0,T]\times \mathbb{R}^{dm}\rightarrow \mathbb{R}^{dm}$, or $f:[0,T]\times \Omega\subseteq [0,T]\times \mathbb{C}^{dm}\rightarrow \mathbb{C}^{dm}$, is a continuous function. The initial data to be met are
\begin{equation} x(0)=x_0,\;\; {}^L\! D^{\alpha}x(0)=x_{0,1},\;\; {}^L\! D^{2\circ \alpha} x(0)=x_{0,2},\;\; \ldots,\;\; {}^L\! D^{(m-1)\circ \alpha}x(0)=x_{0,m-1}, \label{sec_probl_ic} \end{equation}
where $x_0,x_{0,1},x_{0,2},\ldots,x_{0,m-1}\in\mathbb{C}$. Model~\eqref{seq_probl_1} with~\eqref{sec_probl_ic} generalizes, in principle,~\eqref{LEDO}, since $m$ can be greater than $1$. However, as occurs with the ordinary case $\alpha=1$, one may see that~\eqref{seq_probl_1} and~\eqref{LEDO} are equivalent.

\begin{Proposition} \label{propo_canoi}
Equations~\eqref{seq_probl_1} and~\eqref{LEDO} are equivalent.
\end{Proposition}
\begin{proof}
If $x$ satisfies~\eqref{seq_probl_1}, then 
\[ \tilde{x}(t)=(x(t),{}^L\! D^{\alpha}x(t), {}^L\! D^{2\circ \alpha} x(t),\ldots,{}^L\! D^{(m-1)\circ \alpha}x(t)) \]
solves
\begin{align*}
 {}^L\! D^{\alpha}\tilde{x}(t)= {} & ({}^L\! D^{\alpha} x(t),{}^L\! D^{2\circ \alpha}x(t), {}^L\! D^{3\circ \alpha} x(t),\ldots,{}^L\! D^{m\circ \alpha}x(t)) \\
= {} & (\tilde{x}_2(t),\tilde{x}_3(t),\ldots,\tilde{x}_m(t),f(t,\tilde{x}(t))) \\
= {} & \tilde{f}(t,\tilde{x}(t)),
\end{align*}
which is a first-order system of dimension $dm$. The initial condition is
\[ \tilde{x}(0)=(x_0,x_{0,1},\ldots,x_{0,m-1}). \]
\end{proof}

Although this proposition reduces~\eqref{seq_probl_1} to~\eqref{LEDO}, the dimension of the associated system~\eqref{LEDO} is greater, of size $dm$. Hence, in some situations, this procedure may not be convenient to derive explicit or closed-form solutions for~\eqref{seq_probl_1}.

A sequential linear L-fractional differential equation of order $m\geq1$ and dimension $d=1$ is
\begin{equation}
 {}^L\! D^{m\circ \alpha} x(t)+a_{m-1} {}^L\! D^{(m-1)\circ \alpha} x(t)+\ldots+a_1 {}^L\! D^{\alpha} x(t)+a_0 x(t)=0. 
 \label{linear_seq_1}
\end{equation}

The coefficients $a_0,\ldots,a_{m-1}$ are scalars in $\mathbb{C}$ and $x:[0,T]\rightarrow\mathbb{C}$. The initial condition to be met is~\eqref{sec_probl_ic}. Note that~\eqref{linear_seq_1} is scalar, homogeneous, and autonomous.

By Proposition~\ref{propo_canoi},~\eqref{linear_seq_1} can be reduced to a linear system of the form~\eqref{adapted_se}, with matrix $\mathcal{A}\in\mathbb{C}^{m\times m}$ and solution~\eqref{is_connn}: 
\begin{equation}
{}^L\! D^{\alpha} \hat{x}=\mathcal{A}\hat{x},\;\;\hat{x}=\begin{pmatrix} x \\ {}^L\! D^{\alpha} x \\ \vdots \\ {}^L\! D^{(m-1)\circ \alpha} x \end{pmatrix},\;\; \mathcal{A}=\begin{pmatrix} 0 & 1 & 0 & \cdots & 0 \\ 0 & 0 & 1 & \cdots & 0 \\ \vdots & \vdots & \vdots & \ddots & \vdots \\ -a_0 & -a_1 & -a_2 & \cdots & -a_{m-1} \end{pmatrix}.
 \label{redu_lin_siss}
\end{equation}

Let $\mathcal{S}$ be the set of solutions of~\eqref{linear_seq_1}, equivalently~\eqref{redu_lin_siss}, without specifying initial conditions. By Theorem~\ref{soccarr2}, it is clear that $\mathcal{S}\subseteq\mathcal{C}^\omega$. In the following proposition, we examine the algebraic structure of $\mathcal{S}$.

\begin{Proposition} \label{vec_sp_SS}
The set $\mathcal{S}$ is a vector space over $\mathbb{C}$, of dimension $m$. 
\end{Proposition}
\begin{proof}
Since ${}^L\! D^{\alpha}$ is a linear operator, it is obvious that $\mathcal{S}$ satisfies the properties of a vector space. Another proof relies on the fact that $\mathcal{S}=\mathrm{Ker}\Lambda$, where
\begin{equation} \Lambda:\,\mathcal{C}^\omega\rightarrow\mathcal{C}^\omega, \label{xe_lambda} \end{equation}
\begin{equation} \Lambda={}^L\! D^{m\circ \alpha} +a_{m-1} {}^L\! D^{(m-1)\circ \alpha} +\ldots+a_1 {}^L\! D^{\alpha} +a_0. \label{xe_lambda2} \end{equation}

The fact that $\dim\mathcal{S}\leq m$ follows from the injectivity of the linear map
\begin{equation} \Xi:\,\mathcal{S}\rightarrow\mathbb{C}^m, \label{xe_xi} \end{equation}
\begin{equation} \Xi x=(x(0),{}^L\! D^{\alpha} x(0),{}^L\! D^{2\circ \alpha}x(0),\ldots,{}^L\! D^{(m-1)\circ \alpha}x(0)). \label{xe_xi2} \end{equation}
Indeed, since~\eqref{linear_seq_1} can be expressed as a first-order system~\eqref{redu_lin_siss} by Proposition~\ref{propo_canoi}, and uniqueness for these models is known---see Proposition~\ref{prpo_unic}---we then have that initial conditions in $\mathbb{C}^m$ give rise to at most one solution in $\mathcal{S}$.

The surjectivity of \eqref{xe_xi}--\eqref{xe_xi2}, which is equivalent to the existence of a solution for any initial-value problem~\eqref{linear_seq_1} with~\eqref{sec_probl_ic}, is true by Proposition~\ref{propo_canoi} (transformation to first-order system~\eqref{redu_lin_siss}) and Theorem~\ref{th_sol_comp}. It implies that $\dim\mathcal{S}\geq m$. This completes the proof.
\end{proof}

Consider the polynomial
\begin{equation} p(\lambda)=\lambda^m+a_{m-1}\lambda^{m-1}+\ldots+a_1\lambda+a_0, \label{assochi_poly} \end{equation}
which is the characteristic polynomial of the matrix $\mathcal{A}\in\mathbb{C}^{m\times m}$ associated with the corresponding first-order linear system~\eqref{redu_lin_siss}. By the fundamental theorem of algebra,
\[ p(\lambda)=(\lambda-\lambda_1)^{n_1}\cdots (\lambda-\lambda_r)^{n_r}, \]
where $\lambda_1,\ldots,\lambda_r\in\mathbb{C}$ are the distinct roots of $p$ (eigenvalues of $\mathcal{A}$) with multiplicities $n_1,\ldots,n_r\geq1$, respectively, and $n_1+\ldots+n_r=m$. 

To solve~\eqref{linear_seq_1}, we express~\eqref{linear_seq_1} as a sequential model of scalar first-order linear equations of the form~\eqref{simplest_L}. We rely on solving scalar linear problems iteratively, entirely based on power-series calculations, with no need for matrix variables. It is important to emphasize that, since we deal with power series, computations hold for every $t$, and not only almost everywhere; see Theorem~\ref{th_sol_comp} and Theorem~\ref{soccarr2}. Equation~\eqref{linear_seq_1} is rewritten as
\begin{equation}
 ({}^L\! D^{\alpha}-\lambda_1)^{n_1}\circ \cdots \circ ({}^L\! D^{\alpha}-\lambda_r)^{n_r}x=0.
 \label{desc_D_lini}
\end{equation}

All those factors commute. To find $x$, one needs to consider, in order, 
\begin{equation} ({}^L\! D^{\alpha}-\lambda_1)y_{1,\lambda_1}=0,\;\; ({}^L\! D^{\alpha}-\lambda_1)y_{2,\lambda_1}=y_{1,\lambda_1},\;\; \ldots,\;\; ({}^L\! D^{\alpha}-\lambda_1)y_{n_1,\lambda_1}=y_{n_1-1,\lambda_1}, \label{claud1} \end{equation}
\begin{equation} ({}^L\! D^{\alpha}-\lambda_2)y_{1,\lambda_2}=y_{n_1-1,\lambda_1},\;\; ({}^L\! D^{\alpha}-\lambda_2)y_{2,\lambda_2}=y_{1,\lambda_2},\;\; \ldots,\;\; ({}^L\! D^{\alpha}-\lambda_2)y_{n_2,\lambda_2}=y_{n_2-1,\lambda_2}, \label{claud2} \end{equation}
\begin{equation} \ldots, \label{claud3} \end{equation}
\begin{equation} ({}^L\! D^{\alpha}-\lambda_r)y_{1,\lambda_r}=y_{n_{r-1}-1,\lambda_{r-1}},\;\; ({}^L\! D^{\alpha}-\lambda_r)y_{2,\lambda_r}=y_{1,\lambda_r},\;\; \ldots,\;\; ({}^L\! D^{\alpha}-\lambda_r)y_{n_r,\lambda_r}=y_{n_r-1,\lambda_r}, \label{claud4} \end{equation}
\begin{equation} x=y_{n_r,\lambda_r}. \label{claud5} \end{equation}

In the following parts, we investigate how to solve the sequential problems \eqref{claud1}--\eqref{claud5}. We first address the order $m=2$ and then generalize to any $m$. Besides the former case being easier, it permits establishing the methodology and deducing how the general solution should be.

Actually, we will only need the upper bound $\dim\mathcal{S}\leq m$, which holds by uniqueness (Proposition~\ref{prpo_unic}). Note that $\dim\mathcal{S}\leq m$ can be justified alternatively, based on the sequential decomposition~\eqref{desc_D_lini}, by
\vspace{-6pt}
\begin{equation} \dim\mathcal{S}=\dim\mathrm{Ker}\Lambda\leq \sum_{j=1}^m n_j\cdot \dim \underbrace{\mathrm{Ker} ({}^L\! D^{\alpha}-\lambda_j)}_{\langle \mathcal{E}_\alpha (\lambda_j t)\rangle}=\sum_{j=1}^m n_j\cdot 1=m, \label{seqke} \end{equation}
where $\Lambda$ was defined in~\eqref{xe_lambda},~\eqref{xe_lambda2}. The first inequality in~\eqref{seqke} comes from the fact that, if $g_1,g_2:V\rightarrow V$ are two linear maps, then $\mathrm{Ker}(g_1\circ g_2)=g_2^{-1}(\mathrm{Ker} g_1)$ and $G:\mathrm{Ker}(g_1\circ g_2)\rightarrow\mathrm{Ker} g_1$, $v\mapsto g_2(v)$, is well defined and linear with $\mathrm{Ker}G=\mathrm{Ker}g_2$, so that $\mathrm{Ker}(g_1\circ g_2)/\mathrm{Ker}g_2\cong \mathrm{Im}G\leq \mathrm{Ker} g_1$ by the first isomorphism theorem and $\dim \mathrm{Ker}(g_1\circ g_2)\leq \dim \mathrm{Ker} g_1 + \dim \mathrm{Ker} g_2$ holds.

\subsection{Solution for Sequential Order Two}

When $m=2$, Equation~\eqref{linear_seq_1} is
\begin{equation}
 ({}^L\! D^{2\circ \alpha} + a_1 {}^L\! D^{\alpha} + a_0)x=0,
 \label{linear_seq_m2} 
\end{equation}
where $a_0,a_1\in\mathbb{C}$. The associated polynomial $p$ in~\eqref{assochi_poly}, 
\[ p(\lambda)=\lambda^2+a_1\lambda+a_0, \]
has two roots, $\lambda_1$ and $\lambda_2$. Let $\mathcal{S}$ be the complete set of solutions of~\eqref{linear_seq_m2}. From~\eqref{sec_probl_ic}, and consider the initial states
\begin{equation} x(0)=x_0,\;\; {}^L\! D^{\alpha}x(0)=x_{0,1}. \label{ic_dim_m2} \end{equation}

With this notation, the following theorem gives the solution to~\eqref{linear_seq_m2}.

\begin{Theorem} \label{th_increible}
If the roots of the associated polynomial, $\lambda_1$ and $\lambda_2$, are distinct in $\mathbb{R}$ or in $\mathbb{C}$, then
\begin{equation} x(t)=\frac{x_{0,1}-\lambda_2 x_0}{\lambda_1-\lambda_2}\mathcal{E}_\alpha(\lambda_1 t)+\frac{\lambda_1 x_0-x_{0,1}}{\lambda_1-\lambda_2}\mathcal{E}_\alpha(\lambda_2 t) \label{masa_1} \end{equation}
is the solution of~\eqref{linear_seq_m2} with initial conditions~\eqref{ic_dim_m2}, on $[0,\infty)$. The set
\begin{equation} \{ \mathcal{E}_\alpha(\lambda_1 t), \mathcal{E}_\alpha(\lambda_2 t) \} \label{baexi1} \end{equation}
is a basis of $\mathcal{S}$.

On the contrary, if $\lambda_1=\lambda_2=\lambda$ (in $\mathbb{R}$), then
\begin{equation} x(t)=x_0\mathcal{E}_\alpha(\lambda t) + (x_{0,1}-\lambda x_0)t \mathcal{E}_\alpha'(\lambda t) \label{masa_2} \end{equation}
is the solution of~\eqref{linear_seq_m2} with initial conditions~\eqref{ic_dim_m2}, on $[0,\infty)$. The set
\begin{equation} \{ \mathcal{E}_\alpha(\lambda t), t\mathcal{E}_\alpha'(\lambda t) \} \label{baexi2} \end{equation}
is a basis of $\mathcal{S}$.

Here, $\mathcal{E}_\alpha$ is the new Mittag--Leffler-type function~\eqref{mlf2} and $\mathcal{E}_\alpha'$ denotes its ordinary derivative.
\end{Theorem}
\begin{proof} 
Consider the roots $\lambda_1$ and $\lambda_2$, irrespective of whether these are different or not. Problem~\eqref{linear_seq_m2} decomposes as
\[ ({}^L\! D^{\alpha}-\lambda_1)\circ ({}^L\! D^{\alpha}-\lambda_2)x=0; \]
see~\eqref{desc_D_lini} with $m=n_1+n_2=2$, $n_1,n_2\in\{1,2\}$. 

First, we solve 
\[ ({}^L\! D^{\alpha}-\lambda_1)y=0, \]
which gives
\begin{equation} y=\mathcal{E}_{\alpha}(\lambda_1 t)c_1, \label{rosi1} \end{equation}
where $c_1\in\mathbb{C}$ and $t\in [0,\infty)$. See, for example, the general result of Theorem~\ref{th_sol_comp}. Since
\begin{equation} ({}^L\! D^{\alpha}-\lambda_2)x=y, \label{rosi2} \end{equation}
the constant $c_1$ is
\begin{equation} c_1=y(0)={}^L\! D^{\alpha} x(0)-\lambda_2 x(0)=x_{0,1}-\lambda_2 x_0. \label{rosi3}\end{equation}

Second, from~\eqref{rosi1} and~\eqref{rosi2}, we solve
\begin{equation} ({}^L\! D^{\alpha}-\lambda_2)x(t)=\vartheta(t)=\mathcal{E}_{\alpha}(\lambda_1 t)c_1. \label{rosit5} \end{equation}

We need to use Theorem~\ref{soccarr2}, which deals with power-series source terms. In this case,
\[ \vartheta_n=c_1 \lambda_1^n \frac{\prod_{j=1}^n \Gamma(j+1-\alpha)}{\Gamma(2-\alpha)^n \prod_{j=1}^n \Gamma(j+1)}, \]
considering the expansion's coefficients of the new Mittag--Leffler function~\eqref{mlf2}. Therefore, by Theorem~\ref{soccarr2}, the solution of~\eqref{rosit5} is
\begin{align}
x(t)= {} & \mathcal{E}_{\alpha}(\lambda_2 t)x_0 + \sum_{n=0}^\infty \sum_{k=0}^{n-1} \lambda_2^k \frac{\prod_{j=n-k}^n \Gamma(j-\alpha+1)}{\Gamma(2-\alpha)^{k+1} \prod_{j=n-k}^n \Gamma(j+1)}c_1 \lambda_1^{n-k-1} \nonumber \\
{} & \quad \times \frac{\prod_{j=1}^{n-k-1} \Gamma(j+1-\alpha)}{\Gamma(2-\alpha)^{n-k-1}\prod_{j=1}^{n-k-1}\Gamma(j+1)}t^n \nonumber \\
= {} & \mathcal{E}_{\alpha}(\lambda_2 t)x_0 + (x_{0,1}-\lambda_2 x_0)\sum_{n=0}^\infty \frac{\prod_{j=1}^n \Gamma(j+1-\alpha)}{\Gamma(2-\alpha)^n \prod_{j=1}^n \Gamma(j+1)}t^n \sum_{k=0}^{n-1} \lambda_1^{n-k-1}\lambda_2^k \label{esta_prop}
\end{align}
where the constant~\eqref{rosi3} is substituted into~\eqref{esta_prop}. To deal with the sum
\[ \sum_{k=0}^{n-1} \lambda_1^{n-k-1}\lambda_2^k, \]
we distinguish between $\lambda_1\neq\lambda_2$ and $\lambda_1=\lambda_2=\lambda$. In the former case,
\[ \sum_{k=0}^{n-1} \lambda_1^{n-k-1}\lambda_2^k=\frac{\lambda_1^n-\lambda_2^n}{\lambda_1-\lambda_2} \]
and, from~\eqref{esta_prop},
\begin{align*}
x(t)= {} & \mathcal{E}_\alpha(\lambda_2 t)x_0 + \frac{x_{0,1}-\lambda_2 x_0}{\lambda_1-\lambda_2}\left[ \mathcal{E}_\alpha(\lambda_1 t)-\mathcal{E}_\alpha(\lambda_2 t) \right] \\
= {} & \frac{x_{0,1}-\lambda_2 x_0}{\lambda_1-\lambda_2}\mathcal{E}_\alpha(\lambda_1 t)+\frac{\lambda_1 x_0-x_{0,1}}{\lambda_1-\lambda_2}\mathcal{E}_\alpha(\lambda_2 t),
\end{align*}
which is~\eqref{masa_1}. In the latter case,
\[ \sum_{k=0}^{n-1} \lambda_1^{n-k-1}\lambda_2^k=n\lambda^{n-1}, \]
and~\eqref{masa_2} is obtained.

We finally need to justify that~\eqref{baexi1} and~\eqref{baexi2} are bases of $\mathcal{S}$, when $\lambda_1\neq\lambda_2$ and $\lambda_1=\lambda_2=\lambda$, respectively. Since they generate $\mathcal{S}$, we need to prove linear independence. (Analogously, since $\dim\mathcal{S}\leq 2$ by Proposition~\ref{vec_sp_SS} or~\eqref{seqke}, the linear independence of the two elements suffices.)

For~\eqref{baexi1}, consider a linear combination
\[ \beta_1 \mathcal{E}_\alpha(\lambda_1 t)+\beta_2 \mathcal{E}_\alpha(\lambda_2 t)=0 \]
for all $t\in\mathbb{R}$, where $\beta_1,\beta_2\in\mathbb{C}$. Then,
\[ \beta_1 {}^L\! D^{\alpha} \mathcal{E}_\alpha(\lambda_1 t)+\beta_2 {}^L\! D^{\alpha} \mathcal{E}_\alpha(\lambda_2 t)=0. \]

Since
\[ \det\begin{pmatrix} \mathcal{E}_\alpha(\lambda_1 t) & \mathcal{E}_\alpha(\lambda_2 t) \\ {}^L\! D^{\alpha} \mathcal{E}_\alpha(\lambda_1 t) & {}^L\! D^{\alpha} \mathcal{E}_\alpha(\lambda_2 t) \end{pmatrix}= \det\begin{pmatrix} \mathcal{E}_\alpha(\lambda_1 t) & \mathcal{E}_\alpha(\lambda_2 t) \\ \lambda_1 \mathcal{E}_\alpha(\lambda_1 t) & \lambda_2 \mathcal{E}_\alpha(\lambda_2 t) \end{pmatrix}=(\lambda_2-\lambda_1) \mathcal{E}_\alpha(\lambda_1 t) \mathcal{E}_\alpha(\lambda_2 t) \]
gives $\lambda_2-\lambda_1\neq 0$ at $t=0$, we conclude that $\beta_1=\beta_2=0$ and that linear independence of~\eqref{baexi1} holds.

Analogously, for~\eqref{baexi2}, consider a linear combination
\[ \beta_1 \mathcal{E}_\alpha(\lambda t)+\beta_2 t \mathcal{E}_\alpha'(\lambda t)=0 \]
for all $t\in\mathbb{R}$, where $\beta_1,\beta_2\in\mathbb{C}$. Then,
\[ \beta_1 {}^L\! D^{\alpha} \mathcal{E}_\alpha(\lambda t)+\beta_2 {}^L\! D^{\alpha} (t\mathcal{E}_\alpha'(\lambda t))=0. \]

Simple computations yield
\begin{equation} {}^L\! D^{\alpha} (t\mathcal{E}_\alpha'(\lambda t))=\sum_{n=1}^\infty \frac{n\lambda^{n-1}}{\Gamma(2-\alpha)^n \prod_{j=1}^n \frac{\Gamma(j+1)}{\Gamma(j+1-\alpha)}} \frac{\Gamma(n+1)\Gamma(2-\alpha)}{\Gamma(n+1-\alpha)}t^{n-1}. 
 \label{qpeowidjsn}
\end{equation}

Since
\[ \det\left.\begin{pmatrix} \mathcal{E}_\alpha(\lambda t) & t\mathcal{E}_\alpha'(\lambda t) \\ {}^L\! D^{\alpha} \mathcal{E}_\alpha(\lambda t) & {}^L\! D^{\alpha} (t\mathcal{E}_\alpha'(\lambda t)) \end{pmatrix}\right|_{t=0}= \det\begin{pmatrix} 1 & 0 \\ \lambda & 1 \end{pmatrix}=1\neq 0, \]
it follows $\beta_1=\beta_2=0$ and the linear independence of~\eqref{baexi2}.
\end{proof}

Later, in Section~\ref{sec_sequ_AA}, we will address~\eqref{linear_seq_m2} when the coefficients are analytic functions. The solution will be a power series, with coefficients defined recursively.

\subsection{Solution for Arbitrary Sequential Order and Method of Undetermined Coefficients}

Consider the general problem~\eqref{linear_seq_1}. The associated polynomial~\eqref{assochi_poly} has distinct roots $\lambda_1,\ldots,\lambda_r\in\mathbb{C}$, with multiplicities $n_1,\ldots,n_r\geq 1$, $m=n_1+\ldots+n_r$. Let the complete set of solutions be $\mathcal{S}$. Initial conditions are denoted by~\eqref{sec_probl_ic}.

Theorem~\ref{th_increible} provides the intuition to establish the following general result. Later, we will give several remarks, examples, and an immediate corollary on the method of undetermined coefficients (i.e., the resolution of~\eqref{linear_seq_1} when it is perturbed by a certain source term).

\begin{Theorem} \label{th_super_increib}
For each eigenvalue $\lambda_l$ with multiplicity $n_l$, $l=1,\ldots,r$, consider
\begin{equation} \mathcal{B}_{\lambda_l,n_l}=\{ \mathcal{E}_\alpha(\lambda_l t), t \mathcal{E}_\alpha'(\lambda_l t),t^2 \mathcal{E}_\alpha''(\lambda_l t),\ldots,t^{n_l-1}\mathcal{E}_\alpha^{(n_l-1)}(\lambda_l t)\}, \label{basiss_subs} \end{equation}
where $\mathcal{E}_\alpha$ is the new Mittag--Leffler-type function~\eqref{mlf2} and $\mathcal{E}_\alpha^{(k)}$ denotes its ordinary $k$-th derivative (for $k\in\{1,2,3\}$, we use primes). Let
\begin{equation} \mathcal{B}=\bigcup_{l=1}^r \mathcal{B}_{\lambda_l,n_l}. \label{basiss_subs_22} \end{equation}
Then, $\mathcal{B}$ is a basis for $\mathcal{S}$.
\end{Theorem}
\begin{proof}
Fix $1\leq l\leq r$. Successive differentiation for~\eqref{mlf2} gives
\begin{equation} t^k \mathcal{E}_\alpha^{(k)}(\lambda_l t)=\sum_{n=k}^\infty n(n-1)\cdots (n-k+1)\frac{t^n \lambda_l^{n-k}}{\Gamma(2-\alpha)^n \prod_{j=1}^n \frac{\Gamma(j+1)}{\Gamma(j+1-\alpha)}}. \label{fiminik1} \end{equation}
Let us prove by induction on $0\leq q\leq k$ that
\begin{equation} ({}^L\! D^{\alpha} -\lambda_l)^q (t^k \mathcal{E}_\alpha^{(k)}(\lambda_l t))= t^{k-q}\left( \prod_{i=0}^{q-1} (k-i)\right)\sum_{n=k}^\infty \left( \prod_{i=q}^{k-1} (n-i)\right) \frac{t^{n-k}\lambda_l^{n-k}}{\Gamma(2-\alpha)^{n-q}\prod_{j=1}^{n-q} \frac{\Gamma(j+1)}{\Gamma(j+1-\alpha)}}. \label{fiminik2} \end{equation}

For $q=0$,~\eqref{fiminik2} corresponds to~\eqref{fiminik1}. Suppose the expression is true for $q-1$ (induction hypothesis), and we prove it for $q$. With detailed steps, we have:
\begin{align}
{} & ({}^L\! D^{\alpha} -\lambda_l)^q (t^k \mathcal{E}_\alpha^{(k)}(\lambda_l t))= ({}^L\! D^{\alpha} -\lambda_l)\circ ({}^L\! D^{\alpha} -\lambda_l)^{q-1}(t^k \mathcal{E}_\alpha^{(k)}(\lambda_l t)) \nonumber \\
= {} & ({}^L\! D^{\alpha} -\lambda_l)\left( t^{k-q+1}\left( \prod_{i=0}^{q-2} (k-i)\right)\sum_{n=k}^\infty \left( \prod_{i=q-1}^{k-1} (n-i)\right) \frac{t^{n-k}\lambda_l^{n-k}}{\Gamma(2-\alpha)^{n-q+1}\prod_{j=1}^{n-q+1} \frac{\Gamma(j+1)}{\Gamma(j+1-\alpha)}}\right) \label{justifc1} \\
= {} & \left( \prod_{i=0}^{q-2} (k-i)\right)\sum_{n=k}^\infty \left( \prod_{i=q-1}^{k-1} (n-i)\right) \frac{\lambda_l^{n-k}}{\Gamma(2-\alpha)^{n-q+1}\prod_{j=1}^{n-q+1} \frac{\Gamma(j+1)}{\Gamma(j+1-\alpha)}}({}^L\! D^{\alpha} -\lambda_l)(t^{n-q+1}) \label{justifc2} \\
= {} & \left( \prod_{i=0}^{q-2} (k-i)\right)\sum_{n=k}^\infty \left( \prod_{i=q-1}^{k-1} (n-i)\right) \frac{\lambda_l^{n-k}}{\Gamma(2-\alpha)^{n-q+1}\prod_{j=1}^{n-q+1} \frac{\Gamma(j+1)}{\Gamma(j+1-\alpha)}} \nonumber \\
{} & \quad \times \left[ \frac{\Gamma(n-q+2)\Gamma(2-\alpha)}{\Gamma(n-q+2-\alpha)} t^{n-q} - \lambda_l t^{n-q+1}\right] \label{justifc3} \\
= {} & \left( \prod_{i=0}^{q-2} (k-i)\right)\sum_{n=k}^\infty \left( \prod_{i=q-1}^{k-1} (n-i)\right) \frac{\lambda_l^{n-k}t^{n-q}}{\Gamma(2-\alpha)^{n-q}\prod_{j=1}^{n-q} \frac{\Gamma(j+1)}{\Gamma(j+1-\alpha)}} \label{justifc4} \\
{} & - \left( \prod_{i=0}^{q-2} (k-i)\right)\sum_{n=k}^\infty \left( \prod_{i=q-1}^{k-1} (n-i)\right) \frac{\lambda_l^{n-k+1}t^{n-q+1}}{\Gamma(2-\alpha)^{n-q+1}\prod_{j=1}^{n-q+1} \frac{\Gamma(j+1)}{\Gamma(j+1-\alpha)}} \label{justifc5} \\
= {} & \left( \prod_{i=0}^{q-2} (k-i)\right)\sum_{n=k}^\infty \left( \prod_{i=q-1}^{k-1} (n-i)\right) \frac{\lambda_l^{n-k}t^{n-q}}{\Gamma(2-\alpha)^{n-q}\prod_{j=1}^{n-q} \frac{\Gamma(j+1)}{\Gamma(j+1-\alpha)}} \label{justifc55} \\
{} & - \left( \prod_{i=0}^{q-2} (k-i)\right)\sum_{n=k+1}^\infty \left( \prod_{i=q-1}^{k-1} (n-1-i)\right) \frac{\lambda_l^{n-k}t^{n-q}}{\Gamma(2-\alpha)^{n-q}\prod_{j=1}^{n-q} \frac{\Gamma(j+1)}{\Gamma(j+1-\alpha)}} \label{justifc6} \\
= {} & \left( \prod_{i=0}^{q-2} (k-i)\right)\sum_{n=k+1}^\infty \left[ \left(\prod_{i=q-1}^{k-1} (n-i)\right)-\left(\prod_{i=q-1}^{k-1}(n-1-i)\right)\right] \frac{\lambda_l^{n-k}t^{n-q}}{\Gamma(2-\alpha)^{n-q}\prod_{j=1}^{n-q} \frac{\Gamma(j+1)}{\Gamma(j+1-\alpha)}} \label{justifc7} \\
 {} & + \left( \prod_{i=0}^{q-2} (k-i)\right) \left( \prod_{i=q-1}^{k-1} (k-i)\right) \frac{\lambda_l^{k-k}t^{k-q}}{\Gamma(2-\alpha)^{k-q}\prod_{j=1}^{k-q} \frac{\Gamma(j+1)}{\Gamma(j+1-\alpha)}} \label{justifc8} \\
= {} & \left( \prod_{i=0}^{q-1} (k-i)\right)\sum_{n=k+1}^\infty \left( \prod_{i=q}^{k-1} (n-i)\right) \frac{t^{n-q}\lambda_l^{n-k}}{\Gamma(2-\alpha)^{n-q}\prod_{j=1}^{n-q} \frac{\Gamma(j+1)}{\Gamma(j+1-\alpha)}} \label{justifc9} \\
{} & + \left( \prod_{i=0}^{q-1} (k-i)\right)(k-q)!\frac{\lambda_l^{k-k}t^{k-q}}{\Gamma(2-\alpha)^{k-q}\prod_{j=1}^{k-q} \frac{\Gamma(j+1)}{\Gamma(j+1-\alpha)}} \label{justifc10} \\
= {} & t^{k-q}\left( \prod_{i=0}^{q-1} (k-i)\right)\sum_{n=k}^\infty \left( \prod_{i=q}^{k-1} (n-i)\right) \frac{t^{n-k}\lambda_l^{n-k}}{\Gamma(2-\alpha)^{n-q}\prod_{j=1}^{n-q} \frac{\Gamma(j+1)}{\Gamma(j+1-\alpha)}}. \label{justifc11}
\end{align}

Equality~\eqref{justifc1} comes from the induction hypothesis. Equality~\eqref{justifc2} is merely the linearity of ${}^L\! D^{\alpha} -\lambda_l$. In~\eqref{justifc3}, the fractional derivative of the power is computed.  In \eqref{justifc4}--\eqref{justifc5}, we just expand the previous expression. For expression~\eqref{justifc6}, we just change the variable $n$ in the sum, while~\eqref{justifc55} is unchanged. In~\eqref{justifc7}, we join the two sums~\eqref{justifc55} and~\eqref{justifc6} from $n=k+1$, leaving the $k$-th term of~\eqref{justifc55} at~\eqref{justifc8}. For~\eqref{justifc9}, we apply the equality
\[ \left( \prod_{i=0}^{q-2} (k-i)\right)\left[\left(\prod_{i=q-1}^{k-1} (n-i)\right)-\left(\prod_{i=q-1}^{k-1}(n-1-i)\right)\right]=\left( \prod_{i=0}^{q-1} (k-i)\right)\left( \prod_{i=q}^{k-1} (n-i)\right). \]

Expression~\eqref{justifc10} comes from
\[ \left( \prod_{i=0}^{q-2} (k-i)\right) \left( \prod_{i=q-1}^{k-1} (k-i)\right) = \left( \prod_{i=0}^{q-1} (k-i)\right)(k-q)!. \]

Finally, for~\eqref{justifc11}, we merge terms to derive~\eqref{fiminik2}.

Considering~\eqref{fiminik2}, for $q=k$, we obtain
\begin{equation} ({}^L\! D^{\alpha} -\lambda_l)^k (t^k \mathcal{E}_\alpha^{(k)}(\lambda_l t))= k!\sum_{n=k}^\infty \frac{t^{n-k}\lambda_l^{n-k}}{\Gamma(2-\alpha)^{n-k}\prod_{j=1}^{n-k} \frac{\Gamma(j+1)}{\Gamma(j+1-\alpha)}}=k!\mathcal{E}_\alpha(\lambda_l t). \label{auri_crit} \end{equation}

Therefore,
\[
 ({}^L\! D^{\alpha} -\lambda_l)^{k+1} (t^k \mathcal{E}_\alpha^{(k)}(\lambda_l t))=k! ({}^L\! D^{\alpha} -\lambda_l)(\mathcal{E}_\alpha(\lambda_l t))=0. \]
 
Thus,
\begin{equation} ({}^L\! D^{\alpha} -\lambda_l)^{n_l}(t^k \mathcal{E}_\alpha^{(k)}(\lambda_l t))=0 \label{panxinini} \end{equation}
for all $k=0,\ldots,n_l-1$, so the operator $\Lambda$ from \eqref{xe_lambda}--\eqref{xe_lambda2} vanishes at $t^k \mathcal{E}_\alpha^{(k)}(\lambda_l t)$. This result proves that 
\[ \mathcal{B}_{\lambda_l,n_l}\subseteq\mathcal{S} \]
and, in general, 
\begin{equation} \mathcal{B}\subseteq\mathcal{S}. \label{fins_aciii} \end{equation}

Since $\mathcal{B}$ consists of $n_1+\ldots+n_r=m$ elements, and $\dim\mathcal{S}\leq m$ by  Proposition~\ref{vec_sp_SS} or~\eqref{seqke}, it suffices to prove that the functions in $\mathcal{B}$ are linearly independent. 

First, we prove that the functions in each $\mathcal{B}_{\lambda_l,n_l}$ are linearly independent. Consider a linear combination
\[ \beta_0\mathcal{E}_\alpha(\lambda_l t) + \beta_1 t \mathcal{E}_\alpha'(\lambda_l t) + \beta_2 t^2 \mathcal{E}_\alpha''(\lambda_l t) +\ldots+ \beta_{n_l-1}t^{n_l-1}\mathcal{E}_\alpha^{(n_l-1)}(\lambda_l t)=0, \]
for all $t\in\mathbb{R}$, where $\beta_0,\beta_1,\ldots,\beta_{n_l-1}\in\mathbb{C}$. Then,
\begin{align*}
 \beta_0 {}^L\! D^{q\circ \alpha} \mathcal{E}_\alpha(\lambda_l t) + {} & \beta_1 {}^L\! D^{q\circ \alpha} (t \mathcal{E}_\alpha'(\lambda_l t)) \\
 + {} & \beta_2 {}^L\! D^{q\circ \alpha}(t^2 \mathcal{E}_\alpha''(\lambda_l t)) +\ldots+ \beta_{n_l-1}{}^L\! D^{q\circ \alpha}(t^{n_l-1}\mathcal{E}_\alpha^{(n_l-1)}(\lambda_l t))=0, 
\end{align*}
for $1\leq q\leq n_l-1$. Now, by~\eqref{fiminik1},
\begin{align*}
 {}^L\! D^{q\circ \alpha} {} & (t^k \mathcal{E}_\alpha^{(k)}(\lambda_l t))= \sum_{n=k}^\infty n(n-1)\cdots (n-k+1)\frac{{}^L\! D^{q\circ \alpha}(t^n) \lambda_l^{n-k}}{\Gamma(2-\alpha)^n \prod_{j=1}^n \frac{\Gamma(j+1)}{\Gamma(j+1-\alpha)}} \\
= {} & \sum_{n=k}^\infty n(n-1)\cdots (n-k+1)\frac{t^{n-q}\lambda_l^{n-k}}{\Gamma(2-\alpha)^n \prod_{j=1}^n \frac{\Gamma(j+1)}{\Gamma(j+1-\alpha)}}\Gamma(2-\alpha)^q \prod_{i=n-q+2}^{n+1}\frac{ \Gamma(i)}{\Gamma(i-\alpha)},
\end{align*}
which vanishes at $t=0$ for $q+1\leq k\leq n_l-1$ and takes the value $1$ at $t=0$ for $k=q$. Consequently, the matrix
\begin{equation} \begin{pmatrix} \mathcal{E}_\alpha(\lambda_l t) & t \mathcal{E}_\alpha'(\lambda_l t) & \cdots & t^{n_l-1}\mathcal{E}_\alpha^{(n_l-1)}(\lambda_l t) \\
{}^L\! D^{\alpha} \mathcal{E}_\alpha(\lambda_l t) & {}^L\! D^{\alpha}(t \mathcal{E}_\alpha'(\lambda_l t)) & \cdots & {}^L\! D^{\alpha}(t^{n_l-1}\mathcal{E}_\alpha^{(n_l-1)}(\lambda_l t)) \\
{}^L\! D^{2\circ \alpha} \mathcal{E}_\alpha(\lambda_l t) & {}^L\! D^{2\circ \alpha}(t \mathcal{E}_\alpha'(\lambda_l t)) & \cdots & {}^L\! D^{2\circ \alpha}(t^{n_l-1}\mathcal{E}_\alpha^{(n_l-1)}(\lambda_l t)) \\
\vdots & \vdots & \ddots & \vdots \\
{}^L\! D^{(n_l-1)\circ \alpha} \mathcal{E}_\alpha(\lambda_l t) & {}^L\! D^{(n_l-1)\circ\alpha}(t \mathcal{E}_\alpha'(\lambda_l t)) & \cdots & {}^L\! D^{(n_l-1)\circ \alpha}(t^{n_l-1}\mathcal{E}_\alpha^{(n_l-1)}(\lambda_l t)) \end{pmatrix} \label{matriu_wronskian} \end{equation}
is upper-triangular at $t=0$, with non-zero elements at the diagonal. Its determinant is then non-zero, so necessarily $\beta_0=\beta_1=\ldots=\beta_{n_l-1}=0$.

To conclude the proof, consider a linear combination of elements in the complete set $\mathcal{B}$:
\begin{equation} \sum_{k=0}^{n_1-1} \beta_{k,1} t^k \mathcal{E}_\alpha^{(k)}(\lambda_1 t) + \ldots + \sum_{k=0}^{n_r-1} \beta_{k,r} t^k \mathcal{E}_\alpha^{(k)}(\lambda_r t)=0, \label{besonets} \end{equation}
where $\beta_{k,l}\in\mathbb{C}$. Suppose that there are coefficients $\beta_{k_{i},l_i}\neq0$ for $i=1,\ldots,I\leq r$, $I\geq2$, $1\leq l_i\leq r$ distinct, $0\leq k_{i}\leq n_{l_i}-1$, that is, at least one non-zero coefficient for each root $\lambda_{l_i}$. Then,~\eqref{besonets} can be rewritten as
\[ \sum_{i=1}^I \beta_{k_i,l_i} e_i=0, \]
where $e_i\in \langle \mathcal{B}_{\lambda_{l_i},n_{l_i}}\rangle$. We know that $e_i$ is a generalized eigenfunction of ${}^L\! D^{\alpha}$ associated with $\lambda_{l_i}$; see~\eqref{panxinini}. Since $\lambda_1,\ldots,\lambda_r$ are distinct, $\{e_1,\ldots,e_I\}$ are linearly independent by standard theory. Therefore, the assumed condition with $I\geq2$ is impossible. Then, $I=1$ and the linear combination~\eqref{besonets} is actually for a single group $\mathcal{B}_{\lambda_l,n_l}$, for some $l\in\{1,\ldots,r\}$. But the elements within $\mathcal{B}_{\lambda_l,n_l}$ are linearly independent, as already proved above. Hence, all of the coefficients of~\eqref{besonets} are zero, and we are finished.
\end{proof}

\begin{Remark} 
Analogously to the standard theory on linear ordinary differential equations, the determinant of the matrix~\eqref{matriu_wronskian} is the \textit{L-fractional wronskian} of the elements in $\mathcal{B}_{\lambda_l,n_l}$. In general, we define the L-fractional wronskian of $n$ real analytic functions $\phi_1,\ldots,\phi_n$ as
\[ {}^L\! W^\alpha(\phi_1,\ldots,\phi_n)(t)=\det\begin{pmatrix} \phi_1(t) & \phi_2(t) & \cdots & \phi_n(t) \\ {}^L\! D^{\alpha}\phi_1(t) & {}^L\! D^{\alpha}\phi_2(t) & \cdots & {}^L\! D^{\alpha}\phi_n(t) \\ \vdots & \vdots & \ddots & \vdots \\ {}^L\! D^{(n-1)\circ \alpha}\phi_1(t) & {}^L\! D^{(n-1)\circ \alpha}\phi_2(t) & \cdots & {}^L\! D^{(n-1)\circ \alpha}\phi_n(t) \end{pmatrix}. \]

If there is a point $t_1$ where 
\[ {}^L\! W^\alpha(\phi_1,\ldots,\phi_n)(t_1)\neq0, \]
then $\{\phi_1,\ldots,\phi_n\}$ are linearly independent. This fact was justified in the previous proof. Reciprocally, if $m$ functions in $\mathcal{S}$ are linearly independent, then their L-fractional wronskian is non-zero at $t=0$, because $\dim\mathcal{S}=m$, and the map $\Xi$ in \eqref{xe_xi}--\eqref{xe_xi2} is an isomorphism by Proposition~\ref{vec_sp_SS}. The wronskian appears when the coefficients of a linear combination in the basis $\mathcal{B}$ are found: if $\mathcal{B}=\{e_i\}_{i=1}^m$ and $x(t)=\sum_{i=1}^m c_i e_i(t)\in\mathcal{S}$ with initial conditions~\eqref{sec_probl_ic}, where $c_i\in\mathbb{C}$, then 
\[ {}^L\! W^\alpha(e_1,\ldots,e_m)(0)\cdot \begin{pmatrix} c_1 \\ c_2 \\ \vdots \\ c_m \end{pmatrix} = \begin{pmatrix} x_0 \\ x_{0,1} \\ \vdots \\ x_{0,m-1} \end{pmatrix}. \]

For example, the coefficients in~\eqref{masa_1} came from the linear system
\[ \begin{pmatrix} 1 & 1 \\ \lambda_1 & \lambda_2 \end{pmatrix} \begin{pmatrix} c_1 \\ c_2 \end{pmatrix} = \begin{pmatrix} x_0 \\ x_{0,1} \end{pmatrix}, \]
and in~\eqref{masa_2}, from
\[ \begin{pmatrix} 1 & 0 \\ \lambda & 1 \end{pmatrix} \begin{pmatrix} c_1 \\ c_2 \end{pmatrix} = \begin{pmatrix} x_0 \\ x_{0,1} \end{pmatrix}. \]
\end{Remark}

\begin{Example} \label{dutxi1} 
If $\lambda\in\mathbb{C}$, let us see that the solution of
\begin{equation} ({}^L\! D^{\alpha}-\lambda)x=t^l \mathcal{E}_\alpha^{(l)}(\lambda t), \label{lilifam2} \end{equation}
with $l\geq0$ and $x(0)=x_0\in\mathbb{C}$, is
\[ x(t)=\mathcal{E}_\alpha(\lambda t)x_0 + \frac{1}{l+1} t^{l+1} \mathcal{E}_\alpha^{(l+1)}(\lambda t)\in \langle \mathcal{E}_\alpha(\lambda t),t^{l+1} \mathcal{E}_\alpha^{(l+1)}(\lambda t)\rangle, \]
which makes sense with Theorem~\ref{th_super_increib}.

By~\eqref{fiminik1},
\[ \vartheta(t)=t^l \mathcal{E}_\alpha^{(l)}(\lambda t)=\sum_{n=l}^\infty n(n-1)\cdots (n-l+1)\frac{t^n \lambda^{n-l}}{\Gamma(2-\alpha)^n \prod_{j=1}^n \frac{\Gamma(j+1)}{\Gamma(j+1-\alpha)}}. \]

Considering~\eqref{rapidet}, if $\vartheta_n$ denotes the $n$-th term of this power series, we have
\begin{equation} \vartheta_{n-k-1}=(n-k-1)\cdots (n-k-l)\frac{\lambda^{n-k-1-l}}{\Gamma(2-\alpha)^{n-k-1} \prod_{j=1}^{n-k-1} \frac{\Gamma(j+1)}{\Gamma(j+1-\alpha)}} \label{peuitoos} \end{equation}
for $n-k-1\geq l$, and
\[ \vartheta_{n-k-1}=0 \]
for $n-k-1<l$. By Theorem~\ref{soccarr2}, the solution of~\eqref{lilifam2} is
\begin{align}
x(t)= {} & \mathcal{E}_\alpha(\lambda t)x_0 + \sum_{n=l+1}^\infty \sum_{k=0}^{n-1-l} \lambda^k \frac{\prod_{j=n-k}^n \Gamma(j-\alpha+1)}{\Gamma(2-\alpha)^{k+1}\prod_{j=n-k}^n\Gamma(j+1)}t^n \vartheta_{n-k-1} \nonumber \\
= {} & \mathcal{E}_\alpha(\lambda t)x_0 + \sum_{n=l+1}^\infty \sum_{k=0}^{n-1-l} (n-k-1)\cdots (n-k-l) \frac{\prod_{j=1}^n \Gamma(j-\alpha+1)}{\Gamma(2-\alpha)^n \prod_{j=1}^n \Gamma(j+1)}t^n \lambda^k \lambda^{n-k-1-l} \label{acabii1} \\
= {} & \mathcal{E}_\alpha(\lambda t)x_0 + \sum_{n=l+1}^\infty \lambda^{n-1-l} \frac{\prod_{j=1}^n \Gamma(j-\alpha+1)}{\Gamma(2-\alpha)^n \prod_{j=1}^n \Gamma(j+1)}t^n \sum_{k=0}^{n-1-l} (n-k-1)\cdots (n-k-l) \nonumber \\
= {} & \mathcal{E}_\alpha(\lambda t)x_0 + \frac{1}{l+1} t^{l+1}\sum_{n=l+1}^\infty n(n-1)\cdots (n-l) \lambda^{n-1-l} \frac{\prod_{j=1}^n \Gamma(j-\alpha+1)}{\Gamma(2-\alpha)^n \prod_{j=1}^n \Gamma(j+1)}t^{n-1-l} \label{acabii2} \\
= {} & \mathcal{E}_\alpha(\lambda t)x_0 + \frac{1}{l+1} t^{l+1} \mathcal{E}_\alpha^{(l+1)}(\lambda t). \nonumber
\end{align}
The equality in~\eqref{acabii1} comes from~\eqref{peuitoos}. The equality in~\eqref{acabii2} follows from the identity
\begin{equation} (l+1)\sum_{k=0}^{n-1-l} (n-k-1)\cdots (n-k-l)=n(n-1)\cdots (n-l), \label{jeiloo} \end{equation}
which is easy to prove by computing from $k=n-1-l$ to $k=0$, adding term by term and taking common factors.
\end{Example}

\begin{Example} \label{dutxi2} 
For $0\neq \lambda\in\mathbb{C}$, let us see that the solution of
\begin{equation} ({}^L\! D^{\alpha}-\lambda)x=t^l \mathcal{E}_\alpha^{(l)}(0)=\frac{l!}{\Gamma(2-\alpha)^l \prod_{j=1}^l \frac{\Gamma(j+1)}{\Gamma(j+1-\alpha)}} t^l 
, \label{lilisit}
\end{equation}
with $l\geq0$ and $x(0)=x_0\in\mathbb{C}$, is
\vspace{6pt}
\begin{align*}
 x(t)= {} & \mathcal{E}_\alpha(\lambda t)x_0 + \frac{l!}{\lambda^{l+1}}\left( \mathcal{E}_\alpha(\lambda t) - q_l(\lambda t) \right) \\
 {} & \in \langle \mathcal{E}_\alpha(\lambda t),\mathcal{E}_\alpha(0\cdot t),t \mathcal{E}_\alpha'(0\cdot t),\ldots,t^l \mathcal{E}_\alpha^{(l)}(0\cdot t)\rangle,
\end{align*}
where 
\begin{equation} q_l(\lambda t)=\sum_{n=0}^l \lambda^n t^n \frac{\prod_{j=1}^n \Gamma(j+1-\alpha)}{\Gamma(2-\alpha)^n \prod_{j=1}^n \Gamma(j+1)} \label{quieteq} \end{equation}
is a polynomial of degree $l$. This result agrees with Theorem~\ref{th_super_increib}.

Considering the forcing term
\[ \vartheta(t)=t^l \mathcal{E}_\alpha^{(l)}(0) \]
and~\eqref{rapidet}, we have
\begin{equation} \vartheta_{n-k-1}=\mathcal{E}_\alpha^{(n-k-1)}(0)=\frac{(n-k-1)!}{\Gamma(2-\alpha)^{n-k-1} \prod_{j=1}^{n-k-1} \frac{\Gamma(j+1)}{\Gamma(j+1-\alpha)}} \label{rmchisi} \end{equation}
for $n-k-1=l$, and
\[ \vartheta_{n-k-1}=0 \]
for $n-k-1\neq l$. By Theorem~\ref{soccarr2}, the solution of~\eqref{lilisit} is
\begin{align}
x(t)= {} & \mathcal{E}_\alpha(\lambda t)x_0 + \sum_{n=l+1}^\infty \sum_{k=0}^{n-1-l} \lambda^k \frac{\prod_{j=n-k}^n \Gamma(j-\alpha+1)}{\Gamma(2-\alpha)^{k+1}\prod_{j=n-k}^n\Gamma(j+1)}t^n \vartheta_{n-k-1} \nonumber \\
= {} & \mathcal{E}_\alpha(\lambda t)x_0 + l! \sum_{n=l+1}^\infty \frac{\prod_{j=1}^n \Gamma(j-\alpha+1)}{\Gamma(2-\alpha)^n \prod_{j=1}^n \Gamma(j+1)} t^n \lambda^{n-1-l} \label{fjdskfjld} \\
= {} & \mathcal{E}_\alpha(\lambda t)x_0 + \frac{l!}{\lambda^{l+1}} \left( \mathcal{E}_\alpha(\lambda t) - q_l(\lambda t) \right) \label{quis33}
\end{align}

The equality in~\eqref{fjdskfjld} is due to~\eqref{rmchisi}. In~\eqref{quis33}, the definition~\eqref{quieteq} is used. Notice that $q_l(\lambda t)$ is the $l$-th partial sum of $\mathcal{E}_\alpha(\lambda t)$.
\end{Example}

\begin{Corollary} \label{coro_espre}
In the context of this section, consider the non-homogeneous equation
\begin{equation} ({}^L\! D^\alpha-\lambda_1)^{n_1}\circ \cdots \circ ({}^L\! D^\alpha-\lambda_r)^{n_r}x(t)=\sum_{j=0}^J \beta_j t^j \mathcal{E}_\alpha^{(j)}(\mu t), \label{mai_acb} \end{equation}
where $\beta_j,\mu\in\mathbb{C}$ and $J\geq0$. 

If $\mu\neq \lambda_l$ for every $l=1,\ldots,r$, then
\begin{equation} x\in \langle \mathcal{B}\cup \{t^j \mathcal{E}_\alpha^{(j)}(\mu t):\,j=0,\ldots,J\}\rangle. \label{perfichiA} \end{equation}

On the contrary, if $\mu=\lambda_{l_0}$ for some $l_0\in\{1,\ldots,r\}$, then
\begin{equation} x\in \langle \left(\bigcup_{l\neq l_0}\mathcal{B}_{\lambda_l,n_l}\right)\cup \{t^k \mathcal{E}_\alpha^{(k)}(\lambda_{l_0} t):\,k=0,\ldots,J+n_{l_0}\}\rangle. \label{perfichiB} \end{equation}

Recall that $\mathcal{B}=\cup_{l=1}^r \mathcal{B}_{\lambda_l,n_l}$ is the basis of the homogeneous part of~\eqref{mai_acb}; see~\eqref{basiss_subs} and~\eqref{basiss_subs_22}.
\end{Corollary}
\begin{proof}
The uniqueness of solution for~\eqref{mai_acb} (given $m=n_1+\ldots+n_r$ initial conditions~\eqref{sec_probl_ic}) is known by Proposition~\ref{propo_canoi} and Proposition~\ref{prpo_unic}. 

Consider the linear map $\Lambda$ from~\eqref{xe_lambda} and~\eqref{xe_lambda2}, which represents the homogeneous part of~\eqref{mai_acb}.

If $\mu\neq \lambda_l$ for every $l=1,\ldots,r$, then~\eqref{mai_acb} is equivalent to
\begin{equation} ({}^L\! D^\alpha - \mu)^{J+1}\Lambda x=0, \label{perfichi} \end{equation}
for adequate initial conditions, by Theorem~\ref{th_super_increib}. By Theorem~\ref{th_super_increib} again, the solution to~\eqref{perfichi} satisfies~\eqref{perfichiA}.

On the other hand, if $\mu=\lambda_{l_0}$ for some $l_0\in\{1,\ldots,r\}$, then~\eqref{mai_acb} is equivalent to
\begin{equation} ({}^L\! D^\alpha - \lambda_{l_0})^{J+1}\Lambda x=0, \label{perfichi234} \end{equation}
for adequate initial conditions, by Theorem~\ref{th_super_increib}. Notice that $({}^L\! D^\alpha - \lambda_{l_0})^{J+1}\Lambda$ has the factor $({}^L\! D^\alpha - \lambda_{l_0})^{J+1+n_{l_0}}$. By Theorem~\ref{th_super_increib} again, the solution to~\eqref{perfichi234} satisfies~\eqref{perfichiB}.
\end{proof}

\begin{Example} 
We work with a specific numerical case of~\eqref{mai_acb}:
\begin{equation} {}^L\! D^{2\circ \alpha}x-2\cdot {}^L\! D^\alpha x+x=3\mathcal{E}_\alpha(2t), \label{jihinjk} \end{equation}
with initial states
\begin{equation} x(0)=3,\quad {}^L\! D^\alpha x(0)=-1. \label{xzdvdx} \end{equation}

According to Corollary~\ref{coro_espre},
\[ x(t)=\beta_1 \mathcal{E}_\alpha(t)+\beta_2 t \mathcal{E}_\alpha'(t)+\gamma \mathcal{E}_\alpha(2t), \]
where $\beta_1,\beta_2,\gamma\in\mathbb{R}$ are values to be determined. Since $\gamma \mathcal{E}_\alpha(2t)$ is a particular solution of~\eqref{jihinjk}, we have
\[ 4\gamma\mathcal{E}_\alpha(2t)-4\gamma\mathcal{E}_\alpha(2t)+\gamma\mathcal{E}_\alpha(2t)=3\mathcal{E}_\alpha(2t)\,\Rightarrow\, \gamma=3. \]

The other two coefficients are determined from the initial conditions~\eqref{xzdvdx}. First, since $\mathcal{E}_\alpha(0)=1$,
\[ 3=x(0)=\beta_1+\gamma\,\Rightarrow\,\beta_1=0. \]

Second, since ${}^L\! D^\alpha (t \mathcal{E}_\alpha'(t))|_{t=0}=1$ (see~\eqref{qpeowidjsn}),
\[ -1={}^L\! D^\alpha x(0)=\beta_1+\beta_2+2\gamma\,\Rightarrow\,\beta_2=-7. \]
\end{Example}

\begin{Example} 
Now, we deal with the numerical example~\eqref{mai_acb}
\begin{equation} {}^L\! D^{2\circ \alpha}x-2\cdot {}^L\! D^\alpha x+x=3\mathcal{E}_\alpha(t), \label{jihinjk22} \end{equation}
with initial states~\eqref{xzdvdx}. By Corollary~\ref{coro_espre},
\[ x(t)=\beta_1 \mathcal{E}_\alpha(t)+\beta_2 t \mathcal{E}_\alpha'(t)+\gamma t^2\mathcal{E}_\alpha''(t), \]
where $\beta_1,\beta_2,\gamma\in\mathbb{R}$ are values to be determined. We have the fact that $\gamma t^2\mathcal{E}_\alpha''(t)$ is a particular solution of~\eqref{jihinjk22}, which satisfies
\[ 3\mathcal{E}_\alpha(t)=({}^L\! D^\alpha-1)^2 (\gamma t^2\mathcal{E}_\alpha''(t))=2\gamma \mathcal{E}_\alpha(t), \]
by the previous computation~\eqref{auri_crit}. Thus,
\[ \gamma=\frac32. \]

By employing~\eqref{xzdvdx}, we determine $\beta_1$ and $\beta_2$:
\[ 3=x(0)=\beta_1, \]
\begin{align*}
 -1= {} & {}^L\! D^\alpha x(0) \\
= {} & \beta_1\cdot {}^L\! D^\alpha \mathcal{E}_\alpha(t)|_{t=0}+\beta_2 \cdot {}^L\! D^\alpha( t\mathcal{E}_\alpha'(t))|_{t=0}+ \gamma\cdot {}^L\! D^\alpha(t^2\mathcal{E}_\alpha''(t))|_{t=0} \\
= {} & \beta_1+\beta_2 \\
 \Rightarrow {} & \, \beta_2=-4.
\end{align*}

We used~\eqref{fiminik1} to compute ${}^L\! D^\alpha(t^2\mathcal{E}_\alpha''(t))|_{t=0}=0$ and~\eqref{qpeowidjsn} for ${}^L\! D^\alpha (t \mathcal{E}_\alpha'(t))|_{t=0}=1$.
\end{Example}

\begin{Example} 
Finally, in the complex field, consider~\eqref{mai_acb} given by
\begin{equation} {}^L\! D^{2\circ \alpha}x-2\mathrm{i}\cdot {}^L\! D^\alpha x-x=3\mathcal{E}_\alpha(t), \label{jihinjk33} \end{equation}
with initial conditions~\eqref{xzdvdx} and the imaginary unit $\mathrm{i}=\sqrt{-1}$. Corollary~\ref{coro_espre} states that
\[ x(t)=\beta_1 \mathcal{E}_\alpha(\mathrm{i}t)+\beta_2 t \mathcal{E}_\alpha'(\mathrm{i}t)+\gamma \mathcal{E}_\alpha(t), \]
where $\beta_1,\beta_2,\gamma\in\mathbb{C}$. Since $\gamma \mathcal{E}_\alpha(t)$ is a particular solution of~\eqref{jihinjk33}, we have
\[ \gamma\mathcal{E}_\alpha(t)-2\mathrm{i}\gamma\mathcal{E}_\alpha(t)-\gamma\mathcal{E}_\alpha(t)=3\mathcal{E}_\alpha(t)\,\Rightarrow\, \gamma=\frac{3}{2}\mathrm{i}. \]

For $\beta_1$ and $\beta_2$, we use~\eqref{xzdvdx}, as in the other two examples:
\[ 3=x(0)=\beta_1+\gamma\,\Rightarrow\,\beta_1=3-\frac{3}{2}\mathrm{i}, \]
\[ -1={}^L\! D^\alpha x(0)=\beta_1 \mathrm{i}+\beta_2+\gamma\,\Rightarrow\,\beta_2=-\frac{5}{2}-\frac{9}{2}\mathrm{i}. \]
\end{Example}

\begin{Remark} 
In the context of Examples~\ref{dutxi1} and~\ref{dutxi2}, let us try to solve
\begin{equation} ({}^L\! D^{\alpha}-\lambda_1)x=t^l \mathcal{E}_\alpha^{(l)}(\lambda_2 t)
 \label{lilisit_AAAA}
\end{equation}
in closed form in general, where $\lambda_1\neq\lambda_2$ and $\lambda_2\neq 0$ in $\mathbb{C}$ and $l\geq0$. We will see that the fact of changing of space, from $\langle \mathcal{B}_{\lambda_1,n_1}\rangle$ to $\langle \mathcal{B}_{\lambda_1,n_1}\cup \mathcal{B}_{\lambda_2,n_2}\rangle$, complicates computations. According to the previous results (see Theorem~\ref{th_super_increib} or Corollary~\ref{coro_espre}), the solution of~\eqref{lilisit_AAAA} takes the form
\begin{equation} x(t)=\mathcal{E}_\alpha(\lambda_1 t)c+\sum_{i=0}^l \beta_i t^i \mathcal{E}_\alpha^{(i)}(\lambda_2 t), \label{boolia} \end{equation}
where $c,\beta_i\in\mathbb{C}$. These parameters need to be determined. On the one hand,
\begin{equation} t^l \mathcal{E}_\alpha^{(l)}(\lambda_2 t)=\sum_{n=l}^\infty n(n-1)\cdots (n-l+1)\frac{\lambda_2^{n-l}}{\Gamma(2-\alpha)^n \prod_{j=1}^n \frac{\Gamma(j+1)}{\Gamma(j+1-\alpha)}}t^n, \label{guanjand} \end{equation}
see~\eqref{fiminik1}. On the other hand, some computations on~\eqref{boolia} with power series yield
\vspace{6pt}
\begin{equation}
 \begin{split}
 ({}^L\! D^{\alpha} {} & -\lambda_1)x= c\cdot {}^L\! D^{\alpha} (\mathcal{E}_\alpha(\lambda_1 t))+\sum_{i=0}^l \beta_i\cdot {}^L\! D^{\alpha} ( t^i \mathcal{E}_\alpha^{(i)}(\lambda_2 t)) \\
{} & - c \lambda_1 \mathcal{E}_\alpha(\lambda_1 t) - \lambda_1 \sum_{i=0}^l \beta_i t^i \mathcal{E}_\alpha^{(i)}(\lambda_2 t) \\
= {} & \beta_0 \sum_{n=0}^\infty \frac{\lambda_2^{n+1}}{\Gamma(2-\alpha)^n \prod_{j=1}^n \frac{\Gamma(j+1)}{\Gamma(j+1-\alpha)}}t^n \\
{} &+\sum_{i=1}^l \beta_i\sum_{n=i-1}^\infty (n+1)n\cdots (n-i+2)\frac{\lambda_2^{n+1-i}}{\Gamma(2-\alpha)^n \prod_{j=1}^n \frac{\Gamma(j+1)}{\Gamma(j+1-\alpha)}}t^n \\
{} & - \lambda_1\sum_{i=0}^l \beta_i \sum_{n=i}^\infty n(n-1)\cdots (n-i+1)\frac{\lambda_2^{n-i}}{\Gamma(2-\alpha)^n \prod_{j=1}^n \frac{\Gamma(j+1)}{\Gamma(j+1-\alpha)}}t^n \\
= {} & \beta_0 \sum_{n=0}^\infty \frac{\lambda_2^{n+1}}{\Gamma(2-\alpha)^n \prod_{j=1}^n \frac{\Gamma(j+1)}{\Gamma(j+1-\alpha)}}t^n \\
{} & +\sum_{n=0}^\infty \frac{1}{\Gamma(2-\alpha)^n \prod_{j=1}^n \frac{\Gamma(j+1)}{\Gamma(j+1-\alpha)}}\left[ \sum_{i=1}^{\min\{n+1,l\}}\beta_i(n+1)n\cdots (n-i+2)\lambda_2^{n+1-i}\right]t^n \\
{} & -\lambda_1 \sum_{n=0}^\infty \frac{1}{\Gamma(2-\alpha)^n \prod_{j=1}^n \frac{\Gamma(j+1)}{\Gamma(j+1-\alpha)}} \left[\sum_{i=0}^{\min\{n,l\}}\beta_i n(n-1)\cdots (n-i+1)\lambda_2^{n-i}\right]t^n.
\end{split}
 \label{guanjand2}
\end{equation}

After equating~\eqref{guanjand} and~\eqref{guanjand2}, we obtain the relations
\begin{equation}
\begin{split}
 \beta_0 \lambda_2^{n+1}+ {} & \sum_{i=1}^{n+1}\beta_i(n+1)n\cdots (n-i+2)\lambda_2^{n+1-i} \\
 {} & -\lambda_1 \sum_{i=0}^{n}\beta_i n(n-1)\cdots (n-i+1)\lambda_2^{n-i}=0 
\end{split}
 \label{descxop}
\end{equation}
for $0\leq n<l$, and
\begin{equation}
\begin{split}
 \beta_0 \lambda_2^{n+1}+ {} & \sum_{i=1}^{\min\{n+1,l\}}\beta_i(n+1)n\cdots (n-i+2)\lambda_2^{n+1-i} \\
{} & -\lambda_1 \sum_{i=0}^{\min\{n,l\}}\beta_i n(n-1)\cdots (n-i+1)\lambda_2^{n-i} \\
= {} & n(n-1)\cdots (n-l+1)\lambda_2^{n-l} 
 \end{split}
 \label{descxop2}
\end{equation}
for $n\geq l$. The linear equations~\eqref{descxop} can be rewritten, for $0\leq n<l$:
\begin{equation} \beta_{n+1}=\frac{1}{(n+1)!}\left[ -\beta_0 \lambda_2^{n+1} + \sum_{i=0}^{n}\beta_i n(n-1)\cdots (n-i+2)\lambda_2^{n-i}\left(\lambda_1(n-i+1)-\lambda_2(n+1)\right)\right]. \label{voiaor0}
\end{equation}
To determine $\beta_0$, because it cannot be a free parameter, the equation~\eqref{descxop2} is employed for $n=l$:
\begin{equation}
\begin{split}
 \beta_0 \lambda_2^{l+1}+ {} & \sum_{i=1}^{l}\beta_i(l+1)l\cdots (l-i+2)\lambda_2^{l+1-i} \\
{} & -\lambda_1 \sum_{i=0}^{l}\beta_i l(l-1)\cdots (l-i+1)\lambda_2^{l-i} = l! 
 \end{split}
 \label{voiaor2}
\end{equation}

For $c$, one simply uses the initial condition,
\[ x_0=x(0)=c+\beta_0. \]

The $l+1$ equations~\eqref{voiaor0} and~\eqref{voiaor2} cannot be decoupled, in general, for symbolic variables.

If Theorem~\ref{soccarr2} is employed to directly deal with~\eqref{lilisit_AAAA} based on power series, as in Examples~\ref{dutxi1} and~\ref{dutxi2}, we have the expression
\begin{align*}
 x(t)= {} & \mathcal{E}_\alpha(\lambda_1 t)x_0 \\
{} & + \sum_{n=l+1}^\infty \frac{\prod_{j=1}^n \Gamma(j+1-\alpha)}{\Gamma(2-\alpha)^n \prod_{j=1}^n \Gamma(j+1)}t^n \sum_{k=0}^{n-1-l} (n-k-1)\cdots(n-k-l)\lambda_1^k \lambda_2^{n-k-1-l}. 
\end{align*}

Compared to~\eqref{jeiloo}, the sum
\[ \sum_{k=0}^{n-1-l} (n-k-1)\cdots(n-k-l)\lambda_1^k \lambda_2^{n-k-1-l} \]
does not seem to be solvable in explicit form for $\lambda_1\neq\lambda_2$.
\end{Remark}

\begin{Remark} 
An alternative development to Theorem~\ref{th_super_increib} may be carried out, based upon the Jordan form of $\mathcal{A}$ in~\eqref{redu_lin_siss}, to compute the solution 
\begin{equation}
\hat{x}(t)=\mathcal{E}_\alpha(\mathcal{A}t)\hat{x}_0, 
 \label{sort_tindrie}
\end{equation}
where 
\[ \hat{x}_0=\hat{x}(0)=(x_0,x_{0,1},\ldots,x_{0,m-1})^\top\in\mathbb{C}^m. \]

First of all, we notice in this remark that care must be exercised, since some methods for the standard case $\alpha=1$ do not apply when $\alpha<1$. For example, the new Mittag--Leffler-type function~\eqref{mlf2} (the same occurs for the classical Mittag--Leffler function~\eqref{mlf}) does not satisfy the product property of the exponential
\begin{equation} \mathrm{e}^{\mathcal{A}_1+\mathcal{A}_2}=\mathrm{e}^{\mathcal{A}_1}\mathrm{e}^{\mathcal{A}_2}, \label{otro_pai} \end{equation}
\begin{equation} E_\alpha(\mathcal{A}_1+\mathcal{A}_2)\neq E_\alpha(\mathcal{A}_1) E_\alpha(\mathcal{A}_2), \label{otro_pai22} \end{equation}
\begin{equation} \mathcal{E}_\alpha(\mathcal{A}_1+\mathcal{A}_2)\neq\mathcal{E}_\alpha(\mathcal{A}_1)\mathcal{E}_\alpha(\mathcal{A}_2), \label{otro_pai33} \end{equation}
when the matrices $\mathcal{A}_1$ and $\mathcal{A}_2$ commute, in general. The property~\eqref{otro_pai}, which is based on the binomial expansion and canceling out the factorial, is key to compute $\mathrm{e}^{\mathcal{A}}$ when $\mathcal{A}$ is not diagonalizable. For example, a Jordan block 
\[ \mathcal{J}=\mu \mathrm{Id}+\mathcal{N}, \]
where $\mu\in\mathbb{C}$ is an eigenvalue and $\mathcal{N}$ is a nilpotent matrix, satisfies 
\[ \mathrm{e}^\mathcal{J}=\mathrm{e}^\mu \mathrm{e}^{\mathcal{N}}; \]
hence, $\mathrm{e}^\mathcal{J}$ is a finite sum. However, in general, 
\[ \mathcal{E}_\alpha(\mathcal{J})\neq \mathcal{E}_\alpha(\mu)\mathcal{E}_\alpha(\mathcal{N}). \]

Indeed, if $\mathcal{N}^{N_0}=0$, then
\vspace{6pt}
\begin{align*}
\mathcal{E}_\alpha(\mathcal{J})= {} & \sum_{n=0}^\infty \frac{\mathcal{J}^n}{\Gamma(2-\alpha)^n \prod_{j=1}^n \frac{\Gamma(j+1)}{\Gamma(j+1-\alpha)}} \\
= {} & \sum_{n=0}^\infty \frac{(\mu \mathrm{Id}+\mathcal{N})^n}{\Gamma(2-\alpha)^n \prod_{j=1}^n \frac{\Gamma(j+1)}{\Gamma(j+1-\alpha)}} \\
= {} & \sum_{n=0}^\infty \frac{1}{\Gamma(2-\alpha)^n \prod_{j=1}^n \frac{\Gamma(j+1)}{\Gamma(j+1-\alpha)}}\sum_{k=0}^n \binom{n}{k}\mu^{n-k}\mathcal{N}^k \\
= {} & \sum_{n=0}^{N_0} \frac{1}{\Gamma(2-\alpha)^n \prod_{j=1}^n \frac{\Gamma(j+1)}{\Gamma(j+1-\alpha)}}\sum_{k=0}^n \binom{n}{k}\mu^{n-k}\mathcal{N}^k \\
{} & + \sum_{n=N_0+1}^\infty \frac{1}{\Gamma(2-\alpha)^n \prod_{j=1}^n \frac{\Gamma(j+1)}{\Gamma(j+1-\alpha)}}\sum_{k=0}^{N_0} \binom{n}{k}\mu^{n-k}\mathcal{N}^k,
\end{align*}
which is an infinite series. Likewise, if $v$ is a generalized eigenvector of $\mathcal{A}$ associated with an eigenvalue $\mu\in\mathbb{C}$, then 
\[ \mathrm{e}^{\mathcal{A}}v=\mathrm{e}^{\mu\mathrm{Id}+(\mathcal{A}-\mu\mathrm{Id})}v=\mathrm{e}^{\mu}\mathrm{e}^{(\mathcal{A}-\mu\mathrm{Id})}v \]
is a finite sum again. Nonetheless, in general,
\[ \mathcal{E}_\alpha(\mathcal{A})v=\mathcal{E}_{\alpha}(\mu\mathrm{Id}+(\mathcal{A}-\mu\mathrm{Id}))v\neq\mathcal{E}_{\alpha}(\mu)\mathcal{E}_\alpha(\mathcal{A}-\mu\mathrm{Id})v. \]

When $\alpha=1$, Liouville's formula states that
\[ \det\mathrm{e}^{\mathcal{A}}=\mathrm{e}^{\mathrm{trace}\mathcal{A}}. \]

For $\alpha<1$, this is not the case in general, not even for diagonalizable matrices $\mathcal{A}$, on account of~\eqref{otro_pai22} and~\eqref{otro_pai33}. 

Another procedure can be followed to avoid the problematic fact that~\eqref{otro_pai33}. 

When the eigenvalues $\lambda_1,\ldots,\lambda_r$ of $\mathcal{A}$ are distinct, then $\mathcal{A}$ is diagonalizable. Let $\mathcal{A}=\mathcal{P}\mathcal{D}\mathcal{P}^{-1}$, where $\mathcal{D}$ is the diagonal matrix of eigenvalues and $\mathcal{P}$ is the invertible matrix of eigenvectors, of size $m\times m$. If $\hat{y}=\mathcal{P}^{-1}\hat{x}$, then ${}^L\! D^\alpha \hat{y}=\mathcal{D}\hat{y}$ and $\hat{y}_0=\hat{y}(0)=\mathcal{P}^{-1}\hat{x}_0$. Therefore,
\[ \hat{y}=\mathcal{E}_\alpha(\mathcal{D}t)\hat{y}_0=\begin{pmatrix} \mathcal{E}_\alpha(\lambda_1 t) & & \\ & \ddots & \\ & & \mathcal{E}_\alpha(\lambda_r t) \end{pmatrix}\hat{y}_0. \]

This implies that
\[ \hat{x}\in\langle \mathcal{E}_\alpha(\lambda_1 t),\ldots,\mathcal{E}_\alpha(\lambda_r t) \rangle. \]

From
\[ ({}^L\! D^\alpha - \lambda_i )(\mathcal{E}_\alpha(\lambda_i t))=0, \]
it is clear that $\mathcal{E}_\alpha(\lambda_i t)$ solves the sequential problem~\eqref{desc_D_lini}. Since the cardinality of
\begin{equation} \{\mathcal{E}_\alpha(\lambda_1 t),\ldots,\mathcal{E}_\alpha(\lambda_r t)\}\subseteq\mathcal{S} \label{medic_vulll} \end{equation} 
is $m$ and $\dim\mathcal{S}=m$---see Proposition~\ref{vec_sp_SS}---we obtain that~\eqref{medic_vulll} is a basis for $\mathcal{S}$. 

When there are repeated eigenvalues, the matrix $\mathcal{A}$ in~\eqref{redu_lin_siss} cannot be diagonalizable because the minimal polynomial coincides with the characteristic polynomial here. Then, $\mathcal{A}$ is expressed with Jordan blocks $\mathcal{J}_1,\ldots,\mathcal{J}_r$ of size $n_1\times n_1,\ldots,n_r\times n_r$, respectively. Let $\mathcal{A}=\mathcal{P}\mathcal{J}\mathcal{P}^{-1}$, where $\mathcal{J}$ is the Jordan form and $\mathcal{P}$ is the invertible matrix of generalized eigenvectors. If $\hat{y}=\mathcal{P}^{-1}\hat{x}$, then ${}^L\! D^\alpha \hat{y}=\mathcal{J}\hat{y}$ and $\hat{y}_0=\hat{y}(0)=\mathcal{P}^{-1}\hat{x}_0$. Therefore,
\vspace{6pt}
\[ \hat{y}=\mathcal{E}_\alpha(\mathcal{J}t)\hat{y}_0=\begin{pmatrix} \mathcal{E}_\alpha(\mathcal{J}_1 t) & & \\ & \ddots & \\ & & \mathcal{E}_\alpha(\mathcal{J}_r t) \end{pmatrix}\hat{y}_0. \]

For each $\mathcal{E}_\alpha(\mathcal{J}_i t)$, where $\mathcal{J}_i=\lambda_i\mathrm{Id}+\mathcal{N}_i$, we use a matrix Taylor development:
\begin{align*}
\mathcal{E}_\alpha(\mathcal{J}_i t)= {} & \mathcal{E}_\alpha(\lambda_i t\mathrm{Id} + \mathcal{N}_i t) \\
= {} & \sum_{n=0}^{\infty}\frac{1}{n!}\mathcal{E}_\alpha^{(n)}(\lambda_i t)t^n \mathcal{N}_i^n \\
= {} & \sum_{n=0}^{n_i-1}\frac{1}{n!}\mathcal{E}_\alpha^{(n)}(\lambda_i t)t^n \mathcal{N}_i^n.
\end{align*}

Hence
\[ \hat{x}\in\langle \{t^n \mathcal{E}_\alpha^{(n)}(\lambda_i t):\,n=0,\ldots,n_i-1,\,i=1,\ldots,r\} \rangle. \]

Nonetheless, to prove that
\begin{equation} \{t^n \mathcal{E}_\alpha^{(n)}(\lambda_i t):\,n=0,\ldots,n_i-1,\,i=1,\ldots,r\}\subseteq \mathcal{S}, \label{reigiis} \end{equation}
one needs to establish
\[ ({}^L\! D^\alpha - \lambda_i )^{n_i}(t^n \mathcal{E}_\alpha^{(n)}(\lambda_i t))=0 \]
for $n=0,\ldots,n_i-1$. Then, one should proceed as in the proof of Theorem~\ref{th_super_increib}, from~\eqref{fiminik1} until~\eqref{panxinini} and~\eqref{fins_aciii}. For $\alpha=1$, it is more or less straightforward that~\eqref{reigiis} holds, but further computations are needed for the fractional case. Once~\eqref{reigiis} is known, the fact that $\dim\mathcal{S}=m$ from Proposition~\ref{vec_sp_SS} entails that~\eqref{reigiis} is a basis for $\mathcal{S}$.

I decided to conduct the methodology based on scalar power series because of the following facts.
\begin{itemize}
\item It uses the decomposition~\eqref{desc_D_lini} and scalar first-order linear problems iteratively, which enlightens the underlying structure of the problem. This is specially true for the order $m=2$.
\item Essentially, one needs to prove~\eqref{panxinini} and~\eqref{fins_aciii} in any case, to ensure that the functions belong to $\mathcal{S}$. That is the difficult part.
\item With~\eqref{desc_D_lini}, only the upper bound $\dim\mathcal{S}\leq m$ is needed, which can be established from uniqueness or from the first isomorphism theorem; see~\eqref{seqke}. For~\eqref{seqke}, previous existence-and-uniqueness theory or results for linear differential systems are not  a prerequisite.
\item Although well known, the equality between the minimal polynomial and the characteristic polynomial of $\mathcal{A}$ is a key step to distinguish between multiplicities equal to one and repeated eigenvalues. With our methodology, no Jordan forms, generalized eigenvectors or minimal polynomials are needed. Matrix Taylor series are not  required either.
\item Our theory, based on~\eqref{desc_D_lini}, immediately gives the method of undetermined coefficients as a consequence; see Corollary~\ref{coro_espre}. Non-homogeneous equations with certain forcing terms---see Examples~\ref{dutxi1} and~\ref{dutxi2}---can be addressed.
\item Power series have gained importance in the study of fractional models in recent years; see the Introduction section. We show their use for arbitrary sequential problems.
\end{itemize}
\end{Remark}

\section{Sequential Linear Equations with Analytic Coefficients and Order Two: Context and Solution} \label{sec_sequ_AA}

The aim of this section is the study of non-autonomous linear L-fractional equations of sequential type {to extend the analysis conducted in the earlier section.} We focus on the case of order two, with analytic functions. First, we provide the context of the problem, and then we solve it.

\subsection{Context} 

In the previous section, we solved the autonomous linear equation~\eqref{linear_seq_1}, with the operator's decomposition~\eqref{desc_D_lini}. When the coefficients are not constant, such a procedure cannot be carried out.

In this part, we address the following non-autonomous linear equation in dimension $d=1$:
\begin{equation}
{}^L\! D^{2\circ \alpha}x(t)+p(t)\cdot {}^L\! D^\alpha x(t)+q(t)x(t)=c(t),
 \label{decl_juri}
\end{equation}
where
\begin{equation} p(t)=\sum_{n=0}^\infty p_n t^n,\;\; q(t)=\sum_{n=0}^\infty q_n t^n,\;\; c(t)=\sum_{n=0}^\infty c_n t^n \label{hypopopo} \end{equation}
are any power series that are convergent on an interval $[0,T]$, with terms $p_n,q_n,c_n\in\mathbb{C}$. Again, ${}^L\! D^{2\circ \alpha}$ is understood sequentially, as ${}^L\! D^\alpha \circ {}^L\! D^\alpha$. Like in the classical model with ordinary derivative, we seek a power-series solution for~\eqref{decl_juri}. Compared to Theorem~\ref{th_super_increib}, the coefficients of this power series will not be given in the closed form (see~\eqref{vaig_vaigs}).

By Proposition~\ref{propo_canoi}, the equation~\eqref{decl_juri} can be written as a first-order equation of dimension $2$. If $\mathcal{S}$ is the vector space of solutions of the homogeneous part of~\eqref{decl_juri}, then $\dim\mathcal{S}\leq 2$, by the uniqueness Proposition~\ref{prpo_unic}. Indeed, the linear map 
\[ \tilde{\Xi}:\,\mathcal{S}\rightarrow\mathbb{C}^2, \]
\[ \tilde{\Xi} x=(x(0),{}^L\! D^{\alpha} x(0)) \]
is injective. Since we have not tackled non-autonomous equations of the type ${}^L\! D^\alpha z(t)=\tilde{\mathcal{A}}(t)z(t)$, where $\tilde{\mathcal{A}}$ is a continuous matrix function, we cannot ensure the surjectivity of $\tilde{\Xi}$ for the moment. In what follows, we will establish two linearly independent power series that form a basis for $\mathcal{S}$.

For~\eqref{decl_juri}, initial data are defined by~\eqref{ic_dim_m2}.

\subsection{Results}

The main theorem of this section is the following. After discussing it, we will discuss two examples,   the L-fractional Airy's equation and the L-fractional Hermite's equation.

\begin{Theorem} \label{results_darre_part}
Given~\eqref{decl_juri} with coefficients~\eqref{hypopopo} on $[0,T]$, the general solution on $[0,T)$ is given by
\begin{equation} x(t)=\sum_{n=0}^\infty x_n t^n, \label{alca_cami} \end{equation}
where
\begin{equation} x_{n+2}=\frac{\Gamma(n+3-\alpha)\Gamma(n+2-\alpha)}{\Gamma(n+3)\Gamma(n+2)\Gamma(2-\alpha)^2}\left[-\sum_{l=0}^n p_{n-l}x_{l+1}\frac{\Gamma(l+2)\Gamma(2-\alpha)}{\Gamma(l+2-\alpha)}-\sum_{l=0}^n q_{n-l}x_l+c_n\right]. \label{vaig_vaigs} \end{equation}
The coefficients $x_0$ and $x_1$ correspond to $x(0)=x_0$ and ${}^L\! D^{\alpha} x(0)=x_{0,1}$, respectively. A basis of the homogeneous part ($c_n=0$) is obtained for $(x_0,x_{0,1})=(1,0)$ and $(x_0,x_{0,1})=(0,1)$, respectively.
\end{Theorem}
\begin{proof}
Given~\eqref{alca_cami}, the following L-fractional derivatives apply:
\[ {}^L\! D^{\alpha}x(t)=\sum_{n=0}^\infty x_{n+1}\frac{\Gamma(n+2)\Gamma(2-\alpha)}{\Gamma(n+2-\alpha)}t^n, \]
\[ {}^L\! D^{2\circ \alpha}x(t)=\sum_{n=0}^\infty x_{n+2}\frac{\Gamma(n+3)\Gamma(n+2)\Gamma(2-\alpha)^2}{\Gamma(n+3-\alpha)\Gamma(n+2-\alpha)}t^n, \]
see Corollary~\ref{cor_D_s}. Placing these derivatives into~\eqref{decl_juri}, with Cauchy products, and matching terms of the expansions, the recurrence relation~\eqref{vaig_vaigs} is obtained. It remains to check that the series~\eqref{alca_cami} actually converges on $[0,T)$.

Concerning~\eqref{hypopopo}, the coefficients are controlled as follows:
\[ |p_n|\leq \frac{C}{T^n},\quad |q_n|\leq \frac{C}{T^n},\quad |c_n|\leq \frac{C}{T^n}, \]
where $C>0$ is a constant. By the triangular inequality and induction, the sequence $\{x_n\}_{n=0}^\infty$ is ``majorized'' by
\begin{equation} H_{n+2}=\frac{\Gamma(n+3-\alpha)\Gamma(n+2-\alpha)}{\Gamma(n+3)\Gamma(n+2)\Gamma(2-\alpha)^2}\frac{C}{T^n}\left(\sum_{l=0}^n T^l\left[ H_{l+1} \frac{\Gamma(l+2)\Gamma(2-\alpha)}{\Gamma(l+2-\alpha)}+H_l\right]+1\right), \label{que_titulo_ti} \end{equation}
for $n\geq0$,
\[ H_0=|x_0|,\quad H_1=|x_1|. \]

By splitting the sum $\sum_{l=0}^n$ in~\eqref{que_titulo_ti} into $\sum_{l=0}^{n-1}$ and the $n$-th term, one derives
\begin{align*}
 H_{n+2}= {} & \left(\frac{\Gamma(n+3-\alpha)\Gamma(n+1)}{\Gamma(n+3)\Gamma(n+1-\alpha)T}+C\frac{\Gamma(n+3-\alpha)}{\Gamma(n+3)}\right)H_{n+1} \\
{} & + C\frac{\Gamma(n+3-\alpha)\Gamma(n+2-\alpha)}{\Gamma(n+3)\Gamma(n+2)\Gamma(2-\alpha)}H_n, 
\end{align*}
for $n\geq1$. Then, if we pick any $v\in (0,T)$,
\begin{align*}
 H_{n+2}v^{n+2}= {} & \left(\frac{\Gamma(n+3-\alpha)\Gamma(n+1)}{\Gamma(n+3)\Gamma(n+1-\alpha)T}v +C\frac{\Gamma(n+3-\alpha)}{\Gamma(n+3)}v\right)H_{n+1}v^{n+1} \\
{} & +C\frac{\Gamma(n+3-\alpha)\Gamma(n+2-\alpha)}{\Gamma(n+3)\Gamma(n+2)\Gamma(2-\alpha)}v^2 H_nv^n. 
\end{align*}

By letting
\[ K_n=\max_{0\leq l\leq n} H_l v^l, \]
one has the bound
\begin{align*}
 H_{n+2}v^{n+2}\leq {} & \left(\frac{\Gamma(n+3-\alpha)\Gamma(n+1)}{\Gamma(n+3)\Gamma(n+1-\alpha)T}v+C\frac{\Gamma(n+3-\alpha)}{\Gamma(n+3)}v \right. \\
{} & + \left. C\frac{\Gamma(n+3-\alpha)\Gamma(n+2-\alpha)}{\Gamma(n+3)\Gamma(n+2)\Gamma(2-\alpha)}v^2\right)K_{n+1}. 
\end{align*}

Since
\begin{align*}
 \lim_{n\rightarrow\infty} {} & \left(\frac{\Gamma(n+3-\alpha)\Gamma(n+1)}{\Gamma(n+3)\Gamma(n+1-\alpha)T}v+C\frac{\Gamma(n+3-\alpha)}{\Gamma(n+3)}v \right. \\
{} & + \left. C\frac{\Gamma(n+3-\alpha)\Gamma(n+2-\alpha)}{\Gamma(n+3)\Gamma(n+2)\Gamma(2-\alpha)}v^2\right)=\frac{v}{T}<1, 
\end{align*}
by~\eqref{saps_quina_e}, we deduce that $K_{n+2}=K_{n+1}=K$ from a certain $n\geq0$. As a consequence, if we take any $0\leq w<v<T$, then
\[ H_n w^n\leq H_n v^n \left(\frac{w}{v}\right)^n\leq L\left(\frac{w}{v}\right)^n. \]

Therefore,
\[ \sum_{n=0}^\infty H_n w^n<\infty. \]

This proves that~\eqref{alca_cami} converges on $[0,T)$, as wanted.

Concerning the basis of the homogeneous part (with $c_n=0$ for $n\geq0$), let $y$ and $z$ be series in $\mathcal{S}$ with initial terms $(x_0,x_{0,1})=(1,0)$ and $(x_0,x_{0,1})=(0,1)$, respectively. If $\beta_1 y+\beta_2z=0$ on $[0,T)$, for $\beta_1,\beta_2\in\mathbb{C}$, then
\[ 0=\beta_1 y(0)+\beta_2 z(0)=\beta_1 \]
and 
\[ 0=\beta_1 \cdot {}^L\! D^\alpha y(0)+\beta_2 \cdot {}^L\! D^\alpha z(0)=\beta_2, \]
so $\{y,z\}$ are linearly independent and form a basis of $\mathcal{S}$.
\end{proof}

\begin{Example} 
Let
\begin{equation} {}^L\! D^{2\circ \alpha}x(t) + at x(t)=0 \label{eiri} \end{equation}
be the L-fractional version of Airy's equation, where $a\in\mathbb{C}$. Here, $p=c=0$ and $q(t)=at$. By~\eqref{vaig_vaigs},
\[ x_2=0 \]
and
\[ x_{n+2}=-a\frac{\Gamma(n+3-\alpha)\Gamma(n+2-\alpha)}{\Gamma(n+3)\Gamma(n+2)\Gamma(2-\alpha)^2}x_{n-1}, \]
for $n\geq1$. This difference equation can be solved as follows:
\[ x_{3n-1}=0, \]
\[ x_{3n}=(-1)^n a^n \frac{\prod_{j=1}^n \Gamma(3j-\alpha)\Gamma(3j+1-\alpha)}{\Gamma(2-\alpha)^{2n} \prod_{j=1}^n \Gamma(3j)\Gamma(3j+1)}x_0, \]
\[ x_{3n+1}=(-1)^n a^n \frac{\prod_{j=1}^n \Gamma(3j+1-\alpha)\Gamma(3j+2-\alpha)}{\Gamma(2-\alpha)^{2n} \prod_{j=1}^n \Gamma(3j+1)\Gamma(3j+2)}x_{0,1}, \]
for $n\geq1$. Hence,
\[ y(t)=\sum_{n=0}^\infty (-1)^n a^n \frac{\prod_{j=1}^n \Gamma(3j-\alpha)\Gamma(3j+1-\alpha)}{\Gamma(2-\alpha)^{2n} \prod_{j=1}^n \Gamma(3j)\Gamma(3j+1)}t^{3n} \]
and
\[ z(t)=\sum_{n=0}^\infty (-1)^n a^n \frac{\prod_{j=1}^n \Gamma(3j+1-\alpha)\Gamma(3j+2-\alpha)}{\Gamma(2-\alpha)^{2n} \prod_{j=1}^n \Gamma(3j+1)\Gamma(3j+2)}t^{3n+1} \]
form the basis of solutions of~\eqref{eiri}, on $[0,\infty)$.
\end{Example}

\begin{Example} \label{exe_Hermi} 
Let
\begin{equation}
{}^L\! D^{2\circ \alpha}x(t)-2t\cdot {}^L\! D^{\alpha}x(t)+ax(t)=0 
 \label{hermit} 
\end{equation}
be the L-fractional Hermite's equation, where $a\in\mathbb{C}$. The input polynomials are $p(t)=-2t$, $q(t)=a$ and $c(t)=0$. According to~\eqref{vaig_vaigs},
\[ x_{n+2}=\frac{\Gamma(n+3-\alpha)\Gamma(n+2-\alpha)}{\Gamma(n+3)\Gamma(n+2)\Gamma(2-\alpha)^2}\left[2\frac{\Gamma(n+1)\Gamma(2-\alpha)}{\Gamma(n+1-\alpha)}-a\right]x_n, \]
for $n\geq0$. In closed form,
\[ x_{2n+1}=x_1 \frac{\prod_{i=3}^{2n+2} \Gamma(i-\alpha)}{\Gamma(2-\alpha)^{2n}\prod_{i=3}^{2n+2}\Gamma(i)}\prod_{i=1}^n\left( 2\frac{\Gamma(2i)\Gamma(2-\alpha)}{\Gamma(2i-\alpha)}-a\right) \]
and
\[ x_{2n}=x_0 \frac{\prod_{i=2}^{2n+1} \Gamma(i-\alpha)}{\Gamma(2-\alpha)^{2n}\prod_{i=2}^{2n+1}\Gamma(i)}\prod_{i=1}^n\left( 2\frac{\Gamma(2i-1)\Gamma(2-\alpha)}{\Gamma(2i-1-\alpha)}-a\right). \]

As a  consequence, the functions
\[ y(t)=\sum_{n=0}^\infty t^{2n+1}\frac{\prod_{i=3}^{2n+2} \Gamma(i-\alpha)}{\Gamma(2-\alpha)^{2n}\prod_{i=3}^{2n+2}\Gamma(i)}\prod_{i=1}^n\left( 2\frac{\Gamma(2i)\Gamma(2-\alpha)}{\Gamma(2i-\alpha)}-a\right) \]
and
\[ z(t)=\sum_{n=0}^\infty t^{2n}\frac{\prod_{i=2}^{2n+1} \Gamma(i-\alpha)}{\Gamma(2-\alpha)^{2n}\prod_{i=2}^{2n+1}\Gamma(i)}\prod_{i=1}^n\left( 2\frac{\Gamma(2i-1)\Gamma(2-\alpha)}{\Gamma(2i-1-\alpha)}-a\right) \]
form the basis of solutions of~\eqref{hermit}, on $[0,\infty)$. Notice that, if
\[ a=2\lambda, \]
where
\[ \lambda=\frac{\Gamma(i)\Gamma(2-\alpha)}{\Gamma(i-\alpha)},\quad i\geq 1,\quad i\in\mathbb{Z}, \]
then there exists a polynomial solution of~\eqref{hermit}:
\[ {}_N y(t)= \sum_{n=0}^N t^{2n+1}\frac{\prod_{i=3}^{2n+2} \Gamma(i-\alpha)}{\Gamma(2-\alpha)^{2n}\prod_{i=3}^{2n+2}\Gamma(i)}\prod_{i=1}^n\left( 2\frac{\Gamma(2i)\Gamma(2-\alpha)}{\Gamma(2i-\alpha)}-a\right), \]
\[ {}_N z(t)=\sum_{n=0}^N t^{2n}\frac{\prod_{i=2}^{2n+1} \Gamma(i-\alpha)}{\Gamma(2-\alpha)^{2n}\prod_{i=2}^{2n+1}\Gamma(i)}\prod_{i=1}^n\left( 2\frac{\Gamma(2i-1)\Gamma(2-\alpha)}{\Gamma(2i-1-\alpha)}-a\right), \]
for $N\geq0$. These polynomials extend, in a fractional sense, the classical Hermite's polynomials.
\end{Example}

\section{Open Problems} \label{sec_concl}

We broadly list some questions, which also highlight the limitations of the work:
\begin{itemize}
\item Would the L-fractional derivative have better performance than the Caputo fractional derivative in specific modeling problems? According to Section~\ref{sec_L} and Table~\ref{tab}, the L-fractional derivative and its associated differential equations have many appealing properties. For example, the solution is smooth, its ordinary derivative at the initial instant is finite, the vector field of the equation is a velocity with units of time$^{-1}$, and a differential can be associated with the fractional derivative. The appropriateness of the L-fractional derivative shall be checked with applied models, simulations, and fitting to real data, beyond purely theoretical work.
\item Is it possible to derive more formulas, improper/contour integral representations, applications, and numerical algorithms for the new Mittag--Leffler-type function~\eqref{mlf2}? Obviously, the classical Mittag--Leffler function~\eqref{mlf} is much more developed theoretically.
\item Can the ``almost everywhere'' condition in the fundamental theorem of L-fractional calculus (and in Caputo fractional calculus) be weakened? (See Lemma~\ref{lema_rigor_FC_Cap} and Proposition~\ref{propiidsf}.) We know that, for analytic functions and variations of them, the fundamental theorem of L-fractional calculus holds at every point $t$, not just almost everywhere (Corollary~\ref{cor_D_s} and Lemma~\ref{lema_agafa_clau}). Analogously, for fractional analytic functions, the fundamental theorem of Caputo fractional calculus is satisfied at every point $t$, not only almost everywhere (Remark~\ref{cor_D_s_remarK}), hence the potential of power-series expansions in fractional calculus, both for applications and theory. However, it would be of relevance to investigate whether there exists a larger class of functions for which there is equality at every $t$. We highlight the need to conduct rigorous computations in fractional calculus to make it clear what kind of solutions one   obtains (an everywhere solution, an almost-everywhere solution, a solution to the fixed-point integral problem, a solution to the modified Caputo equation, etc.; see Remark~\ref{rmk_villi_rel}, for example).
\item Is it possible to find closed-form expressions for the composed integral operator ${}^L\! J^{m\circ\alpha}$? A probabilistic structure was given to ${}^L\! J^{m\circ\alpha}$ depending on beta-distributed delays (Section~\ref{link_prrrp}), and expressions were obtained for source terms based on power functions (Section~\ref{subs_jardinir}). We wonder whether ${}^L\! J^{m\circ\alpha}$ could be given as a convolution, like in the Caputo case, and whether the solution $x(t)$ would depend on some new two-parameter Mittag--Leffler-type function.
\item Would the Laplace transform have any role when solving L-fractional differential equations? The power-series method is a powerful tool for L-fractional differential equations, by the analyticity of the solutions. However, the use of the Laplace transform has not been checked. The increase in the nonlinearity in the equation with $t^{1-\alpha}$ may complicate the applicability or the usefulness of the transform. Furthermore, the use should be precise, under appropriate hypotheses.
\item Can the probability link (Section~\ref{link_prrrp}) established in the paper help understand and generalize the concept of fractional derivative more? The L-fractional derivative and the associated integral operator distribute the past time with a beta distribution. Hence, the L derivative includes history's effects on the model, according to a fixed probability law. For the fractional order~$1$, the ordinary derivative is local, while the time of the Riemann integral is distributed uniformly. Given an interval, the uniform distribution maximizes the Shannon entropy, so the benefits of the fractional derivative in terms of memory terms shall be investigated.
\item Can the new Mittag--Leffler-type function~\eqref{mlf2} be used in other settings as a substitute for the exponential function, for example, to define novel probability distributions, such as a ``Poisson distribution'' with mass function related to the Mittag--Leffler-type function, or to study partial differential equations with exponentials involved, such as the heat equation? In the fractional case, the new Mittag--Leffler function would emerge.
\item Can we expect (Section~\ref{link_prrrp}) a better characterization of the finiteness of the fractional moment-generating function of random variables? One would probably need to apply the Cauchy--Hadamard theorem adequately, instead of the ratio test. Since the new Mittag--Leffler-type function is defined with products of gamma functions, the ratio test is the most straightforward tool to analyze the convergence of the series. On the other hand, the fractional moment-generating function may be of use to study some stochastic/random fractional differential equations.
\item Can the theory on $m$-th order autonomous linear equations be generalized to variable coefficients? Is it possible to find a variation-of-constants formula when forcing terms are present? This new research would continue the results from Section~\ref{sec_sequ}.
\item Can we build a theory about L-fractional dynamical systems? The corresponding fractional exponential, which is the proposed Mittag--Leffler-type function~\eqref{mlf2}, should play a key role, as it solves the linearized problem. The monotonicity and asymptotic properties of the new function shall be investigated. Relevant applications, such as the study of the L-fractional SIR (susceptible--infected--recovered) epidemiological model, would come up.
\item Is the theory on linear L-fractional differential equations with analytic coefficients extensible to the case of regular singular points? The problems are that changes in the variable and the product rule for the fractional derivative are not amenable to computing. This new research would continue the results from Section~\ref{sec_sequ_AA}.
\item What are the properties of the fractional Hermite's polynomial defined in Example~\ref{exe_Hermi}? Do they satisfy certain formulas or orthogonality conditions? A similar analysis would yield fractional Legendre's polynomials, fractional Laguerre's polynomials, and so on.
\item Does it make sense to rescale other fractional derivatives? For example, we commented that the $\Lambda$-fractional derivative normalizes the Riemann--Liouville derivative, and it shall be investigated mathematically. Would fractional operators $D^\alpha$ with continuous or bounded kernels improve their applicability if a factor $(D^\alpha t)^{-1}$ is included?
\item Can we explicitly solve other models, with nonlinearities, under the L-fractional derivative? With the experience of the Caputo derivative, the main tool shall be the power-series method, under analytic inputs. The solution will be local, as predicted by the Cauchy--Kovalevskaya theorem. It will be well defined and pointwise, according to Lemma~\ref{lema_rigor_FC_Cap}, Proposition~\ref{propiidsf}, Corollary~\ref{cor_D_s}, and Remark~\ref{cor_D_s_remarK}.
\item Finally, what about fractional partial differential equations? There are no studies for the L-fractional derivative. In the Caputo context, formal solutions have been found in terms of bivariate fractional power series, but rigorous theorems are yet to  be investigated.
\end{itemize}

\section*{Funding}
This research received no external funding.

\section*{Data Availability Statement}
No new data were created or analyzed in this study. Data sharing is not applicable to this article.

\section*{Conflict of interest}
The author declares that there are no conflicts of interest.


\end{document}